\documentclass[a4paper]{article}

\usepackage{amsmath,amssymb,amsthm,amsfonts}   
\numberwithin{equation}{section}   
\usepackage{physics}   
\usepackage[english]{babel}   
\usepackage[margin=3cm]{geometry}   
\usepackage{csquotes}   
\usepackage{mathtools}   
\usepackage{relsize}   
\usepackage{indentfirst}  
\usepackage{thm-restate}   
\usepackage{mathrsfs}   
\usepackage[dvipsnames]{xcolor}   
\usepackage{enumitem}    
\usepackage{caption}    
\usepackage{appendix}   
\usepackage{tabularx}   
\usepackage{xltabular}  
\usepackage{booktabs}   
\allowdisplaybreaks

\usepackage[dvipsnames]{xcolor}     
\usepackage{hyperref}	            
\newcommand\myshade{85}         
\colorlet{mylinkcolor}{violet}  
\colorlet{mycitecolor}{YellowOrange}    
\colorlet{myurlcolor}{Aquamarine}       
\hypersetup{            
    linkcolor  = blue!\myshade!black,      
    citecolor  = mycitecolor!\myshade!black,   
    urlcolor   = myurlcolor!\myshade!black, 
    colorlinks = true,      
    hypertexnames=false,
}               
\usepackage{cleveref}           
\usepackage[                                
style = alphabetic,                         
maxbibnames=9,maxcitenames=9,               
doi=true,isbn=true,giveninits=true,         
backend=biber,                              
url=false,arxiv=true]{biblatex}             
\theoremstyle{plain}   
 
\theoremstyle{plain}   
\newtheorem{assumption}{Assumption}     
\crefname{assumption}{assumption}{assumptions}

\newtheorem{lemma}{Lemma}[section]   
\crefname{lemma}{lemma}{lemmas}
\newtheorem{prop}[lemma]{Proposition}   
\crefname{prop}{proposition}{propositions}  
\newtheorem{cor}[lemma]{Corollary}   
\newtheorem{remark}[lemma]{Remark}   
\newtheorem{dfn}{Definition}    
\usepackage{fancyhdr}   
\usepackage{xparse} 
\ExplSyntaxOn       
\NewDocumentCommand{\definealphabet}{mmmm}{
  \int_step_inline:nnn { `#3 } { `#4 }{
    \cs_new_protected:cpx { #1 \char_generate:nn { ##1 }{ 11 } }{
      \exp_not:N #2 { \char_generate:nn { ##1 } { 11 } }}}}              
\ExplSyntaxOff      
\definealphabet{mb}{\mathbb}{A}{Z}      
\definealphabet{mc}{\mathcal}{A}{Z}     
\definealphabet{scr}{\mathscr}{A}{Z}    
\definealphabet{mf}{\mathfrak}{A}{Z}    

\DeclareMathOperator{\states}{\mbS_d}   
\DeclareMathOperator{\mapspace}{\mcL(\mbM_d)}   
\DeclareMathOperator{\dummyborel}{\mathscr{B}} 
\newcommand\borel[1]{\dummyborel\left(#1\right)}
\DeclareMathOperator{\pr}{\mathsf{pr}}           
\newcommand{\adj}{^\dagger}    
\DeclareMathOperator{\matrices}{\mbM_d} 
\DeclarePairedDelimiterX{\inner}[2]{\langle}{\rangle}{#1| #2} 
\DeclareMathOperator*{\esssup}{ess\,sup}    
\newcommand{\s}{\rho_{\mathsf{ss}}}
\newcommand{\dee}{\mathrm{d}}
\DeclareMathOperator{\Var}{Var}
\renewcommand{\tr}[1]{\operatorname{Tr}\left(#1\right)} 
\newcommand{\oln}{\overline}                        
\DeclareMathOperator{\bQ}{\oln{\mbQ}}               
\newcommand{\proj}{{\mathlarger{\boldsymbol{\cdot}}}} 
\let\oldker\ker         
\renewcommand{\ker}[1]{\oldker\left(#1\right)}  
\renewcommand{\tilde}{\widetilde}       

\title{Limit Theorems for Disordered Quantum Trajectories}
\usepackage{authblk}

\author[1]{Yeor Hafouta\thanks{yeor.hafouta@mail.huji.ac.il}}
\affil[1]{Ben-Gurion University of the Negev and the University of Florida}
\author[2]{Lubashan Pathirana\thanks{lpk@math.ku.dk}}
\affil[2]{Department of Mathematical Sciences and QMATH, University of Copenhagen, Denmark.}
\date{}

\begin{document}
\pagenumbering{arabic}
\lhead{\thepage}
\maketitle
\vspace{-1cm}


\begin{abstract}

    We study discrete-time quantum trajectories of a finite-dimensional system in a disordered environment, where the instrument applied at each step is determined by an invertible measure-preserving dynamical system. 
    For general quantum instruments on a standard Borel outcome space, we construct the quenched probability law on sequences of measurement outcomes for a fixed realization of the disorder and a given initial state.
    Under a quenched Doeblin condition with a uniform deterministic minorization constant and an environment-dependent minorizing probability measure, we establish the existence and uniqueness of an equivariant family of random posterior laws and identify its barycenter as the unique dynamically stationary state for the associated non-selective channel cocycle.
    If the environment is ergodic, we prove a pathwise ergodic theorem valid for every measurable choice of initial state: for almost every realization of the disorder, time averages of posterior states converge almost surely under the corresponding quenched law to the disorder average of the unique dynamically stationary state.
    For an ergodic environment, we establish quenched variance asymptotics and, when the asymptotic variance is positive, a central limit theorem and standardized Berry–Esseen bounds for additive functionals generated by bounded measurable real-valued functions of the environment and posterior state.
    These results hold for every measurable initial state and almost every disorder realization. 
    We then show that the stationary annealed joint process of instruments and posterior states inherits the \(\alpha\)-mixing of the instrument process, up to an exponentially decaying term.
    Finally, we provide classes of examples involving perfect and imperfect measurements, including models with continuous or finite outcome spaces, that satisfy the standing Doeblin condition.

\end{abstract}




\section{Introduction}
\label{section:intro}


Quantum trajectories describe the evolution of a quantum system conditioned on the outcomes of successive measurements.
These measurements are described by a quantum instrument \cite{Davies_Lewis_1970,Ozawa_1984}.
For a system on \(\mbC^d\), a quantum instrument is a countably additive measure \(A\mapsto\mfI(A)\) taking values in completely positive maps, such that the total map \(\Phi=\mfI(\Xi)\) is completely positive and trace-preserving (CPTP).
Here, \(\Xi\) is the outcome space.
Writing \(\mfI(\dee a)=\mcT_a\,\nu(\dee a)\) for a density representation of the instrument, the outcome of a measurement in state \(\rho\) has distribution \(\tr{\mcT_a(\rho)}\,\nu(\dee a)\).
Conditioned on the observed outcome \(a\), the posterior state is obtained through the update
\[
    \rho\longmapsto
    \frac{\mcT_a(\rho)}{\tr{\mcT_a(\rho)}},
\]
whenever the denominator is positive.
This update is defined for almost every outcome under the measurement distribution.
Such trajectories arise, for example, when a system interacts successively with probes that are measured after each interaction \cite{AP06,BJM14}.
The cavity-QED experiments of Haroche's group demonstrated how repeated measurements can be used to follow the collapse of a field state \cite{Gue+07} and to prepare and stabilize photon-number states through feedback \cite{Sayrin_2011}; see also \cite{HR06}.
This formalism gives a classical probabilistic description of the conditional state evolution, including measurement backaction, that is, the change in the system's state induced by measurement \cite{Dav76,Gis84,Car93,Hol01,BGM04,BG09,WM09}.

Quantum trajectories are studied in both discrete and continuous time. 
In discrete time, the state is updated after each successive measurement. 
In continuous time, the conditional evolution is described by stochastic master equations driven by diffusive or jump noise, corresponding to continuous monitoring.
The two descriptions are related through scaling limits of repeated-interaction and repeated-measurement models \cite{AP06,Pel08,BBB13,NP09}.
In this paper, we study discrete-time trajectories generated by quantum instruments, allowing both discrete and continuous outcome spaces.

Another distinction in the study of quantum trajectories is between perfect and imperfect measurements.
A measurement is called \emph{perfect} when, for \(\nu\)-almost every outcome \(a\), the corresponding map admits a single-Kraus representation \(\mcT_a(\rho)=V_a\rho V_a^\dagger\).
In the \emph{imperfect-measurement} setting, we allow general completely positive maps,
\[
    \mcT_a(\rho)
    =\sum_{j=1}^{r(a)}V_{a,j}\rho V_{a,j}^\dagger.
\]
Here \(j\) labels unobserved alternatives contributing to the same recorded outcome \(a\).
This occurs, for example, with mixed probe preparations or detectors that do not fully distinguish the microscopic outcomes \cite{tristant_2025}.
Perfect measurements preserve the purity of a pure initial state, but need not cause a mixed state to purify \cite{MK06}. 
Our results cover both perfect and imperfect measurements.

The long-time behaviour of quantum trajectories has been studied through both the measurement record and the posterior-state process. 
K{\"u}mmerer and Maassen established ergodic results for quantum counting records \cite{kummerer2003ergodic} and for time averages of posterior states \cite{KM04}. 
For repeated perfect measurements, their work on purification identified the role of dark subspaces in preventing the approach to pure states \cite{MK06}.
For homogeneous perfect measurements, purification and ergodicity of the non-selective channel lead to uniqueness of the invariant measure and geometric convergence in Wasserstein distance, with averaging over a possible period \cite{BFPP19}.
Laws of large numbers and central limit theorems for additive functionals of the posterior-state process were subsequently developed in \cite{BFP23}, while spectral methods yielded Berry--Esseen bounds and restricted large-deviation principles \cite{BHP25}.
The measurement record gives rise to a related set of statistical questions. A central limit theorem for empirical outcome frequencies is obtained in \cite{Attal_2014}.
Fluctuations of the record have also been studied through entropy production and its large deviations \cite{BJPP18}, with explicit models exhibiting non-Gaussian fluctuations \cite{Benoist_2021}.
For perfect measurements, invariant measures have been classified through dark subspaces and the dynamics within them under an irreducibility assumption \cite{BPS24}.
A related question is the long-time selection of invariant sectors, studied for general quantum instruments in \cite{BGP25}.
For imperfect measurements, asymptotic stability and uniqueness of the invariant measure have been established under irreducibility and contractivity assumptions \cite{tristant_2025}.

Related questions arise for non-selective evolution, described by products of quantum channels.
Ergodicity and mixing for a fixed channel have been studied in \cite{Bur+13}, while random repeated-interaction models have been investigated in \cite{BJM08,NP12}.
More generally, limit theorems for products of random quantum operations in ergodic environments were established in \cite{PS23}.

The present work concerns discrete-time quantum trajectories in a disordered environment, with an emphasis on fluctuations of additive functionals of the posterior-state process.
Previous work on disordered trajectories addresses asymptotic purification \cite{traj}, ergodic properties \cite{qtlln}, and central limit theorems for the measurement record \cite{luba_clt}.


\subsection{Disordered quantum trajectories}
\label{subsection:disordered-quantum-trajectories}


We work throughout with a finite-dimensional quantum system on \(\mbC^d\), with $d\ge 2$. 
We write
\[
    \matrices := \mbM_d(\mbC),
    \qquad
    \states := \left\{\rho\in\matrices : \rho\ge 0,\ \tr{\rho}=1\right\}.
\]
Let \((\Omega,\mcF,\pr)\) be a probability space describing the \emph{disorder}, and let \(\theta:\Omega\to\Omega\) be an invertible measure-preserving transformation.
We refer to \((\Omega,\mcF,\pr,\theta)\) as the \emph{base dynamical system}.
A point \(\omega\in\Omega\) represents a realization of the disorder, and the instrument applied at step \(n\ge1\) is indexed by \(\theta^{n-1}\omega\).

Let \((\Xi,\mcX)\) be a standard Borel space, interpreted as the measurement-outcome space. 
We encode the disorder-dependent measurement by a family of instruments.
Fix an orthonormal basis \((e_i)_{i=1}^d\) of \(\mbC^d\), and let \(E_{ij} = \ket{e_i}\bra{e_j}\in\matrices\) be the associated matrix units. 
We use the Choi--Jamio\l kowski isomorphism
\[
    J:\mcL(\matrices)\to \matrices\otimes\matrices \cong \mbM_{d^2}(\mbC),
    \qquad
    J(\Psi):=\sum_{i,j=1}^d \Psi(E_{ij})\otimes E_{ij},
\]
so that \(\Psi\) is completely positive if and only if \(J(\Psi)\ge 0\) \cite{Jamio_kowski_1972,Choi_1975,Watrous_2018}.

Denote by \(\mathsf{CP}(\matrices)\) the set of completely positive linear maps on \(\matrices\). 
For each \(\omega\in\Omega\), let
\[
    \mfI_\omega:\mcX\to \mathsf{CP}(\matrices)
\]
be a one-step instrument on \((\Xi,\mcX)\) in the sense of \cite{Davies_Lewis_1970,Ozawa_1984}, meaning that \(A\mapsto \mfI_\omega(A)\) is countably additive, each \(\mfI_\omega(A)\) is completely positive and trace-nonincreasing, and the total map
\[
    \Phi_\omega:=\mfI_\omega(\Xi)
\]
is trace-preserving. Thus \(\Phi_\omega\) is the associated non-selective channel. 
We assume throughout that, for every \(A\in\mcX\), the map
\[
    \omega\longmapsto J\left(\mfI_\omega(A)\right)
\]
is \(\mcF\)-measurable.

\paragraph{Measurable density realization.}
    For the purposes of the present paper, the primitive object is the instrument family \(\omega\mapsto \mfI_\omega\). 
    The quenched law of the observed record is canonically determined by this family alone, whereas a pathwise posterior update requires a measurable density representation. 
    A key technical fact, proved later in \Cref{prop:instrument-density-representation}, is that after modifying the family on a \(\pr\)-null set if necessary, one may choose a measurable kernel
    \[
        \omega\longmapsto \nu_\omega\in\mcP(\Xi)
    \]
    and a jointly measurable map
    \[
        (\omega,x)\longmapsto \mcT_{x;\omega}\in\mathsf{CP}(\matrices)
    \]
    such that
    \[
        \mfI_\omega(A)(X)
        =
        \int_A \mcT_{x;\omega}(X)\,\nu_\omega(\dee x)
    \]
    for every \(\omega\in\Omega\), every \(A\in\mcX\), and every \(X\in\matrices\).
    Conversely, any such measurable pair \((\nu_\omega,\mcT_{x;\omega})\) with trace-preserving total map
    \[
        X\mapsto \int_\Xi \mcT_{x;\omega}(X)\,\nu_\omega(\dee x)
    \]
    defines a random Davies--Lewis instrument by
    \[
        \mfI_\omega(A)(X):=\int_A \mcT_{x;\omega}(X)\,\nu_\omega(\dee x).
    \]
    Thus, after imposing the natural measurability and normalization assumptions, the instrument-level and density-level formulations are equivalent.
    This realization is auxiliary rather than part of the primitive data, but once fixed it provides the pathwise selective updates used below.
    As explained in \Cref{rem:current-instrument-measurable-realization}, it may moreover be chosen measurably with respect to the current one-step instrument. Consequently, the kernels \(\Gamma_\omega\) and \(K_\omega\) defined below can be taken to depend only on the current one-step instrument, in the sense of \Cref{def:depends-only-current-instrument}.

\paragraph{Selective update and one-step kernels.}
    Fix once and for all a reference state \(\rho_\ast\in\states\). 
    For \(\rho\in\states\), \(x\in\Xi\), and \(\omega\in\Omega\), define the selective update
    \[
        \mcT_{x;\omega}\proj \rho
        :=
        \begin{cases}
            \dfrac{\mcT_{x;\omega}(\rho)}{\tr{\mcT_{x;\omega}(\rho)}},
            & \tr{\mcT_{x;\omega}(\rho)}>0,\\[1.2ex]
            \rho_\ast,
            & \tr{\mcT_{x;\omega}(\rho)}=0.
        \end{cases}
    \]
    The corresponding one-step outcome and posterior-state kernels are
    \begin{equation}
    \label{eq:outcomes-kernel}
        \Gamma_\omega(\rho,A)
        :=
        \tr{\mfI_\omega(A)(\rho)}
        =
        \int_A \tr{\mcT_{x;\omega}(\rho)}\,\nu_\omega(\dee x),
        \qquad
        \rho\in\states,\quad A\in\mcX,
    \end{equation}
    and
    \begin{equation}
    \label{eq:transition_kernel}
        K_\omega(\rho,B)
        :=
        \int_\Xi
        \mathbf 1_B\!\left(\mcT_{x;\omega}\proj\rho\right)\,
        \tr{\mcT_{x;\omega}(\rho)}\,
        \nu_\omega(\dee x),
        \qquad
        \rho\in\states,\quad B\in\borel{\states}.
    \end{equation}
    By \Cref{prop:projective-action-is-measurable,prop:Gamma-K-are-kernels}, the map
    \[
        (\omega,x,\rho)\longmapsto \mcT_{x;\omega}\proj\rho
    \]
    is measurable, \(\Gamma_\omega\) is a probability kernel from  \((\states,\borel{\states})\) to \((\Xi,\mcX)\), \(K_\omega\) is a probability kernel on \((\states,\borel{\states})\), and for every \(A\in\mcX\) and \(B\in\borel{\states}\) the maps
    \[
        (\omega,\rho)\longmapsto \Gamma_\omega(\rho,A),
        \qquad
        (\omega,\rho)\longmapsto K_\omega(\rho,B)
    \]
    are measurable.

\paragraph{Quenched law on outcomes and posterior states.}
    Let \(\vartheta:\Omega\to\states\) be a measurable random initial state. 
    For each fixed \(\omega\in\Omega\), the quenched law of the observed outcome record is the unique probability measure
    \[
        \mbQ_{\vartheta;\omega}\in\mcP(\Xi^{\mbN})
    \]
    whose finite-dimensional marginals are given by
    \[
        \mbQ_{\vartheta;\omega}
        \left(
            A_1\times\cdots\times A_n\times \Xi\times\Xi\times\cdots
        \right)
        :=
        \tr{
            \left(
                \mfI_{\theta^{n-1}\omega}(A_n)
                \circ\cdots\circ
                \mfI_\omega(A_1)
            \right)\left(\vartheta(\omega)\right)
        },
        \qquad
        A_1,\dots,A_n\in\mcX.
    \]
    Thus, the quenched law of the record is canonically determined by the instrument family alone; see \Cref{prop:instrument-cylinder-formula,prop:construction-of-mbQ}. 
    Moreover, \(\omega\mapsto \mbQ_{\vartheta;\omega}\) is a measurable probability kernel from \((\Omega,\mcF)\) to \((\Xi^{\mbN},\mcX^{\otimes\mbN})\); see \Cref{prop:mbQ-kernel-measurable}.

\medskip
\paragraph{Posterior-state chain.}
    Once a measurable density realization has been fixed, the same quenched law admits the usual recursive description of the posterior states. Writing \(\bar x=(x_1,x_2,\dots)\in\Xi^{\mbN}\), define
    \[
        \rho_0^{\omega,\vartheta}(\bar x):=\vartheta(\omega),
    \]
    and, for \(n\ge0\),
    \[
        \rho_{n+1}^{\omega,\vartheta}(\bar x)
        :=
        \mcT_{x_{n+1};\theta^n\omega}\proj
        \rho_n^{\omega,\vartheta}(\bar x).
    \]
    The resulting maps
    \[
        \rho_n^{\omega,\vartheta}:\Xi^{\mbN}\to\states
    \]
    are measurable and depend only on the first \(n\) coordinates; see \Cref{prop:recursive-posterior-measurable}. Under \(\mbQ_{\vartheta;\omega}\), the posterior-state process \(\left(\rho_n^{\omega,\vartheta}\right)_{n\ge0}\) is a time-inhomogeneous Markov chain on \(\states\) with successive transition kernels
    \[
        K_\omega,\ K_{\theta\omega},\ K_{\theta^2\omega},\ \dots;
    \]
    see \Cref{prop:posterior-markov-under-mbQ}. 

\paragraph{Annealed law.}
    The annealed law of the disorder and the observed record is
    \[
        \bQ_\vartheta(\dee\omega,\dee\bar x)
        :=
        \pr(\dee\omega)\,\mbQ_{\vartheta;\omega}(\dee\bar x),
    \]
    a probability measure on \(\Omega\times\Xi^{\mbN}\).

\paragraph{Finite-outcome special case.}
    The general framework contains the usual finite-outcome repeated-measurement models as a special case. 
    If \(\Xi=\mcA\) is a non-empty finite alphabet and \(\mcX=2^{\mcA}\), then the instrument is determined by its singleton branches
    \[
        \mcT_{a;\omega}:=\mfI_\omega(\{a\}),
        \qquad a\in\mcA,
    \]
    and
    \[
        \mfI_\omega(A)=\sum_{a\in A}\mcT_{a;\omega},
        \qquad
        \Phi_\omega=\sum_{a\in\mcA}\mcT_{a;\omega}.
    \]
    Thus, in the finite-outcome case, the general density representation reduces to the elementary atomic decomposition of the instrument. 
    The corresponding cylinder probabilities take the familiar form
    \[
        \mbQ_{\vartheta;\omega}
        \left(
            \{a_1\}\times\cdots\times\{a_n\}\times \mcA\times\mcA\times\cdots
        \right)
        =
        \tr{
            \left(
                \mcT_{a_n;\theta^{n-1}\omega}
                \circ\cdots\circ
                \mcT_{a_1;\omega}
            \right)\left(\vartheta(\omega)\right)
        }.
    \]
    The model is called \emph{perfect} when each visible branch is implemented by a single Kraus operator,
    \[
        \mcT_{a;\omega}(\rho)=V_{a;\omega}\rho V_{a;\omega}\adj,
    \]
    and \emph{imperfect} when the visible branches are allowed to be arbitrary, completely positive trace-nonincreasing maps; equivalently, by Kraus' theorem,
    \[
        \mcT_{a;\omega}(\rho)
        =
        \sum_{j=1}^{d^2}
        V_{a,j;\omega}\rho V_{a,j;\omega}\adj,
    \]
    after padding by zero Kraus operators if necessary.


\subsection{Assumptions and Main Results}
\label{subsection:assumptions_results}


We continue to work with the objects introduced in \Cref{subsection:disordered-quantum-trajectories}. 
In particular, we fix once and for all a measurable density realization
\[
    (\omega,x)\longmapsto \mcT_{x;\omega},
    \qquad
    \omega\longmapsto \nu_\omega,
\]
together with the associated one-step outcome kernels \(\Gamma_\omega\), posterior-state kernels \(K_\omega\), quenched laws \(\mbQ_{\vartheta;\omega}\), and posterior-state processes \((\rho_n^{\omega,\vartheta})_{n\ge0}\). 
The later sections provide the detailed well-definedness and measurability proofs for these constructions.

For \(n\ge1\), define the \(n\)-step non-selective channel and the \(n\)-step posterior-state kernel by
\[
    \Phi_\omega^{(n)}
    :=
    \Phi_{\theta^{n-1}\omega}\circ\cdots\circ\Phi_\omega,
    \qquad
    K_\omega^{(n)}
    :=
    K_{\theta^{n-1}\omega}\circ\cdots\circ K_\omega,
\]
with the conventions
\[
    \Phi_\omega^{(0)}:=\mbI_{\matrices},
    \qquad
    K_\omega^{(0)}(\rho,\,\cdot\,):=\delta_\rho.
\]

Above, for probability kernels $M$ from $(S,\mcS)$ to $(T,\mcT)$ and $N$ from $(T,\mcT)$ to $(U,\mcU)$, we use the convention
\[
    (N\circ M)(s,C)=(MN)(s,C)
    :=\int_T N(t,C)\,M(s,\dee t),
    \qquad s\in S,\quad C\in\mcU.
\]

\begin{assumption}
\label{assumption_1}
    Throughout, the environment \((\Omega,\mcF,\pr,\theta)\) is invertible,  measurable, and \(\pr\)-preserving.
\end{assumption}

Our basic hypothesis is a uniform Doeblin minorization for the quenched posterior-state dynamics.

\begin{assumption}
\label{assumption_doeblin}
    There exist an integer \(L\ge1\), a constant \(\varepsilon\in(0,1)\), and a measurable family
    \[
        \omega\longmapsto \lambda_\omega\in\mcP(\states)
    \]
    such that for \(\pr\)-almost every \(\omega\),
    \[
        K_\omega^{(L)}(\rho,B)
        \ge
        \varepsilon\,\lambda_{\theta^L\omega}(B),
        \qquad
        \rho\in\states,\quad B\in\borel{\states}.
    \]
\end{assumption}

Equivalently, for \(\pr\)-almost every \(\omega\), one may write
\[
    K_\omega^{(L)}(\rho,\,\cdot\,)
    =
    \varepsilon\,\lambda_{\theta^L\omega}(\,\cdot\,)
    +
    (1-\varepsilon)\,R_\omega(\rho,\,\cdot\,),
\]
for a measurable family \(R_\omega\) of probability kernels on \((\states,\borel{\states})\). \Cref{Section:examples} contains several concrete classes of examples satisfying \Cref{assumption_doeblin}.

\begin{remark}
    Under additional mixing assumptions on the instrument process, one can also consider non-uniform versions of \Cref{assumption_doeblin}, in which the block length and minorization constant are allowed to depend on the environment, as in \cite{HW25,YH26}.
\end{remark}

\begin{remark}
\label{rem:doeblin-scope}
    For perfect instruments, \Cref{assumption_doeblin} implies finite-time entry into the pure-state set, a stronger conclusion than asymptotic purification.
    Let \(\mcP=\{\rho\in\states:\operatorname{rank}\rho=1\}\).
    Starting from a pure initial state, the posterior remains pure.
    Indeed, for a pure state $\rho = \ket{\psi}\bra{\psi}$, the posterior state after observing an outcome $a$ is 
    \[
        \frac{\mcT_{a;\omega}(\rho)}
        {\tr{\mcT_{a;\omega}(\rho)}}
        =
        \frac{\ket{V_{a;\omega}\psi}\bra{V_{a;\omega}\psi}}
        {\norm{V_{a;\omega}\psi}_{\mbC^d}^2}
    \]
    which has rank one, and hence a pure state. Thus, a trajectory started in a pure state remains pure almost surely. 
    Now, applying the minorization to pure initial state $\rho\in\mcP$  therefore gives, for almost every $\omega$,
    \[
        0 = K_\omega^{(L)}(\rho,\mcP^c)
        \ge \varepsilon \lambda_{\theta^L\omega}(\mcP^c)
    \]
    Therefore, $\lambda_{\theta^L\omega}(\mcP) = 1$ for $\pr$-almost every $\omega$.
    Applying the minorization once again, we get that for any initial state $\vartheta\in\states$, that
    \[
        K_\omega^{(L)}(\vartheta,\mcP)
        \ge\varepsilon \lambda_{\theta^L\omega}(\mcP)=\varepsilon.
    \]
    For any measurable initial state \(\vartheta\), let \(\tau^{\omega,\vartheta} =\inf\{n\geq0:\rho_n^{\omega,\vartheta}\in\mcP\}\).
    On a common $\theta$-invariant set of full $\pr$-measure, the Markov property (\Cref{prop:posterior-markov-under-mbQ}) and preservation of purity yield
    \[
        \mbQ_{\vartheta;\omega}
        \{\tau^{\omega,\vartheta}>kL\}
        \leq(1-\varepsilon)^k,
        \qquad k\geq0.
    \]
    Since \(\{\tau^{\omega,\vartheta}=\infty\}=\bigcap_{k\ge1}\{\tau^{\omega,\vartheta}>kL\}\),
    letting \(k\to\infty\) gives
    \[
        \mbQ_{\vartheta;\omega}
        \{\tau^{\omega,\vartheta}<\infty\}=1
        \qquad\text{for \(\pr\)-almost every \(\omega\)}.
    \]

    Thus, for \(\pr\)-almost every environment \(\omega\), the trajectory reaches the pure-state set in finite time \(\mbQ_{\vartheta;\omega}\)-almost surely and remains there.
    For imperfect instruments, purification is not required: for example, a single visible replacement branch \(A\mapsto\operatorname{Tr}(A)\eta\), with \(\eta\) mixed, gives \(K(\rho,\,\cdot\,)=\delta_\eta\); see \Cref{subsec:examples-preparation}.
\end{remark}

A random quantum state is a measurable map \(\rho:\Omega\to\states\). 
We call such a map \(\s\) a \emph{dynamically stationary state} for the non-selective evolution if
\[
    \Phi_\omega\left(\s(\omega)\right)
    =
    \s(\theta\omega)
    \qquad
    \text{for \(\pr\)-almost every }\omega.
\]
In the homogeneous case, when the environment consists of a single point, \(\Phi_\omega\equiv\Phi\), this reduces to the usual invariance relation \(\Phi(\s)=\s\). 
In the present finite-dimensional setting, such states always exist via the vector-valued ergodic theorem of Beck--Schwartz \cite{Beck_1957}.

Our first theorem establishes the existence and uniqueness of an \emph{equivariant family of posterior laws}. 
It proves uniform geometric convergence to this family when the dynamics is started in the remote past. 

\begin{dfn}
    A measurable family \(\omega\longmapsto \mu_\omega\in\mcP(\states)\)  is called \emph{equivariant} for \(K\) if
    \[
        (K_\omega)^*\mu_\omega=\mu_{\theta\omega}, \qquad \text{for \(\pr\)-almost every \(\omega\).}
    \]
\end{dfn}

\begin{restatable}[]{thm}{rhosexists}
\label{thm:s_exists}
    Under \Cref{assumption_1,assumption_doeblin}, for the kernel $K$ (in \cref{eq:transition_kernel}) there exists a measurable equivariant family
    \[
        \omega\longmapsto \mu_\omega\in\mcP(\states),
    \]
    unique up to \(\pr\)-almost sure equality. 
    Its barycenter
    \[
        \s(\omega)
        :=
        \int_{\states}\eta\,\mu_\omega(\dee\eta)
    \]
    is the \(\pr\)-almost surely unique dynamically stationary state. 
    Moreover, for every measurable family \(\omega\longmapsto \zeta_\omega\in\mcP(\states)\)  and every \(n\ge0\),
    \[
        \esssup_{\omega}\,
        \norm{
            \zeta_{\theta^{-n}\omega}K_{\theta^{-n}\omega}^{(n)}
            -
            \mu_\omega
        }_{\mathrm{TV}}
        \le
        (1-\varepsilon)^{\lfloor n/L\rfloor}
        \operatorname*{ess\,sup}_{\omega}
        \norm{\zeta_\omega-\mu_\omega}_{\mathrm{TV}}.
    \]
    In particular,
    \[
        \esssup_{\omega}\,
        \sup_{\rho\in\states}
        \norm{
            \Phi_{\theta^{-n}\omega}^{(n)}(\rho)-\s(\omega)
        }_1
        \le
        2(1-\varepsilon)^{\lfloor n/L\rfloor}.
    \]
\end{restatable}

Our next result is the disordered counterpart of the pathwise ergodic theorem for quantum trajectories; see \cite{KM04} in the homogeneous case.

\begin{restatable}[Pathwise Ergodic Theorem]{thm}{pathbirkhoffergodic}
\label{thm:birkhoff_ergodic}
    Make \Cref{assumption_1,assumption_doeblin}, and assume moreover that \((\Omega,\mcF,\pr,\theta)\) is ergodic. 
    Let \(\s\) be the unique dynamically stationary state from \Cref{thm:s_exists}, and set
    \[
        \bar\s
        :=
        \int_\Omega \s(\eta)\,\pr(\dee\eta).
    \]
    Then for every measurable initial state \(\vartheta:\Omega\to\states\) there exists a full-measure set \(\Omega_0\subseteq\Omega\) such that for every \(\omega\in\Omega_0\),
    \[
        \lim_{N\to\infty}
        \norm{
            \frac1N\sum_{n=1}^N \rho_n^{\omega,\vartheta}(\bar x)-\bar\s
        }_1
        =0
    \]
    for \(\mbQ_{\vartheta;\omega}\)-almost every \(\bar x\in\Xi^{\mbN}\).
\end{restatable}

\paragraph{The posterior-state path law.}
    Consider the measurable state-path space
    \[
        \left(
            \states^{\mbN_0},
            \borel{\states}^{\otimes\mbN_0}
        \right),
    \]
    with coordinate maps
    \[
        e_n:\states^{\mbN_0}\longrightarrow\states,
        \qquad
        e_n((\rho_j)_{j\ge0}):=\rho_n,
        \qquad n\in\mbN_0.
    \]
    For a measurable initial state \(\vartheta:\Omega\to\states\), let \(\mbK_{\vartheta;\omega}\) denote the unique Markov law on this space with initial distribution \(\delta_{\vartheta(\omega)}\) and successive transition kernels
    \[
        K_\omega,\ K_{\theta\omega},\
        K_{\theta^2\omega},\dots.
    \]
    The existence of this law and its measurable dependence on \(\omega\) are established in \Cref{prop:construction-of-bbK}.
    The following corollary expresses \Cref{thm:birkhoff_ergodic} on this state-path space.

\begin{cor}[Pathwise ergodic theorem under \(\mbK_{\vartheta;\omega}\)]
\label{cor:pathwise-ergodic-under-mbK}
    Make \Cref{assumption_1,assumption_doeblin}, and assume moreover that
    \((\Omega,\mcF,\pr,\theta)\) is ergodic.
    Let \(\s\) be the unique dynamically stationary state from \Cref{thm:s_exists}, and set
    \[
        \bar\s
        :=
        \int_\Omega \s(\eta)\,\pr(\dee\eta).
    \]
    Then for every measurable initial state \(\vartheta:\Omega\to\states\) there exists a full-measure set
    \(\Omega_0\subseteq\Omega\) such that for every \(\omega\in\Omega_0\),
    \[
        \lim_{N\to\infty}
        \norm{
            \frac1N\sum_{n=1}^N e_n(\rho)-\bar\s
            }_1
        =0
    \]
    for \(\mbK_{\vartheta;\omega}\)-almost every
    \(\rho=(\rho_0,\rho_1,\rho_2,\dots)\in\states^{\mbN_0}\).
\end{cor}
\begin{proof}
    Fix a measurable initial state \(\vartheta\).
    Choose a common full-measure set \(\Omega_0\), allowed to depend on \(\vartheta\), on which \Cref{thm:birkhoff_ergodic,prop:bbK-is-pushforward-of-mbQ} both apply.
    For \(\omega\in\Omega_0\), define
    \[
        R_{\omega,\vartheta}:\Xi^{\mbN}
        \longrightarrow\states^{\mbN_0},
        \qquad
        R_{\omega,\vartheta}(\bar x)
        :=\left(\rho_n^{\omega,\vartheta}(\bar x)\right)_{n\ge0}.
    \]
    This map is measurable because each of its coordinates is measurable.
    By \Cref{prop:bbK-is-pushforward-of-mbQ},
    \[
        \mbK_{\vartheta;\omega}
        =\mbQ_{\vartheta;\omega}\circ R_{\omega,\vartheta}^{-1}.
    \]
    The event
    \[
        A:=
        \left\{
            \rho\in\states^{\mbN_0}:
            \lim_{N\to\infty}
            \norm{
                \frac1N\sum_{n=1}^N e_n(\rho)-\bar\s
                }_1=0
        \right\}
    \]
    is measurable, since its defining averages are measurable and convergence to zero is a countable combination of measurable conditions.
    Since \(e_n\circ R_{\omega,\vartheta}=\rho_n^{\omega,\vartheta}\), \Cref{thm:birkhoff_ergodic} gives
    \[
        \mbK_{\vartheta;\omega}(A)
        =\mbQ_{\vartheta;\omega}
          \left(R_{\omega,\vartheta}^{-1}(A)\right)
        =1
    \]
    for every \(\omega\in\Omega_0\).
\end{proof}

In addition to the pathwise ergodic theorem, we obtain quenched fluctuation results for additive functionals generated by bounded measurable functions of the environment and posterior state: the existence of a deterministic asymptotic variance, a quenched central limit theorem, a standardized Berry--Esseen bound with random variance normalization, and, under a stronger mixing assumption on the instrument process, a Berry--Esseen bound with deterministic normalization \(\Sigma\sqrt N\).

For a measurable initial state \(\vartheta:\Omega\to\states\) and a bounded measurable function \(f:\Omega\times\states\to\mbR\), write \(\rho_n=e_n\) for the coordinate process and define
\[
    S_N^\omega f
    :=
    \sum_{n=0}^{N-1} f\left(\theta^n\omega,\rho_n\right),
\]
and set
\[
    \bar S_N^\omega f
    :=
    S_N^\omega f-\mbE_{\mbK_{\vartheta;\omega}}\left[S_N^\omega f\right].
\]
Our next theorem gives the quenched variance asymptotics, the central limit theorem, and the standardized Berry--Esseen bound for these additive functionals.

\begin{restatable}[Quenched standardized Berry--Esseen theorem for additive functionals]{thm}{quenchedcltposterior}
\label{thm:quenched_clt_posterior}
    Let \(f:\Omega\times\states\to\mbR\) be a bounded measurable function.
    Make \Cref{assumption_1,assumption_doeblin} and assume that \(\theta\) is ergodic.
    Then there exists a constant \(\Sigma\ge0\), independent of \(\omega\) and of the initial state, such that for every measurable initial state \(\vartheta:\Omega\to\states\) the following holds for $\pr$-almost every $\omega$. 
    \[
        \lim_{N\to\infty}
        \frac1N
        \Var_{\mbK_{\vartheta;\omega}}\left(S_N^\omega f\right)
        =
        \Sigma^2.
    \]
    If moreover \(\Sigma>0\), then for \(\pr\)-almost every \(\omega\) and every
    \(t\in\mbR\),
    \[
        \lim_{N\to\infty}
        \mbK_{\vartheta;\omega}\left\{
            \frac{\bar S_N^\omega f}{\sqrt N}\le t\Sigma
        \right\}
        =
        \Phi(t)
        :=
        \frac1{\sqrt{2\pi}}
        \int_{-\infty}^t e^{-x^2/2}\,\dee x.
    \]
    Moreover, when \(\Sigma>0\),
    \[
        \sup_{t\in\mbR}
        \left|
            \mbK_{\vartheta;\omega}\left\{
                \frac{\bar S_N^\omega f}
                {\norm{\bar S_N^\omega f}_{L^2(\mbK_{\vartheta;\omega})}}
                \le t
            \right\}
            -
            \Phi(t)
        \right|
        =
        O(N^{-1/2})
    \]
    for \(\pr\)-almost every \(\omega\).
\end{restatable}

\paragraph{Instrument dependence and quenched Berry--Esseen bounds.}

To replace the random standard-deviation normalization in the Berry--Esseen bound of \Cref{thm:quenched_clt_posterior} by the deterministic asymptotic normalization \(\Sigma\sqrt N\), we impose an additional mixing assumption on the instrument process.

Let \(\mcG_k^{\mathrm{inst}}\) denote the \(\sigma\)-algebra generated by the current abstract instrument at time \(k\), equivalently by the maps
\[
    \omega\longmapsto J\!\left(\mfI_{\theta^k\omega}(A)\right),
    \qquad A\in\mcX;
\]
see \Cref{subsection:canonical-coding-instrument-process}. 
Define the past and future \(\sigma\)-algebras generated by the instrument process by
\[
    \mcF_{k}^{\mathrm{inst},-}
    :=
    \sigma\!\left(\mcG_j^{\mathrm{inst}}:j\le k\right),
    \qquad
    \mcF_{k}^{\mathrm{inst},+}
    :=
    \sigma\!\left(\mcG_j^{\mathrm{inst}}:j\ge k\right).
\]
We then define the corresponding \(\phi\)-mixing coefficients by
\[
    \phi^{\mathrm{inst}}(n)
    :=
    \sup_{k\in\mbZ}
    \phi_{\pr}\!\left(
        \mcF_{k}^{\mathrm{inst},-},
        \mcF_{k+n}^{\mathrm{inst},+}
    \right),
    \qquad n\ge1.
\]
We also say that \(f_\omega\) or \(K_\omega\) depends only on the current one-step instrument if it is measurable with respect to \(\mcG_0^{\mathrm{inst}}\), in the precise sense of \Cref{def:depends-only-current-instrument}.

\begin{assumption}
\label{assumption_mix}
    For the quantitative quenched fluctuation results, assume that
    \[
        \sum_{n=1}^{\infty}\phi^{\mathrm{inst}}(n)<\infty.
    \]
\end{assumption}

\begin{remark}
\label{rem:instrument-factor-ergodicity}
Let
\[
    \mcG^{\mathrm{inst}}
    :=\sigma\!\left(\mcG_j^{\mathrm{inst}}:j\in\mbZ\right)
\]
be the sigma-algebra generated by the entire instrument process.
This sigma-algebra is invariant under \(\theta\) and \(\theta^{-1}\).
Under \Cref{assumption_1,assumption_mix}, we have \(\phi^{\mathrm{inst}}(n)\to0\), and hence
\[
    \pr\!\left(A\cap\theta^{-n}B\right)
    \longrightarrow \pr(A)\pr(B),
    \qquad A,B\in\mcG^{\mathrm{inst}}.
\]
Indeed, this follows first for events depending on finite instrument blocks from the definition of \(\phi^{\mathrm{inst}}\), and then for all events in \(\mcG^{\mathrm{inst}}\) by approximation in probability.
Consequently, \(\left(\Omega,\mcG^{\mathrm{inst}},\pr|_{\mcG^{\mathrm{inst}}},\theta \right)\)
is mixing and therefore ergodic.
This does not necessarily imply ergodicity of the full environment \((\Omega,\mcF,\pr,\theta)\).
\end{remark}

\begin{restatable}[Quenched Berry--Esseen theorem]{thm}{quenchedberryesseenposterior}
 \label{thm:quenched_berry_esseen_posterior}
Make \Cref{assumption_1,assumption_doeblin,assumption_mix}.
Let \(f:\Omega\times\states\to\mbR\) be bounded and measurable.
Suppose that \(K_\omega\) and \(f_\omega\) depend only on the current one-step instrument in the sense of \Cref{def:depends-only-current-instrument}.
Then there exists a deterministic constant \(\Sigma\ge0\) such that, for every measurable initial state \(\vartheta:\Omega\to\states\), there is a measurable set \(\Omega_0\subseteq\Omega\) with \(\pr(\Omega_0)=1\) on which
\[
    \lim_{N\to\infty}
    \frac1N
    \operatorname{Var}_{\mbK_{\vartheta;\omega}}
    \left(S_N^\omega f\right)
    =\Sigma^2.
\]
If \(\Sigma>0\), then for every \(\omega\in\Omega_0\),
\[
    \sup_{t\in\mbR}
    \left|
        \mbK_{\vartheta;\omega}
        \left\{
            \frac{\bar S_N^\omega f}{\Sigma\sqrt N}\le t
        \right\}
        -\Phi(t)
    \right|
    =
    O\!\left(N^{-1/2}(\log N)^{3/2}\right)
    \qquad (N\to\infty),
\]
where \(\Phi\) denotes the standard normal distribution function.
Whenever \Cref{thm:quenched_clt_posterior} applies, this is the same \(\Sigma\) as in that theorem for the same function $f$.
\end{restatable}

\paragraph{Annealed coding and inherited mixing.}
    After the quenched results, we turn to the equilibrium annealed law, which averages over both the environment and the measurement outcomes.
    The instrument coding places the joint instrument--posterior process in a stationary framework, allowing us to relate its temporal dependence to that of the instrument environment.

    For the annealed results, we also use the canonical instrument coding from \Cref{subsection:canonical-coding-instrument-process}.
    With the generating algebra \((C_m)_{m\in\mbN}\) fixed there, write
    \[
        \mathbf I_\omega
        :=
        \left(J(\mfI_\omega(C_m))\right)_{m\in\mbN}\in \mathrm{Instr},
        \qquad
        \widehat\Omega:=\mathrm{Instr}^{\mbZ}.
    \]
    Let \(\widehat\theta:\widehat\Omega\to\widehat\Omega\) be the left shift, let \(\pi_k:\widehat\Omega\to\mathsf{Instr}\) denote the \(k\)-th coordinate map, and let
    \[
        \iota:\Omega\to\widehat\Omega,
        \qquad
        \iota(\omega):=\left(\mathbf I_{\theta^n\omega}\right)_{n\in\mbZ},
    \]
    be the coding map.
    
    Let \(\mcY:=\states^{\mbZ}\), and let \(\sigma:\mcY\to\mcY\) be the left shift. Define
    \[
        Z_n:\Omega\times\mcY\to\Omega\times\mcY,
        \qquad
        Z_n(\omega,y):=(\theta^n\omega,\sigma^n y)
        =
        (\theta\otimes\sigma)^n(\omega,y),
        \qquad n\ge0.
    \]
    For \(j\in\mbZ\), let
    \[
        \varrho_j:\mcY\to\states,
        \qquad
        \varrho_j(y):=y_j,
    \]
    denote the \(j\)-th coordinate map.

Under the standing null-set convention of \Cref{section:prelim}, we may assume that the equivariant family \(\omega\mapsto\mu_\omega\in\mcP(\states)\) from \Cref{thm:s_exists} is defined for every \(\omega\in\Omega\) and satisfies
\[
    \mu_\omega K_\omega = \mu_{\theta\omega}
    \qquad\text{for every }\omega\in\Omega.
\]
For \(n\ge0\), let
\[
    \widetilde{\mbN}_n:=\{-n,-n+1,\dots\},
\]
and let \(\bar\kappa_{\omega,n}\) be the probability measure on \(\states^{\widetilde{\mbN}_n}\) obtained by starting, at time \(-n\), with initial law \(\mu_{\theta^{-n}\omega}\) and evolving forward with the successive kernels
\[
    K_{\theta^j\omega},
    \qquad j\ge -n.
\]
Since \(\mu_{\theta^j\omega}K_{\theta^j\omega}=\mu_{\theta^{j+1}\omega}\), the family \((\bar\kappa_{\omega,n})_{n\ge0}\) is projectively consistent. Hence, by the Kolmogorov extension theorem, there exists a unique probability measure
\[
    \bar\kappa_\omega\in\mcP(\mcY)
\]
whose restriction to the coordinates \(\widetilde{\mbN}_n\) is \(\bar\kappa_{\omega,n}\) for every \(n\ge0\).

Define the annealed state-path measure \(\bar\kappa\in\mcP(\Omega\times\mcY)\) by
\[
    \bar\kappa(d\omega,dy):=\pr(d\omega)\,\bar\kappa_\omega(dy).
\]
Then \(\bar\kappa\) is \(\theta\otimes\sigma\)-invariant.

Now define
\[
    \Pi:\Omega\times\mcY\to \widehat\Omega\times\mcY,
    \qquad
    \Pi(\omega,y):=(\iota(\omega),y),
\]
and let
\begin{equation}
\label{eq:kappa_from_kappa_bar}
    \kappa:=\bar\kappa\circ\Pi^{-1}\in\mcP(\widehat\Omega\times\mcY).
\end{equation}
Since \(\iota\circ\theta=\widehat\theta\circ\iota\), it follows that \(\kappa\) is \(\widehat\theta\otimes\sigma\)-invariant. 
Thus, for every bounded measurable \(f:\widehat\Omega\times\mcY\to\mbR\), one may consider the partial sums
\[
    S_n f
    :=
    \sum_{j=0}^{n-1}
    f\circ(\widehat\theta\otimes\sigma)^j.
\]
Equivalently, on \((\Omega\times\mcY,\bar\kappa)\),
\[
    S_n f(\omega,y)
    =
    \sum_{j=0}^{n-1}
    f\left(\Pi(Z_j(\omega,y))\right).
\]

To state the annealed mixing result, we define, for two sub-\(\sigma\)-algebras \(\mcG,\mcH\subseteq\mcF\),
\[
    \alpha_{\pr}(\mcG,\mcH)
    :=
    \sup_{A\in\mcG,\;B\in\mcH}
    \left|
        \pr(A\cap B)-\pr(A)\pr(B)
    \right|.
\]
We also set
\[
    \alpha^{\mathrm{inst}}(n)
    :=
    \sup_{k\in\mbZ}
    \alpha_{\pr}\!\left(
        \mcF_k^{\mathrm{inst},-},
        \mcF_{k+n}^{\mathrm{inst},+}
    \right),
    \qquad n\ge0,
\]
where \(\mcF_k^{\mathrm{inst},-}\) and \(\mcF_k^{\mathrm{inst},+}\) are the past and future \(\sigma\)-algebras generated by the instrument process from \Cref{subsection:canonical-coding-instrument-process}.

For integers \(k\le \ell\), let \(\mcH_{k,\ell}\) denote the \(\sigma\)-algebra on \(\widehat\Omega\times\mcY\) generated by the coordinate maps
\[
    \pi_j:\widehat\Omega\to\mathrm{Instr},
    \qquad
    \varrho_j:\mcY\to\states,
    \qquad
    k\le j\le \ell.
\]
Equivalently,
\[
    \mcH_{k,\ell}
    =
    \sigma\left(\pi_j,\varrho_j:\,k\le j\le \ell\right).
\]
We further write
\[
    \mcH_{-\infty,k}
    :=
    \sigma\left(\pi_j,\varrho_j:\,j\le k\right),
    \qquad
    \mcH_{k,\infty}
    :=
    \sigma\left(\pi_j,\varrho_j:\,j\ge k\right).
\]
The corresponding annealed \(\alpha\)-mixing profile is
\[
    \alpha_\kappa(m)
    :=
    \sup_{k\in\mbZ}
    \alpha_\kappa\!\left(
        \mcH_{-\infty,k},
        \mcH_{k+m,\infty}
    \right),
    \qquad m\ge1.
\]

\begin{restatable}[Inherited annealed \(\alpha\)-mixing]{thm}{annealedalphamixing} 
\label{thm:annealed_alpha_mixing}
    Make  \Cref{assumption_1,assumption_doeblin}, and suppose that $K_\omega$ depends only on the current one-step instrument in the sense of \Cref{def:depends-only-current-instrument}.
    There exist constants \(\zeta\in(0,1)\) and \(C_1>0\) such that for every \(m\ge1\),
    \[
        \alpha_\kappa(m)
        \le
        C_1\zeta^m
        +
        \alpha^{\mathrm{inst}}\!\left(\lfloor m/3\rfloor\right).
    \]
\end{restatable}

To interpret \Cref{thm:annealed_alpha_mixing}, note that the coordinate process
\[
    z_j:\widehat\Omega\times\mcY\to \mathrm{Instr}\times\states,
    \qquad
    z_j(\widehat\omega,y):=(\pi_j(\widehat\omega),\varrho_j(y)),
    \qquad j\in\mbZ,
\]
is stationary under \(\kappa\). Thus \Cref{thm:annealed_alpha_mixing} shows that the two-sided annealed process \((z_j)_{j\in\mbZ}\) inherits the \(\alpha\)-mixing of the instrument environment, up to an additive exponentially decaying term.

Accordingly, for bounded measurable \(f:\widehat\Omega\times\mcY\to\mbR\), one may apply the results of \cite{KV,HK,YHMD} to the partial sums \(S_nf\), under decay assumptions on the local approximation coefficients
\[
    \beta_{f,p}(r)
    :=
    \norm{
        f-\mbE_\kappa\!\left[f\,\middle|\,\mcH_{-r,r}\right]
    }_{L^p(\kappa)},
    \qquad r\to\infty,
\]
together with suitable decay of \(\alpha^{\mathrm{inst}}(n)\) as \(n\to\infty\).
The same framework also yields limit theorems for nonconventional sums of the form
\[
    \sum_{j=1}^n
    F\circ(\widehat\theta\otimes\sigma)^{q_1(j)}
    \times\cdots\times
    F\circ(\widehat\theta\otimes\sigma)^{q_\ell(j)},
\]
or, equivalently, for functions of the tuple
\[
    \left(
        \Pi(Z_{q_1(j)}),\dots,\Pi(Z_{q_\ell(j)})
    \right),
\]
where \(q_1,\dots,q_\ell\) are integer-valued polynomials and \(F\) is sufficiently regular.

\paragraph{Markovian coded instruments.}
    Suppose that the coded instrument process \(\left(\mathbf I_{\theta^n\omega}\right)_{n\in\mbZ}
    \) is a Markov chain under \(\pr\), and assume in addition that the posterior kernel \(\omega\mapsto K_\omega\) depends only on the current one-step instrument in the sense of \Cref{def:depends-only-current-instrument}. 
    Then the annealed coordinate process
    \[
        z_j=(\pi_j,\varrho_j),\qquad j\in\mbZ,
    \]
    is a Markov chain on \(\mathrm{Instr}\times\states\) under \(\kappa\).
    If moreover the coded instrument chain satisfies a Doeblin minorization on \(\mathrm{Instr}\), and if the minorizing family \(\omega\mapsto\lambda_\omega\) from \Cref{assumption_doeblin} also depends only on the current one-step instrument, then the chain \((z_j)\) satisfies a Doeblin minorization on \(\mathrm{Instr}\times\states\).
    In particular, \((z_j)\) is exponentially fast \(\phi\)-mixing, and one may then invoke \cite{HafMS1,HafMS2} to obtain Berry--Esseen bounds and local central limit theorems for suitable functions of the annealed process.

\subsection{Implications in the homogeneous deterministic case}


To consider the homogeneous deterministic case, where the same instrument is applied at each step, one can take \(\Omega=\{\omega_0\}\), \(\pr=\delta_{\omega_0}\), and $\theta$ to be the identity map on $\Omega$.
Here the quenched and annealed posterior laws coincide, and \(\alpha^{\mathrm{inst}}(n)=\phi^{\mathrm{inst}}(n)=0\).
Under \Cref{assumption_doeblin}, the posterior kernel \(K\) has a unique invariant law \(\mu\) and attracts every initial law uniformly in total variation at a geometric rate, as in the classical
Doeblin theorem \cite{RR04}.
It is also \(\lambda\)-irreducible: every measurable set \(B\) with \(\lambda(B)>0\) is accessible from every initial state, since
\[
    K^L(\rho,B)\ge\varepsilon\lambda(B)>0
    \qquad\text{for every }\rho\in\states.
\]

The non-selective channel \(\Phi=\mfI(\Xi)\) has the unique invariant state
\[
    \s=\int_{\states}\eta\,\mu(\dee\eta),
    \qquad
    \sup_{\rho\in\states}\norm{\Phi^n(\rho)-\s}_1
    \le 2(1-\varepsilon)^{\lfloor n/L\rfloor}.
\]
Thus \(\Phi\) is mixing in the sense that its iterates converge in trace norm to the same state \(\s\) for every initial state \cite{Bur+13}.
Here \emph{mixing} refers to convergence of the channel iterates, whereas \Cref{assumption_mix} concerns the decay of statistical dependence in the instrument process. 
In particular, \(\Phi\) satisfies \((\Phi\text{-}\mathrm{Erg})\), the condition of uniqueness of the invariant density matrix used in \cite{BFPP19}.
It need not be irreducible: in dimension \(d>1\), the replacement channel \(\Phi(X)=\tr{X}p\), with \(p\) pure, preserves the proper subspace \(\operatorname{ran}p\) and admits an instrument with posterior kernel \(K(\rho,\,\cdot\,)=\delta_p\).

The identification of an invariant posterior law's barycenter with the invariant channel state is already known for homogeneous perfect measurements \cite{BFPP19} and for imperfect measurements \cite{tristant_2025}.
For perfect instruments, \Cref{rem:doeblin-scope} gives the stronger conclusion that the trajectory reaches the pure-state set in finite time almost surely.
For imperfect instruments, the minorization need not imply even asymptotic purification, as illustrated by replacement with a fixed mixed state.


\section{Notation and Preliminaries}
\label{section:prelim}


This section records the measurable-space, dynamical, and operator-theoretic conventions used throughout the paper. 
It then passes from the abstract instrument family to a measurable density realization and develops the three further constructions used throughout the paper: the quenched outcome law, the canonical coding of the instrument process, and the posterior-state processes in their outcome-path and state-path forms.


\subsection{General conventions}
\label{subsection:general-conventions}



\subsubsection*{Measurable spaces, products, and kernels}


For a measurable space \((S,\mcS)\), we write \(\mcP(S)\) for the set of probability measures on \((S,\mcS)\). If \((S,\mcS)\) and \((T,\mcT)\) are measurable spaces, a map \(f:S\to T\) is called \((\mcS,\mcT)\)-measurable if
\[
    f^{-1}(B)\in\mcS
    \qquad
    \text{for every }B\in\mcT.
\]
When the \(\sigma\)-algebras are clear from the context, we simply say that \(f\) is measurable.

If \((S,\mcS)\) is a measurable space and \(n\in\mbN\), we write \(\mcS^{\otimes n}\) for the \(n\)-fold product \(\sigma\)-algebra on \(S^n\), and \(\mcS^{\otimes\mbN}\) for the \(\sigma\)-algebra on \(S^{\mbN}\) generated by cylinder sets. We also use the convention
\[
    S^0:=\{\ast\},
    \qquad
    \mcS^{\otimes 0}:=\{\varnothing,\{\ast\}\}.
\]
If \(S\) is a topological space, we write \(\borel{S}\) for its Borel \(\sigma\)-algebra. A measurable space is called \emph{standard Borel} if it is measurably isomorphic to the Borel space of a Polish space. 
Finite and countable products of standard Borel spaces are again standard Borel. 
In particular, since \((\Xi,\mcX)\) is standard Borel, so are \((\Xi^n,\mcX^{\otimes n})\) for \(n\in\mbN\) and \((\Xi^{\mbN},\mcX^{\otimes\mbN})\). 
Whenever \((S,\mcS)\) is standard Borel, we equip \(\mcP(S)\) with the Borel \(\sigma\)-algebra generated by the evaluation maps \(\mcP(S)\ni\mu\longmapsto\mu(A)\) for $A\in \mcS$; this makes \(\mcP(S)\) itself a standard Borel space.

Given measurable spaces \((S,\mcS)\) and \((T,\mcT)\), a map
\[
    M:S\times\mcT\to[0,\infty]
\]
is called a \emph{kernel from \((S,\mcS)\) to \((T,\mcT)\)} if, for each \(s\in S\), the set function \(B\mapsto M(s,B)\) is a measure on \((T,\mcT)\), and for each \(B\in\mcT\), the map \(s\mapsto M(s,B)\) is \(\mcS\)-measurable. 
If each \(M(s,\,\cdot\,)\) is a finite measure, we call \(M\) a \emph{finite positive kernel}; if each \(M(s,\,\cdot\,)\) is a probability measure, we call \(M\) a \emph{probability kernel}.

If \(M\) is a kernel from \((S,\mcS)\) to \((T,\mcT)\) and \(N\) is a kernel from \((T,\mcT)\) to \((U,\mcU)\), their composition is the kernel \(MN = N \circ M\) from \((S,\mcS)\) to \((U,\mcU)\) defined by
\[
    (MN)(s,C)
    :=
    \int_T N(t,C)\,M(s,\dee t),
    \qquad
    s\in S,\quad C\in\mcU.
\]
If \(\mu\in\mcP(S)\) and \(M\) is a probability kernel from \((S,\mcS)\) to \((T,\mcT)\), we write
\[
    (\mu M)(B)
    :=
    \int_S M(s,B)\,\mu(\dee s),
    \qquad
    B\in\mcT.
\]
In particular, if \(\zeta\in\mcP(\states)\) and \(K\) is a probability kernel on \((\states,\borel{\states})\), we write
\[
    (\zeta K)(B)
    :=
    \int_{\states} K(\rho,B)\,\zeta(\dee\rho),
    \qquad
    B\in\borel{\states},
\]
and, when \(K=K_\omega\), we also write
\[
    K_\omega^*\zeta:=\zeta K_\omega.
\]

\smallskip

For a measurable space \(E\), let \(B_b(E)\) denote the Banach space of bounded measurable complex-valued functions on \(E\), equipped with the supremum norm
\[
    \norm{g}_\infty
    :=
    \sup_{x\in E}\left|g(x)\right|.
\]
In particular, we use this convention when \(E=\states\) or \(E=\Omega\times\states\), equipped with the respective Borel and product \(\sigma\)-algebras. 
For measurable spaces \(E,F\) and a bounded linear operator \(T:B_b(E)\to B_b(F)\), we use the induced operator norm
\[
    \norm{T}_{\infty\to\infty}
    :=
    \sup_{\substack{
        g\in B_b(E)\\
        \norm{g}_\infty\le1
    }}
    \norm{Tg}_\infty.
\]
For a bounded linear functional \(\ell\in B_b(E)^*\), we use the dual norm
\[
    \norm{\ell}_{B_b(E)^*}
    :=
    \sup_{\substack{
        g\in B_b(E)\\
        \norm{g}_\infty\le1
    }}
    \left|\ell(g)\right|.
\]

\smallskip

For \(\zeta_1,\zeta_2\in\mcP(\states)\), we use the total variation norm
\[
    \norm{\zeta_1-\zeta_2}_{\mathrm{TV}}
    :=
    \sup_{\norm{f}_\infty\le1}
    \left|
        \int_{\states} f(\rho)\,(\zeta_1-\zeta_2)(\dee\rho)
    \right|.
\]
Equivalently, for probability measures,
\[
    \norm{\zeta_1-\zeta_2}_{\mathrm{TV}}
    =
    2\sup_{B\in\borel{\states}}
    |\zeta_1(B)-\zeta_2(B)|.
\]


\subsubsection*{Mixing coefficients}


Let \((X,\mcA,\lambda)\) be a probability space, and let \(\mcU,\mcV\subseteq\mcA\) be sub-\(\sigma\)-algebras. 
We define
\[
    \phi_\lambda(\mcU,\mcV)
    :=
    \sup\left\{
        |\lambda(B\mid A)-\lambda(B)| :
        A\in\mcU,\ B\in\mcV,\ \lambda(A)>0
    \right\},
\]
where
\[
    \lambda(B\mid A):=\frac{\lambda(A\cap B)}{\lambda(A)}
    \qquad
    \text{whenever }\lambda(A)>0.
\]

Now let \((\mcG_j)_{j\in\mbZ}\) be a family of sub-\(\sigma\)-algebras of \(\mcA\). 
For \(k\in\mbZ\) and \(n\ge1\), set
\[
    \mcG_{-\infty}^k
    :=
    \sigma\!\left(\bigcup_{j\le k}\mcG_j\right),
    \qquad
    \mcG_{k+n}^{\infty}
    :=
    \sigma\!\left(\bigcup_{j\ge k+n}\mcG_j\right).
\]
The associated \(\phi\)-mixing profile is
\[
    \phi_{\lambda,\mcG}(n)
    :=
    \sup_{k\in\mbZ}
    \phi_\lambda(\mcG_{-\infty}^k,\mcG_{k+n}^{\infty}).
\]
When \(\mcG_j=\sigma(Z_j)\) for a two-sided process \(Z=(Z_j)_{j\in\mbZ}\), we also write \(\phi_\lambda^Z(n)\) for these coefficients. 
When the underlying probability measure \(\lambda\) is clear from the context, we omit the subscript \(\lambda\).


\subsubsection*{The base dynamical system}


A map \(\theta:\Omega\to\Omega\) is called \emph{\(\pr\)-preserving} if it is \((\mcF,\mcF)\)-measurable and
\[
    \pr(\theta^{-1}A)=\pr(A)
    \qquad
    \text{for every }A\in\mcF.
\]
It is called \emph{invertible} if it is bijective and \(\theta^{-1}:\Omega\to\Omega\) is also \((\mcF,\mcF)\)-measurable. In that case \(\theta^{-1}\) is automatically \(\pr\)-preserving as well. 
We say that \((\Omega,\mcF,\pr,\theta)\) is \emph{ergodic} if every \(\theta\)-invariant set \(A\in\mcF\) satisfies \(\pr(A)\in\{0,1\}\). 
Equivalently, \((\Omega,\mcF,\pr,\theta)\) is ergodic if every essentially \(\theta\)-invariant measurable set \(E\in\mcF\), that is, every \(E\in\mcF\) such that \(\pr(E\triangle\theta^{-1}(E))=0\), satisfies \(\pr(E)\in\{0,1\}\). 
See \cite{erg_walter} for other equivalent characterizations.


\subsubsection*{Matrices and superoperators}


Recall that
\[
    \matrices := \mbM_d(\mbC),
    \qquad
    \states := \left\{\rho\in\matrices : \rho\ge 0,\ \tr{\rho}=1\right\}.
\]
We equip \(\states\) with the topology induced by the trace norm and write \(\borel{\states}\) for the corresponding Borel \(\sigma\)-algebra. 
Since \(\matrices\) is finite-dimensional, all norms on \(\matrices\) are equivalent, and \(\states\) is a compact Polish space.

We write \(\mapspace:=\mcL(\matrices)\) for the space of linear maps from \(\matrices\) to itself. 
For \(X\in\matrices\), we denote by \(X\adj\) its conjugate transpose. For \(1\le p<\infty\), the Schatten \(p\)-norm is defined by
\[
    \norm{X}_p
    :=
    \left(\tr{|X|^p}\right)^{1/p},
    \qquad
    |X|:=(X\adj X)^{1/2},
\]
and \(\norm{X}_\infty\) denotes the operator norm. In particular,
\[
    \norm{X}_1=\tr{|X|}
\]
is the trace norm. We shall use the duality between the Schatten \(1\)- and \(\infty\)-norms:
\[
    \norm{X}_1
    =
    \sup_{\norm{H}_\infty\le1}
    \left|\tr{H\adj X}\right|.
\]
If \(X=X\adj\), this is equivalently
\[
    \norm{X}_1
    =
    \sup_{\substack{\norm{H}_\infty\le1\\ H=H\adj}}
    \left|\tr{HX}\right|.
\]
See \cite{Watrous_2018} for details.

For \(\Psi\in\mapspace\), we write
\[
    \norm{\Psi}_{1\to1}
    :=
    \sup_{X\neq0}\frac{\norm{\Psi(X)}_1}{\norm{X}_1}.
\]
We also write
\[
    \mathsf{CP}(\matrices)
    :=
    \left\{
        \Psi\in\mapspace :
        \Psi \text{ is completely positive}
    \right\},
\]
\[
    \mathsf{CP}_{\le}(\matrices)
    :=
    \left\{
        \Psi\in\mapspace :
        \Psi \text{ is completely positive and trace-nonincreasing}
    \right\},
\]
and
\[
    \mathsf{CPTP}(\matrices)
    :=
    \left\{
        \Psi\in\mapspace :
        \Psi \text{ is completely positive and trace-preserving}
    \right\}.
\]
Since \(\mapspace\) is finite-dimensional, we equip it with its Borel \(\sigma\)-algebra, and each of the subsets above with the induced subspace \(\sigma\)-algebra.

Fix an orthonormal basis \((e_i)_{i=1}^d\) of \(\mbC^d\), and let \(E_{ij}\in\matrices\) be the associated matrix units. 
The Choi--Jamio\l kowski map is
\[
    J:\mapspace\to \matrices\otimes\matrices \cong \mbM_{d^2}(\mbC),
    \qquad
    J(\Psi):=\sum_{i,j=1}^d \Psi(E_{ij})\otimes E_{ij}.
\]
Then \(\Psi\in\mathsf{CP}(\matrices)\) if and only if \(J(\Psi)\ge0\). Moreover, every \(\Psi\in\mathsf{CP}(\matrices)\) admits a Kraus decomposition with at most \(d^2\) Kraus operators.

For \(n\ge1\), the \(n\)-step non-selective channel cocycle is
\[
    \Phi_\omega^{(n)}
    :=
    \Phi_{\theta^{n-1}\omega}\circ\cdots\circ\Phi_\omega,
\]
with the convention \(\Phi_\omega^{(0)}:=\mathrm{Id}_{\matrices}\). We refer to \cite{Watrous_2018} for background on complete positivity and quantum channels.


\subsection{Measurable realization of the instrument family}
\label{subsection:instrument-realization}


We now pass from the abstract instrument family
\(
    \omega\mapsto \mfI_\omega
\)
to a measurable density realization. The first step is a measurable Radon--Nikodym theorem for finite kernels on a standard Borel target space, which we use to extract measurable scalar densities from the Choi-valued instrument measure.

We need the following result, which is a direct corollary of the kernel-density theorem (see \cite[Theorem~1.28]{Kallenberg_2017} or \cite[Chapter~V-Theorem~4.44]{Cinlar_2011}), specialized to finite kernels on a standard Borel target space.

\begin{lemma}[Measurable Radon--Nikodym density for kernels]
\label{lem:scalar-kernel-rn}
    Let \((\Omega,\mcF,\pr)\) be a probability space, let \((\Xi,\mcX)\) be a standard Borel space, and let \(\omega\mapsto\mu_\omega\) and \(\omega\mapsto\nu_\omega\) be finite positive kernels from \((\Omega,\mcF)\) to \((\Xi,\mcX)\).
    Assume that
    \[
        \mu_\omega\ll \nu_\omega
        \qquad
        \text{for \(\pr\)-almost every }\omega\in\Omega.
    \]
    Then there exist a jointly measurable map \(f:\Omega\times\Xi\to[0,\infty)\) and a measurable set \(\Omega_0\in\mcF\) with \(\pr(\Omega_0)=1\) such that for every \(\omega\in\Omega_0\) and every \(A\in\mcX\),
    \begin{equation}
    \label{eq:kernel_R-N}
        \mu_\omega(A)
        =
        \int_A f(\omega,x)\,\nu_\omega(\dee x).
    \end{equation}
    Moreover, for every \(\omega\in\Omega_0\), if \(g:\Omega\times\Xi\to[0,\infty)\) is another jointly measurable map that satisfies the condition in \Cref{eq:kernel_R-N} for all $A\in \mcX$, then \(f(\omega,\,\cdot\,)=g(\omega,\,\cdot\,)\)  $\nu_\omega$-almost everywhere.
\end{lemma}
\begin{proof}
    Since \((\Xi,\mcX)\) is standard Borel, it is a Borel space in the sense of \cite{Kallenberg_2017}. 
    Since the kernels \(\mu\) and \(\nu\) are finite, they are in particular \(\sigma\)-finite. 
    Hence Theorem 1.28 of \cite{Kallenberg_2017} applies and yields an $(\mcF\otimes \mcX)$-measurable function 
    \[
        f_0:\Omega\times\Xi\to[0,\infty]
    \]
    such that, for each \(\omega\in\Omega\) and every $A \in \mcX$, $\mu_\omega =f_0(\omega,\,\cdot\,)\,\nu_\omega$
    on the set $\{x:f_0(\omega,x)<\infty\}$, i.e., 
    \[
        \mu_\omega
        \left(
            A \cap \{x : f_0(\omega, x) < \infty\}
        \right)
        =
        \int_{A \cap \{x : f_0(\omega,x)<\infty\}} f_0(\omega, x) \, \nu_\omega(\dee x).
    \]
    Furthermore, the restriction of \(\mu_\omega\) to \( \{x\in\Xi:f_0(\omega,x)=\infty\}\) is singular with respect to \(\nu_\omega\).
    By the same theorem, the set
    \[
        \Omega_0
        :=
        \{\omega\in\Omega:\mu_\omega\ll\nu_\omega\}
    \]
    is $\mcF$-measurable. 
    By assumption, \(\pr(\Omega_0)=1\).
    
    Fix \(\omega\in\Omega_0\). Since \(\mu_\omega\ll\nu_\omega\), the restriction of \(\mu_\omega\) to \(\{f_0(\omega,\,\cdot\,)=\infty\}\) is also absolutely continuous with respect to \(\nu_\omega\). 
    But that restriction is singular with respect to \(\nu_\omega\) by Theorem 1.28 of \cite{Kallenberg_2017}. Hence, it must be the zero measure. 
    Therefore, for every \(A\in\mcX\),
    \[
        \mu_\omega(A)
        =
        \int_{A\cap\{f_0(\omega,\,\cdot\,)<\infty\}}
        f_0(\omega,x)\,\nu_\omega(\dee x).
    \]
    
    Define
    \[
        f(\omega,x):=
        \begin{cases}
            f_0(\omega,x), & f_0(\omega,x) < \infty ,\\
            0, & f_0(\omega,x) = \infty.
        \end{cases}
    \]
    Then \(f:\Omega\times\Xi\to[0,\infty)\) is jointly measurable, and for every \(\omega\in\Omega_0\) and every \(A\in\mcX\),
    \[
        \mu_\omega(A)
        =
        \int_A f(\omega,x)\,\nu_\omega(\dee x).
    \]
    For uniqueness, fix \(\omega\in\Omega_0\), and suppose that \(g:\Omega\times\Xi\to[0,\infty)\) is another jointly measurable map such that
    \[
        \mu_\omega(A)
        =
        \int_A g(\omega,x)\,\nu_\omega(\dee x)
        \qquad
        \text{for all }A\in\mcX.
    \]
    Then both \(f(\omega,\,\cdot\,)\) and \(g(\omega,\,\cdot\,)\) are Radon--Nikodym derivatives of \(\mu_\omega\) with respect to \(\nu_\omega\).
    By the usual uniqueness part of the Radon--Nikodym theorem, \(f(\omega,\,\cdot\,)=g(\omega,\,\cdot\,)\) \(\nu_\omega\)-almost everywhere.
\end{proof}

We now apply \Cref{lem:scalar-kernel-rn} to the Choi-valued measure associated with the instrument family. 
This yields the measurable density realization that will be fixed throughout the paper.

\smallskip

The following proposition may be viewed as a measurable-parametric refinement of a standard fact in finite-dimensional quantum measurement theory. 
For a single instrument, representations by a scalar control measure together with a completely positive density are classical; see Holevo's Radon--Nikodym theorem for quantum instruments \cite{Holevo1998RN} and, in the finite-dimensional language of Choi measures, \cite[Proposition~2]{ChiribellaDArianoPerinotti2009}. 
The point of the present statement is that, for a measurable family \((\mfI_\omega)_{\omega\in\Omega}\), one can choose these densities jointly measurably in \((\omega,x)\). 
The proof below is essentially the finite-dimensional Choi-matrix argument underlying \cite[Proposition~2]{ChiribellaDArianoPerinotti2009}, upgraded to the kernel setting via Lemma~\ref{lem:scalar-kernel-rn}. 
For general background on completely positive instruments and their dilation theory for continuous outcome spaces, see also Ozawa's foundational papers \cite{Ozawa_1984,Ozawa1985Posteriori}.

\begin{prop}
\label{prop:instrument-density-representation}
    Let \((\Xi,\mcX)\) be a standard Borel space.
    For each \(\omega\in\Omega\), let
    \[
        \mfI_\omega:\mcX\to\mathsf{CP}(\matrices)
    \]
    be an instrument such that \(\mfI_\omega(\Xi)\) is trace-preserving.
    Assume that for every \(A\in\mcX\), the map
    \[
        \omega\longmapsto J\left(\mfI_\omega(A)\right)
    \]
    is \(\mcF\)-measurable.
    Then, after possibly modifying the instrument family on a \(\pr\)-null set, there exist
    \begin{enumerate}
        \item a measurable kernel \(\omega\mapsto \nu_\omega\in\mcP(\Xi)\), and
        \item a jointly measurable map
        \[
            (\omega,x)\longmapsto \mcT_{x;\omega}\in\mathsf{CP}(\matrices),
        \]
    \end{enumerate}
    such that for every \(\omega\in\Omega\), every \(A\in\mcX\), and every
    \(X\in\matrices\),
    \[
        \mfI_\omega(A)(X)
        =
        \int_A \mcT_{x;\omega}(X)\,\nu_\omega(\dee x).
    \]
    For the fixed control kernel \(\omega\mapsto\nu_\omega\), the density \((\omega,x)\mapsto\mcT_{x;\omega}\) is unique up to \(\nu_\omega\)-almost everywhere equality in \(x\), for each fixed \(\omega\).
\end{prop}

 In the proposition above, we note that the uniqueness statement compares densities representing the same instrument family with respect to the same control kernel.
 
\begin{proof}
    Set \(n:=d^2\), and for \(\omega\in\Omega\) and \(A\in\mcX\), define \(Z_\omega(A):=J\left(\mfI_\omega(A)\right)\in \mbM_n(\mbC)\).
    Since \(\mfI_\omega(A)\) is completely positive, the Choi criterion gives
    \[
        Z_\omega(A)\ge 0
        \qquad
        \text{for all }\omega\in\Omega,\ A\in\mcX.
    \]
    Because \(J\) is linear, for each fixed \(\omega\) the map \(A\longmapsto Z_\omega(A)\) is a countably additive \(\mbM_n(\mbC)\)-valued measure on \((\Xi,\mcX)\).
    By assumption, for every \(A\in\mcX\), the map \(\omega\longmapsto Z_\omega(A)\) is \(\mcF\)-measurable.

    We first construct a scalar control kernel. Define
    \[
        \ell:\mbM_n(\mbC)\to\mbC,
        \qquad
        \ell(M):=\frac1d\,\tr{J^{-1}(M)(I)}.
    \]
    Since \(J^{-1}\) is linear and continuous, \(\ell\) is a continuous linear functional.
    For \(\omega\in\Omega\) and \(A\in\mcX\), set
    \[
        \nu_\omega(A)
        :=
        \ell\left(Z_\omega(A)\right)
        =
        \frac1d\,\tr{\mfI_\omega(A)(I)}.
    \]
    Hence, for every \(A\in\mcX\), the map \(\omega\longmapsto \nu_\omega(A)\) is \(\mcF\)-measurable.    

    For each fixed \(\omega\), the set function \(A\mapsto \nu_\omega(A)\) is a finite positive measure on \((\Xi,\mcX)\). 
    Indeed, since \(\mfI_\omega(A)\) is completely positive, one has \(\mfI_\omega(A)(I)\ge0\), hence \(\nu_\omega(A)\ge0\), and countable additivity follows from that of \(A\mapsto \mfI_\omega(A)\).
    Moreover, since \(\mfI_\omega(\Xi)\) is trace-preserving,
    \[
        \nu_\omega(\Xi)
        =
        \frac1d\,\tr{\mfI_\omega(\Xi)(I)}
        =
        \frac1d\,\tr{I}
        =
        1.
    \]
    Thus \(A\mapsto \nu_\omega(A)\) is a probability measure for every \(\omega\).
    We claim that for fixed $A\in \mcX$ the map $\omega\mapsto\nu_\omega(A)$ is $\mcF$-measurable and thus $\nu$ is a probability kernel from \((\Omega,\mcF)\) to \((\Xi, \mcX)\). 
    Fix \(A\in\mcX\). 
    By assumption, the map \(\omega\mapsto J(\mfI_\omega(A))\) is \(\mcF\)-measurable, and since \(J^{-1}\) is continuous, \(\omega\longmapsto \mfI_\omega(A)\) is also \(\mcF\)-measurable
    as an \(\mcL(\matrices)\)-valued map. 
    Since
    \[
    \Psi\longmapsto \frac1d\,\tr{\Psi(I)}
    \]
    is a continuous linear functional on \(\mcL(\matrices)\), it follows that
    \[
    \omega\longmapsto \nu_\omega(A)
    =
    \frac1d\,\tr{\mfI_\omega(A)(I)}
    \]
    is \(\mcF\)-measurable. 
    Therefore \(\omega\mapsto \nu_\omega\) is a measurable probability kernel.
    It is also worth noting that the map \(\omega\mapsto \nu_\omega(A)\) is \(\mcG_0^{\mathrm{inst}}\)-measurable, so the kernel \(\omega\mapsto \nu_\omega\) depends only on the current one-step instrument.

    We next show that \(\nu_\omega\) dominates the instrument. 
    Fix \(\omega\in\Omega\) and \(A\in\mcX\), and suppose that \(\nu_\omega(A)=0\).
    Then \(0=\tr{\mfI_\omega(A)(I)}\).
    Since \(\mfI_\omega(A)(I)\ge0\), it follows that \(\mfI_\omega(A)(I)=0\).
    Because \(\matrices\) is finite-dimensional, \(\mfI_\omega(A)\) admits  a finite Kraus decomposition,
    \[
        \mfI_\omega(A)(X)=\sum_{r=1}^m V_r X V_r\adj
        \qquad (X\in\matrices),
    \]
    for some \(V_1,\dots,V_m\in\matrices\). 
    Hence
    \[
        0
        =
        \mfI_\omega(A)(I)
        =
        \sum_{r=1}^m V_rV_r\adj.
    \]
    Each summand \(V_rV_r\adj\) is positive semidefinite, so each one must be zero.
    Thus \(V_r=0\) for every \(r\), and therefore \(\mfI_\omega(A)=0\).
    Consequently,
    \[
        Z_\omega(A)=J\left(\mfI_\omega(A)\right)=0.
    \]
    
    Let \(D:=(\mbQ+i\mbQ)^n\subseteq\mbC^n\).
    Then \(D\) is countable and dense in \(\mbC^n\). For each \(u\in D\), define
    \[
        q_u(\omega,A):=u^*Z_\omega(A)u,
        \qquad
        \omega\in\Omega,\ A\in\mcX.
    \]
    Here $u^*$ is the conjugate transpose of $u$. 
    Since \(Z_\omega(A)\ge 0\), for each fixed \(\omega\) the set function \(A\mapsto q_u(\omega,A)\) is a finite positive measure. 
    Since \(\omega\mapsto Z_\omega(A)\) is measurable for every \(A\in\mcX\), the map \(q_u\) is a finite positive kernel from \((\Omega,\mcF)\) to \((\Xi,\mcX)\).
    In fact, for each \(A\in\mcX\), the map \(\omega\mapsto q_u(\omega,A)\) is \(\mcG_0^{\mathrm{inst}}\)-measurable, so \(q_u\) depends only on the current one-step instrument.
    By the domination established above,
    \[
        q_u(\omega,\,\cdot\,)\ll \nu_\omega
        \qquad
        \text{for every }\omega\in\Omega.
    \]
    Therefore, by \Cref{lem:scalar-kernel-rn}, for each \(u\in D\) there exist a jointly measurable function
    \[
        g_u:\Omega\times\Xi\to[0,\infty)
    \]
    and a measurable set \(\Omega_u\in\mcF\) with \(\pr(\Omega_u)=1\) such that for every \(\omega\in\Omega_u\) and every \(A\in\mcX\),
    \[
        q_u(\omega,A)
        =
        \int_A g_u(\omega,x)\,\nu_\omega(\dee x).
    \]
    Set
    \[
        \Omega_*:=\bigcap_{u\in D}\Omega_u.
    \]
    Since \(D\) is countable, \(\Omega_*\in\mcF\) and \(\pr(\Omega_*)=1\).

    Let \(e_1,\dots,e_n\) be the standard basis of \(\mbC^n\). 
    Define jointly measurable functions \(z^\circ_{ij}:\Omega\times\Xi\to\mbC\) by
    \[
        z^\circ_{ii}(\omega,x):=g_{e_i}(\omega,x),
    \]
    and, for \(i\neq j\),
    \[
        z^\circ_{ij}(\omega,x)
        :=
        \frac14\left(
            g_{e_i+e_j}(\omega,x)
            -
            g_{e_i-e_j}(\omega,x)
            -
            i\,g_{e_i+i e_j}(\omega,x)
            +
            i\,g_{e_i-i e_j}(\omega,x)
        \right).
    \]
    Here we note that the vectors \(e_i\), \(e_i\pm e_j\), and \(e_i\pm i e_j\) all belong to \(D\).
    Now set
    \[
        z^\circ(\omega,x):=\left(z^\circ_{ij}(\omega,x)\right)_{i,j=1}^n
        \in \mbM_n(\mbC).
    \]
    We claim that for every \(\omega\in\Omega_*\) and every \(A\in\mcX\),
    \[
        Z_\omega(A)
        =
        \int_A z^\circ(\omega,x)\,\nu_\omega(\dee x),
    \]
    where the integral is taken entrywise (equivalently, as a Bochner integral in the finite-dimensional space \(\mbM_n(\mbC)\)). 
    For \(i=j\), this is immediate from the definition of \(z^\circ_{ii}\). For \(i\neq j\), the complex polarization identity gives
    \[
        (Z_\omega(A))_{ij}
        =
        \frac14\left(
            q_{e_i+e_j}(\omega,A)
            -
            q_{e_i-e_j}(\omega,A)
            -
            i\,q_{e_i+i e_j}(\omega,A)
            +
            i\,q_{e_i-i e_j}(\omega,A)
        \right).
    \]
    Since \(\omega\in\Omega_*\), each \(q_u(\omega,A)\) is represented by \(g_u(\omega,\,\cdot\,)\), and therefore
    \[
        (Z_\omega(A))_{ij}
        =
        \int_A z^\circ_{ij}(\omega,x)\,\nu_\omega(\dee x).
    \]
    This proves the claim.

    We next modify \(z^\circ\) on a jointly measurable fiberwise \(\nu_\omega\)-null set so that it becomes pointwise positive semidefinite on \(\Omega_*\times\Xi\).
    Fix \(u\in D\), and define
    \[
        h^\circ_u(\omega,x):=u^*z^\circ(\omega,x)u.
    \]
    Since each entry of \(z^\circ\) is a finite linear combination of the functions \(g_v\), the map \(h^\circ_u\) is jointly measurable; moreover, for each fixed \(\omega\in\Omega_*\), the function \(h^\circ_u(\omega,\,\cdot\,)\) is \(\nu_\omega\)-integrable. 
    For every \(\omega\in\Omega_*\) and every \(A\in\mcX\),
    \begin{align*}
        \int_A h^\circ_u(\omega,x)\,\nu_\omega(\dee x)
        &=
        u^*\left(\int_A z^\circ(\omega,x)\,\nu_\omega(\dee x)\right)u \\
        &=
        u^*Z_\omega(A)u \\
        &=
        q_u(\omega,A) \\
        &=
        \int_A g_u(\omega,x)\,\nu_\omega(\dee x).
    \end{align*}
    Fix \(\omega\in\Omega_*\). 
    Then the complex-valued integrable function
    \[
        f_{\omega,u}:=h^\circ_u(\omega,\,\cdot\,)-g_u(\omega,\,\cdot\,)
    \]
    satisfies
    \[
        \int_A f_{\omega,u}\,\nu_\omega=0
        \qquad
        \text{for all }A\in\mcX.
    \]
    Hence the same is true for \(\Re f_{\omega,u}\) and \(\Im f_{\omega,u}\). 
    By the usual uniqueness statement for real-valued integrable densities on the measure space \((\Xi,\mcX,\nu_\omega)\), both \(\Re f_{\omega,u}\) and \(\Im f_{\omega,u}\) vanish \(\nu_\omega\)-almost everywhere. 
    Therefore
    \[
        h^\circ_u(\omega,\,\cdot\,)=g_u(\omega,\,\cdot\,)
        \qquad
        \nu_\omega\text{-almost everywhere}
    \]
    for every \(\omega\in\Omega_*\).

    Define
    \[
        B_u
        :=
        \left\{(\omega,x)\in\Omega_*\times\Xi:
        h^\circ_u(\omega,x)\neq g_u(\omega,x)\right\}.
    \]
    Then \(B_u\in\mcF\otimes\mcX\), and for every \(\omega\in\Omega_*\),  \(\nu_\omega\left((B_u)_\omega\right)=0\).
    Since \(D\) is countable, the set
    \[
        B:=\bigcup_{u\in D} B_u
    \]
    belongs to \(\mcF\otimes\mcX\), and for every \(\omega\in\Omega_*\), we have \(\nu_\omega(B_\omega)=0\), where $B_\omega$ is the $\omega$-section of $B$, i.e., $B_\omega = \{x\in\Xi \ : \ (\omega,x) \in B\}$, and similarly for $B_u$.

    Now redefine \(z^\circ\) on \(B\) by setting
    \[
        z^\circ(\omega,x):=0
        \qquad
        \text{for }(\omega,x)\in B.
    \]
    The resulting map is still jointly measurable; for notational simplicity we continue to denote it by \(z^\circ\). Since \(B_\omega\) is \(\nu_\omega\)-null  for every \(\omega\in\Omega_*\), the identity
    \[
        Z_\omega(A)
        =
        \int_A z^\circ(\omega,x)\,\nu_\omega(\dee x)
    \]
    remains valid for every \(\omega\in\Omega_*\) and every \(A\in\mcX\). 
    We now show that
    \[
        z^\circ(\omega,x)\ge 0
        \qquad
        \text{for all }(\omega,x)\in\Omega_*\times\Xi.
    \]
    Fix \((\omega,x)\in\Omega_*\times\Xi\). 
    If \((\omega,x)\in B\), then \(z^\circ(\omega,x)=0\), so the conclusion is immediate.
    If \((\omega,x)\notin B\), then for every \(u\in D\),
    \[
        u^*z^\circ(\omega,x)u
        =
        g_u(\omega,x)
        \ge 0.
    \]
    Thus, in either case,
    \[
        u^*z^\circ(\omega,x)u\ge 0
        \qquad
        \text{for every }u\in D.
    \]
    Now fix \((\omega,x)\in\Omega_*\times\Xi\). 
    The map \(u\mapsto u^*z^\circ(\omega,x)u\) is continuous on \(\mbC^n\), and \(D\) is dense in \(\mbC^n\). 
    Since \([0,\infty)\subseteq\mbC\) is closed, it follows that
    \[
        u^*z^\circ(\omega,x)u\ge 0
        \qquad
        \text{for all }u\in\mbC^n.
    \]
    Hence \(z^\circ(\omega,x)\) is a positive semidefinite matrix in $\mbM_n(\mbC)$.

    Define, for \((\omega,x)\in\Omega_*\times\Xi\),
    \[
        \mcT^\circ_{x;\omega}:=J^{-1}\left(z^\circ(\omega,x)\right).
    \]
    Since \(J^{-1}:\mbM_n(\mbC)\to\mcL(\matrices)\) is linear and continuous, the map
    \[
        (\omega,x)\longmapsto \mcT^\circ_{x;\omega}
    \]
    is jointly measurable on \(\Omega_*\times\Xi\). 
    Because \(z^\circ(\omega,x)\ge0\), each \(\mcT^\circ_{x;\omega}\) is completely positive via the Choi--Jamio\l kowski isomorphism. 
    
    Fix \(\omega\in\Omega_*\). 
    Each entry of \(z^\circ(\omega,\,\cdot\,)\) is \(\nu_\omega\)-integrable, hence \(z^\circ(\omega,\,\cdot\,)\) is Bochner  integrable as an \(\mbM_n(\mbC)\)-valued function. 
    Since \(J^{-1}\) is continuous and \(\mcL(\matrices)\) is finite-dimensional, \(\mcT^\circ_{\,\cdot\,;\omega}=J^{-1}\circ z^\circ(\omega,\,\cdot\,)\) is Bochner integrable as an \(\mcL(\matrices)\)-valued function. 
    Therefore, for every \(A\in\mcX\),
    \begin{align*}
        J\left(\int_A \mcT^\circ_{x;\omega}\,\nu_\omega(\dee x)\right)
        =
        \int_A J(\mcT^\circ_{x;\omega})\,\nu_\omega(\dee x) 
        =
        \int_A z^\circ(\omega,x)\,\nu_\omega(\dee x) 
        =
        Z_\omega(A) 
        =
        J\left(\mfI_\omega(A)\right).
    \end{align*}
    By injectivity of \(J\),
    \[
        \mfI_\omega(A)
        =
        \int_A \mcT^\circ_{x;\omega}\,\nu_\omega(\dee x)
        \qquad
        \text{in }\mcL(\matrices),
    \]
    for every \(\omega\in\Omega_*\) and every \(A\in\mcX\). Equivalently, for every \(X\in\matrices\),
    \[
        \mfI_\omega(A)(X)
        =
        \int_A \mcT^\circ_{x;\omega}(X)\,\nu_\omega(\dee x).
    \]
    We have therefore obtained the desired representation on the full-measure set \(\Omega_*\). 
    To obtain a statement valid for every \(\omega\in\Omega\), we now  modify the family on \(\Omega\setminus\Omega_*\).

    First note that we must necessarily have that \(\Xi\neq\varnothing\): otherwise \(\Xi=\varnothing\), so \(\mfI_\omega(\Xi)=\mfI_\omega(\varnothing)=0\), which cannot be  trace-preserving. 
    Fix \(x_*\in\Xi\). Define a new family \((\mfI'_\omega)_{\omega\in\Omega}\) by
    \[
        \mfI'_\omega(A)
        :=
        \begin{cases}
            \mfI_\omega(A), & \omega\in\Omega_*,\\[0.5ex]
            \mathbf 1_A(x_*)\,\operatorname{id}_{\matrices},
            & \omega\in\Omega\setminus\Omega_*,
        \end{cases}
        \qquad A\in\mcX.
    \]
    For \(\omega\in\Omega\setminus\Omega_*\), the map \(A\mapsto \mathbf 1_A(x_*)\,\operatorname{id}_{\matrices}\) is an instrument, and \(\mfI'_\omega(\Xi)=\operatorname{id}_{\matrices}\) is trace-preserving.

    Next define a kernel \((\nu'_\omega)_{\omega\in\Omega}\) by
    \[
        \nu'_\omega
        :=
        \begin{cases}
            \nu_\omega, & \omega\in\Omega_*,\\[0.5ex]
            \delta_{x_*}, & \omega\in\Omega\setminus\Omega_*,
        \end{cases}
    \]
    and define
    \[
        \mcT_{x;\omega}
        :=
        \begin{cases}
            \mcT^\circ_{x;\omega}, & \omega\in\Omega_*,\\[0.5ex]
            \operatorname{id}_{\matrices}, & \omega\in\Omega\setminus\Omega_*,
        \end{cases}
        \qquad (\omega,x)\in\Omega\times\Xi.
    \]
    Since \(\Omega_*\in\mcF\), the map \(\omega\mapsto \nu'_\omega\) is still a measurable probability kernel, and \((\omega,x)\mapsto \mcT_{x;\omega}\) is still jointly measurable. Moreover, for \(\omega\in\Omega\setminus\Omega_*\) and \(A\in\mcX\),
    \[
        \int_A \mcT_{x;\omega}\,\nu'_\omega(\dee x)
        =
        \int_A \operatorname{id}_{\matrices}\,\delta_{x_*}(\dee x)
        =
        \mathbf 1_A(x_*)\,\operatorname{id}_{\matrices}
        =
        \mfI'_\omega(A).
    \]
    Thus
    \[
        \mfI'_\omega(A)
        =
        \int_A \mcT_{x;\omega}\,\nu'_\omega(\dee x)
        \qquad
        \text{for every }\omega\in\Omega,\ A\in\mcX.
    \]
    
    Since the proposition allows modification of the original family on a \(\pr\)-null set, we now replace \((\mfI_\omega)_{\omega\in\Omega}\) by \((\mfI'_\omega)_{\omega\in\Omega}\), and we write again  \(\nu_\omega\) for \(\nu'_\omega\). 
    Define, on all of \(\Omega\times\Xi\),
    \[
        z(\omega,x):=J(\mcT_{x;\omega}).
    \]
    Then \(z\) is jointly measurable, \(z(\omega,x)\ge0\) for all \((\omega,x)\in\Omega\times\Xi\), and for every \(\omega\in\Omega\) and every \(A\in\mcX\),
    \[
        J\left(\mfI_\omega(A)\right)
        =
        \int_A z(\omega,x)\,\nu_\omega(\dee x).
    \]
    Equivalently,
    \[
        \mfI_\omega(A)(X)
        =
        \int_A \mcT_{x;\omega}(X)\,\nu_\omega(\dee x)
        \qquad
        \text{for every }\omega\in\Omega,\ A\in\mcX,\ X\in\matrices.
    \]

    Finally, we prove the uniqueness statement for this kernel. 
    Fix \(\omega\in\Omega\), and let \(x\mapsto \widetilde{\mcT}_x\in\mathsf{CP}(\matrices)\) be an \(\mcX\)-measurable map such that
    \[
        \mfI_\omega(A)
        =
        \int_A \widetilde{\mcT}_x\,\nu_\omega(\dee x)
        \qquad
        \text{for every }A\in\mcX.
    \]
    Set
    \[
        \widetilde z(x):=J(\widetilde{\mcT}_x),
        \qquad x\in\Xi.
    \]
    Then \(\widetilde z:\Xi\to\mbM_n(\mbC)\) is \(\mcX\)-measurable, and each entry  of \(\widetilde z\) is \(\nu_\omega\)-integrable. 
    For every \(A\in\mcX\),
    \[
        \int_A z(\omega,x)\,\nu_\omega(\dee x)
        =
        J\left(\mfI_\omega(A)\right)
        =
        \int_A \widetilde z(x)\,\nu_\omega(\dee x).
    \]
    Hence, for each \(1\le i,j\le n\),
    \[
        \int_A \left(z_{ij}(\omega,x)-\widetilde z_{ij}(x)\right)\,
        \nu_\omega(\dee x)
        =
        0
        \qquad
        \text{for all }A\in\mcX.
    \]
    Applying the usual uniqueness statement for real-valued integrable densities to the real and imaginary parts, we obtain
    \[
        z_{ij}(\omega,\,\cdot\,)=\widetilde z_{ij}(\,\cdot\,)
        \qquad
        \nu_\omega\text{-almost everywhere}
    \]
    for each \(i,j\). Since there are only finitely many entries, it follows that
    \[
        z(\omega,\,\cdot\,)=\widetilde z(\,\cdot\,)
        \qquad
        \nu_\omega\text{-almost everywhere}.
    \]
    Applying \(J^{-1}\), we conclude that
    \[
        \mcT_{x;\omega}=\widetilde{\mcT}_x
        \qquad
        \nu_\omega\text{-almost everywhere in }x.
    \]
    This proves the claimed uniqueness.
\end{proof}

We also record the converse of \Cref{prop:instrument-density-representation}.
Suppose that \(\omega\mapsto\nu_\omega\in\mcP(\Xi)\) is a measurable probability kernel and that \((\omega,x)\mapsto\mcT_{x;\omega}\in\mathsf{CP}(\matrices)\) is jointly measurable.
Assume that there is a measurable set \(\Omega_0\subseteq\Omega\), with \(\pr(\Omega_0)=1\), such that, for every \(\omega\in\Omega_0\), the integrals
\[
    \Phi_\omega(X)
    :=\int_\Xi\mcT_{x;\omega}(X)\,\nu_\omega(\dee x),
    \qquad X\in\matrices,
\]
are well defined and yield a trace-preserving linear map.
For \(\omega\in\Omega_0\), define
\[
    \mfI_\omega(A)(X)
    :=\int_A\mcT_{x;\omega}(X)\,\nu_\omega(\dee x),
    \qquad A\in\mcX,\quad X\in\matrices.
\]
Then \(\mfI_\omega\) is a Davies--Lewis instrument for every \(\omega\in\Omega_0\): complete positivity follows by integrating positive Choi matrices, countable additivity follows from the integral, and trace-nonincreasingness follows from positivity and the trace-preserving total map.
For every \(A\in\mcX\), the identity
\[
    J(\mfI_\omega(A))
    =\int_A J(\mcT_{x;\omega})\,\nu_\omega(\dee x),
    \qquad\omega\in\Omega_0,
\]
shows that \(\omega\mapsto J(\mfI_\omega(A))\) is measurable on \(\Omega_0\), equipped with the restriction of \(\mcF\).
Thus the given density representation defines a normalized Davies--Lewis instrument for \(\pr\)-almost every environment.

The following is a useful observation. 

\begin{remark}
\label{rem:current-instrument-measurable-realization}
    Applying \Cref{prop:instrument-density-representation} with \((\Omega,\mcG_0^{\mathrm{inst}})\) in place of \((\Omega,\mcF)\), and modifying the family on a \(\pr\)-null set if necessary, we may and do fix a
    \(\mcG_0^{\mathrm{inst}}\)-measurable probability kernel
    \[
        \omega\longmapsto \nu_\omega\in\mcP(\Xi)
    \]
    and a \((\mcG_0^{\mathrm{inst}}\otimes\mcX)\)-measurable map
    \[
        (\omega,x)\longmapsto \mcT_{x;\omega}\in\mathsf{CP}(\matrices)
    \]
    such that
    \[
        \mfI_\omega(A)(X)
        =
        \int_A \mcT_{x;\omega}(X)\,\nu_\omega(\dee x)
    \]
    for every \(\omega\in\Omega\), every \(A\in\mcX\), and every \(X\in\matrices\).
    This density representation is auxiliary rather than part of the primitive data, but we fix one such \(\mcG_0^{\mathrm{inst}}\)-measurable realization once and for all for the pathwise constructions below.
\end{remark}

Having obtained a measurable density realization, we next record that the associated selective update is measurable as a function of the disorder, the outcome, and the prior state.

\begin{prop}
\label{prop:projective-action-is-measurable}
    Let \(\rho_\ast\in\states\) be the reference state fixed above.
    Assume that
    \[
        (\omega,x)\longmapsto \mcT_{x;\omega}
    \]
    is measurable from \((\Omega\times\Xi,\mcF\otimes\mcX)\) into \(\mathsf{CP}(\matrices)\), where \(\mathsf{CP}(\matrices)\) is equipped  with the Borel \(\sigma\)-algebra induced by the finite-dimensional vector  space \(\mcL(\matrices)\).
    Then the map
    \[
        (\omega,x,\rho)\longmapsto \mcT_{x;\omega}\proj\rho
    \]
    from \((\Omega\times\Xi\times\states,\mcF\otimes\mcX\otimes\borel{\states})\) to \((\states,\borel{\states})\) is measurable.
\end{prop}
\begin{proof}
    Write
    \[
        \mathrm{ev}:\mcL(\matrices)\times\matrices\to\matrices,
        \qquad
        \mathrm{ev}(\Phi,\rho):=\Phi(\rho).
    \]
    Since \(\matrices\) is finite-dimensional, \(\mcL(\matrices)\) is also finite-dimensional, and \(\mathrm{ev}\) is bilinear, the map \(\mathrm{ev}\) is continuous.

    Define
    \[
        F:\Omega\times\Xi\times\states\to\matrices,
        \qquad
        F(\omega,x,\rho):=\mcT_{x;\omega}(\rho).
    \]
    The map
    \[
        (\omega,x,\rho)\longmapsto (\mcT_{x;\omega},\rho)
    \]
    from \(\Omega\times\Xi\times\states\) into \(\mcL(\matrices)\times\matrices\) is measurable, and therefore
    \[
        F=\mathrm{ev}\circ\left((\omega,x,\rho)\mapsto(\mcT_{x;\omega},\rho)\right)
    \]
    is measurable.

    Since each \(\mcT_{x;\omega}\) is completely positive and \(\rho\ge 0\), the matrix \(F(\omega,x,\rho)\) is positive semidefinite for every \((\omega,x,\rho)\). Thus \(F\) takes values in
    \[
        \matrices_+
        :=
        \{A\in\matrices:A\ge 0\},
    \]
    equipped with the subspace Borel \(\sigma\)-algebra inherited from \(\matrices\).

    The trace functional
    \[
        t:\matrices_+\to\mbR_+,
        \qquad
        t(A):=\tr{A},
    \]
    is continuous. 
    Hence
    \[
        E
        :=
        \{(\omega,x,\rho):\tr{\mcT_{x;\omega}(\rho)}>0\}
        =
        (t\circ F)^{-1}\left((0,\infty)\right)
    \]
    is measurable.

    Now define
    \[
        N:\matrices_+\to\states,
        \qquad
        N(A)
        :=
        \begin{cases}
            \dfrac{A}{\tr{A}}, & \tr{A}>0,\\[1.2ex]
            \rho_\ast, & \tr{A}=0.
        \end{cases}
    \]
    We claim that \(N\) is Borel measurable. Indeed, the set
    \[
        U:=\{A\in\matrices_+:\tr{A}>0\}
    \]
    is relatively open in \(\matrices_+\), and on \(U\) the map \(A\mapsto A/\tr{A}\) is continuous.
    On the complement
    \[
        \{A\in\matrices_+:\tr{A}=0\},
    \]
    we necessarily have \(A=0\), since a positive semidefinite matrix with zero trace must vanish. Hence \(N\) is constant on this set.
    Therefore \(N\) is Borel measurable.

    Finally,
    \[
        \mcT_{x;\omega}\proj\rho
        =
        N\left(F(\omega,x,\rho)\right),
    \]
    so the map \((\omega,x,\rho)\mapsto \mcT_{x;\omega}\proj\rho\) is measurable as a composition of measurable maps.
\end{proof}

We can now define the one-step outcome and posterior-state kernels associated with the chosen density realization and verify that they are genuine measurable probability kernels.

\begin{prop}
\label{prop:Gamma-K-are-kernels}
    Under the assumptions of \Cref{prop:instrument-density-representation}, fix the density representation  \((\nu_\omega,\mcT_{x;\omega})\), and let \(\rho_\ast\in\states\) be the reference state fixed above.
    For \(\omega\in\Omega\), \(\rho\in\states\), \(A\in\mcX\), and \(B\in\borel{\states}\), define
    \[
        \Gamma_\omega(\rho,A)
        :=
        \tr{\mfI_\omega(A)(\rho)}
        =
        \int_A \tr{\mcT_{x;\omega}(\rho)}\,\nu_\omega(\dee x),
    \]
    and
    \[
        K_\omega(\rho,B)
        :=
        \int_\Xi
        1_B\left(\mcT_{x;\omega}\proj\rho\right)\,
        \tr{\mcT_{x;\omega}(\rho)}\,
        \nu_\omega(\dee x).
    \]
    Then:
    \begin{enumerate}
        \item \(\Gamma_\omega\) is a probability kernel from
        \((\states,\borel{\states})\) to \((\Xi,\mcX)\);
        \item \(K_\omega\) is a probability kernel on \((\states,\borel{\states})\);
        \item for every \(A\in\mcX\) and \(B\in\borel{\states}\), the maps
        \[
            (\omega,\rho)\longmapsto \Gamma_\omega(\rho,A),
            \qquad
            (\omega,\rho)\longmapsto K_\omega(\rho,B)
        \]
        are \((\mcF\otimes\borel{\states},\borel{[0,1]})\)-measurable.
    \end{enumerate}
\end{prop}

\begin{proof}
    We first consider \(\Gamma\).
    Fix \(\omega\in\Omega\) and \(\rho\in\states\).
    Since \(x\mapsto \mcT_{x;\omega}\) is measurable,  \((\Phi,\eta)\mapsto \Phi(\eta)\) is continuous on \(\mcL(\matrices)\times\matrices\), and the trace is continuous on  \(\matrices\), the map \(x\mapsto \tr{\mcT_{x;\omega}(\rho)}\) is \(\mcX\)-measurable.
    It is nonnegative because each \(\mcT_{x;\omega}\) is completely positive.
    Hence \(A\mapsto \Gamma_\omega(\rho,A)\) is a measure on \((\Xi,\mcX)\).

    Since \(\mfI_\omega(\Xi)\) is trace-preserving,
    \[
        \Gamma_\omega(\rho,\Xi)
        =
        \tr{\mfI_\omega(\Xi)(\rho)}
        =
        \tr{\rho}
        =
        1.
    \]
    Thus \(\Gamma_\omega(\rho,\,\cdot\,)\) is a probability measure.

    Now fix \(A\in\mcX\). Define
    \[
        \widetilde\nu_{(\omega,\rho)}:=\nu_\omega,
        \qquad
        (\omega,\rho)\in\Omega\times\states,
    \]
    which is a measurable kernel from \((\Omega\times\states,\mcF\otimes\borel{\states})\) to \((\Xi,\mcX)\).
    Also define
    \[
        g_A(\omega,\rho,x)
        :=
        1_A(x)\,\tr{\mcT_{x;\omega}(\rho)},
        \qquad
        (\omega,\rho,x)\in\Omega\times\states\times\Xi.
    \]
    Since \((\omega,x,\rho)\mapsto \mcT_{x;\omega}(\rho)\) is measurable and the  trace is continuous, \(g_A\) is nonnegative and \((\mcF\otimes\borel{\states}\otimes\mcX)\)-measurable.
    Therefore, by the standard measurability theorem for integration against a measurable kernel,
    \[
        (\omega,\rho)\longmapsto
        \int_\Xi g_A(\omega,\rho,x)\,\widetilde\nu_{(\omega,\rho)}(\dee x)
    \]
    is \((\mcF\otimes\borel{\states})\)-measurable.
    Since \(\widetilde\nu_{(\omega,\rho)}=\nu_\omega\), this is exactly
    \[
        (\omega,\rho)\longmapsto \Gamma_\omega(\rho,A).
    \]
    This proves that \(\Gamma\) is a probability kernel and proves part (3) for
    \(\Gamma\).

    We next consider \(K\).
    Fix \(\omega\in\Omega\) and \(\rho\in\states\), and define
    \[
        F_{\omega,\rho}(x):=\mcT_{x;\omega}\proj\rho,
        \qquad x\in\Xi.
    \]
    By \Cref{prop:projective-action-is-measurable}, the map \(x\mapsto F_{\omega,\rho}(x)\) is measurable from \((\Xi,\mcX)\) to  \((\states,\borel{\states})\).
    Therefore, for every \(B\in\borel{\states}\),
    \[
        F_{\omega,\rho}^{-1}(B)\in\mcX.
    \]
    Using the definition of \(\Gamma_\omega\), we may write
    \[
        K_\omega(\rho,B)
        =
        \int_\Xi 1_B\left(F_{\omega,\rho}(x)\right)\,
        \Gamma_\omega(\rho,\dee x)
        =
        \Gamma_\omega\left(\rho,F_{\omega,\rho}^{-1}(B)\right).
    \]
    Thus \(K_\omega(\rho,\,\cdot\,)\) is the pushforward of the probability measure \(\Gamma_\omega(\rho,\,\cdot\,)\) under the measurable map \(F_{\omega,\rho}\).
    It follows that \(K_\omega(\rho,\,\cdot\,)\) is a probability measure on \((\states,\borel{\states})\).

    Now fix \(B\in\borel{\states}\), and define
    \[
        h_B(\omega,\rho,x)
        :=
        1_B\left(\mcT_{x;\omega}\proj\rho\right)\,
        \tr{\mcT_{x;\omega}(\rho)},
        \qquad
        (\omega,\rho,x)\in\Omega\times\states\times\Xi.
    \]
    By \Cref{prop:projective-action-is-measurable}, the map \((\omega,\rho,x)\mapsto \mcT_{x;\omega}\proj\rho\)  is measurable, and as above \((\omega,\rho,x)\mapsto \tr{\mcT_{x;\omega}(\rho)}\) is also measurable.
    Hence \(h_B\) is nonnegative and \((\mcF\otimes\borel{\states}\otimes\mcX)\)-measurable.
    Applying the kernel-integration theorem again, we obtain that
    \[
        (\omega,\rho)\longmapsto
        \int_\Xi h_B(\omega,\rho,x)\,\widetilde\nu_{(\omega,\rho)}(\dee x)
    \]
    is \((\mcF\otimes\borel{\states})\)-measurable.
    Since \(\widetilde\nu_{(\omega,\rho)}=\nu_\omega\), this is exactly \((\omega,\rho)\mapsto K_\omega(\rho,B)\).
    This proves that \(K\) is a probability kernel and proves part (3) for \(K\).
\end{proof}

For later use, we write
\[
    K_\omega^{(n)}
    :=
    K_{\theta^{n-1}\omega}\circ\cdots\circ K_\omega,
    \qquad n\ge1,
\]
with the convention
\[
    K_\omega^{(0)}(\rho,\,\cdot\,):=\delta_\rho.
\]

\paragraph{Null-set convention.}
    We keep fixed the instrument family and the current-instrument-measurable density realization chosen in \Cref{rem:current-instrument-measurable-realization}.
    The associated one-step probability kernels are defined for every \(\omega\in\Omega\).
    Whenever an almost-sure assertion holds on a measurable set \(\Omega_0\subseteq\Omega\) with \(\pr(\Omega_0)=1\), we may restrict it to
    \[
        \widetilde\Omega_0
        :=\bigcap_{n\in\mbZ}\theta^{-n}\Omega_0.
    \]
    Since \(\theta\) is invertible and measure-preserving, \(\widetilde\Omega_0\) is measurable, \(\theta\)-invariant, and has probability one.
    Passing to this set leaves the fixed instrument family, density realization, and one-step kernels unchanged.
    

\subsection{Construction of the quenched outcome law}
\label{subsection:construction-quenched-law}


We now construct the quenched law of the observed outcome record. 
We begin with the canonical cylinder formula determined directly by the instrument family.
After that, once the measurable density realization from \Cref{subsection:instrument-realization} has been fixed, we recover the same law recursively from the one-step outcome kernels induced by the posterior updates.
We conclude by verifying that the resulting quenched law depends measurably on the disorder realization.

\begin{prop}
\label{prop:instrument-cylinder-formula}
    Let \(\vartheta:\Omega\to\states\) be a measurable random initial state.
    For each fixed \(\omega\in\Omega\), each \(n\ge1\), and each \(A_1,\dots,A_n\in\mcX\), define
    \[
        \mu_{\vartheta;\omega}^{(n)}(A_1\times\cdots\times A_n)
        :=
        \tr{
            \left(
                \mfI_{\theta^{n-1}\omega}(A_n)
                \circ\cdots\circ
                \mfI_\omega(A_1)
            \right)\left(\vartheta(\omega)\right)
        }.
    \]
    Then these cylinder values are consistent, in the sense that for every \(n\ge1\) and every \(A_1,\dots,A_n\in\mcX\),
    \[
        \mu_{\vartheta;\omega}^{(n+1)}
        (A_1\times\cdots\times A_n\times\Xi)
        =
        \mu_{\vartheta;\omega}^{(n)}(A_1\times\cdots\times A_n).
    \]
\end{prop}
\begin{proof}
    Fix \(\omega\in\Omega\), \(n\ge1\), and \(A_1,\dots,A_n\in\mcX\), and set
    \[
        Y
        :=
        \left(
            \mfI_{\theta^{n-1}\omega}(A_n)
            \circ\cdots\circ
            \mfI_\omega(A_1)
        \right)\left(\vartheta(\omega)\right).
    \]
    Then
    \begin{align*}
        \mu_{\vartheta;\omega}^{(n+1)}
        (A_1\times\cdots\times A_n\times\Xi)
        =
        \tr{
            \left(
                \mfI_{\theta^n\omega}(\Xi)
                \circ
                \mfI_{\theta^{n-1}\omega}(A_n)
                \circ\cdots\circ
                \mfI_\omega(A_1)
            \right)\left(\vartheta(\omega)\right)
        }
        =
        \tr{\Phi_{\theta^n\omega}(Y)}.
    \end{align*}
    Since \(\Phi_{\theta^n\omega}=\mfI_{\theta^n\omega}(\Xi)\) is trace-preserving, \(\tr{\Phi_{\theta^n\omega}(Y)}=\tr{Y}\). 
    Therefore
    \[
        \mu_{\vartheta;\omega}^{(n+1)}
        (A_1\times\cdots\times A_n\times\Xi)
        =
        \mu_{\vartheta;\omega}^{(n)}(A_1\times\cdots\times A_n).
    \]
\end{proof}

We next describe the same quenched law recursively, in terms of the selective updates arising from the chosen measurable density realization.

\begin{prop}
\label{prop:recursive-posterior-measurable}
    Let \(\vartheta:\Omega\to\states\) be measurable, and fix the density representation from \Cref{prop:instrument-density-representation}.
    Use the convention \(\Xi^0:=\{\ast\}\), and define recursively
    \[
        \widehat\rho_0^{\omega,\vartheta}(\ast):=\vartheta(\omega),
    \]
    and, for \(n\ge0\),
    \[
        \widehat\rho_{n+1}^{\omega,\vartheta}(x_1,\dots,x_{n+1})
        :=
        \mcT_{x_{n+1};\theta^n\omega}\proj
        \widehat\rho_n^{\omega,\vartheta}(x_1,\dots,x_n).
    \]
    Then the following hold.
    \begin{enumerate}
        \item
        For every \(n\in\mbN_0\), the map
        \[
            (\omega,z)\longmapsto \widehat\rho_n^{\omega,\vartheta}(z)
        \]
        from \(\Omega\times\Xi^n\) to \(\states\) is
        \((\mcF\otimes\mcX^{\otimes n},\borel{\states})\)-measurable.
    
        \item
        Consequently, for every \(n\in\mbN_0\), the map
        \[
            \rho_n^{\omega,\vartheta}:\Xi^{\mbN}\to\states,
            \qquad
            \rho_n^{\omega,\vartheta}(\bar x)
            :=
            \widehat\rho_n^{\omega,\vartheta}(x_1,\dots,x_n),
        \]
        is measurable and depends only on the first \(n\) coordinates.
    
        \item
        For every \(n\in\mbN_0\) and every \(A\in\mcX\), the map
        \[
            (\omega,z)\longmapsto
            \kappa_{n,\vartheta}^\omega(z,A)
            :=
            \Gamma_{\theta^n\omega}\left(\widehat\rho_n^{\omega,\vartheta}(z),A\right)
        \]
        is \((\mcF\otimes\mcX^{\otimes n},\borel{[0,1]})\)-measurable.
    
        \item
        For every \(\omega\in\Omega\) and every \(n\in\mbN_0\), the map
        \[
            \kappa_{n,\vartheta}^\omega(z,A)
            :=
            \Gamma_{\theta^n\omega}\left(\widehat\rho_n^{\omega,\vartheta}(z),A\right),
            \qquad
            z\in\Xi^n,\ A\in\mcX,
        \]
        defines a probability kernel from \((\Xi^n,\mcX^{\otimes n})\) to
        \((\Xi,\mcX)\).
    \end{enumerate}
\end{prop}
\begin{proof}
    We prove the claims in order.
    For (1), we argue by induction on \(n\). For \(n=0\), the map \((\omega,\ast)\mapsto \widehat\rho_0^{\omega,\vartheta}(\ast) = \vartheta(\omega)\) is measurable by assumption.
    Assume now that \((\omega,z)\mapsto \widehat\rho_n^{\omega,\vartheta}(z)\) is measurable on \(\Omega\times\Xi^n\). 
    Consider the measurable map
    \[
        \Omega\times\Xi^{n+1}\to\Omega\times\Xi\times\states,
        \qquad
        (\omega,x_1,\dots,x_{n+1})
        \longmapsto
        \left(
            \theta^n\omega,\,
            x_{n+1},\,
            \widehat\rho_n^{\omega,\vartheta}(x_1,\dots,x_n)
        \right).
    \]
    By \Cref{prop:projective-action-is-measurable}, the map \((\eta,x,\rho)\mapsto \mcT_{x;\eta}\proj\rho\) is measurable from \(\Omega\times\Xi\times\states\) to \(\states\). 
    Hence
    \[
        (\omega,x_1,\dots,x_{n+1})
        \longmapsto
        \mcT_{x_{n+1};\theta^n\omega}\proj
        \widehat\rho_n^{\omega,\vartheta}(x_1,\dots,x_n)
        =
        \widehat\rho_{n+1}^{\omega,\vartheta}(x_1,\dots,x_{n+1})
    \]
    is measurable, completing the induction.

    Statement (2) follows by composing  \(\widehat\rho_n^{\omega,\vartheta}\) with the projection  \(\Xi^{\mbN}\to\Xi^n\), \((x_1,x_2,\dots)\mapsto(x_1,\dots,x_n)\).

    For (3), fix \(n\in\mbN_0\) and \(A\in\mcX\). 
    By
    \Cref{prop:Gamma-K-are-kernels}, the map \((\eta,\rho)\mapsto \Gamma_\eta(\rho,A)\) is measurable on \(\Omega\times\states\). 
    Since \((\omega,z)\mapsto
        \left(
            \theta^n\omega,\,
            \widehat\rho_n^{\omega,\vartheta}(z)
        \right)
    \)
    is measurable by (1), the composition
    \[
        (\omega,z)\longmapsto
        \Gamma_{\theta^n\omega}\left(\widehat\rho_n^{\omega,\vartheta}(z),A\right)
        =
        \kappa_{n,\vartheta}^\omega(z,A)
    \]
    is measurable.

    For (4), fix \(\omega\in\Omega\) and \(n\in\mbN_0\). For each \(z\in\Xi^n\), the set function
    \[
        A\longmapsto \kappa_{n,\vartheta}^\omega(z,A)
        =
        \Gamma_{\theta^n\omega}\left(\widehat\rho_n^{\omega,\vartheta}(z),A\right)
    \]
    is a probability measure on \((\Xi,\mcX)\) by  \Cref{prop:Gamma-K-are-kernels}. For each \(A\in\mcX\), the map \(z\mapsto \kappa_{n,\vartheta}^\omega(z,A)\) is measurable by restricting the jointly measurable map from (3). 
    Hence  \(\kappa_{n,\vartheta}^\omega\) is a probability kernel from \((\Xi^n,\mcX^{\otimes n})\) to \((\Xi,\mcX)\).
\end{proof}

The kernels \(\kappa_{n,\vartheta}^\omega\) furnish the recursive construction of the quenched law on outcome paths. 
The next proposition shows that this construction agrees with the canonical cylinder formula from \Cref{prop:instrument-cylinder-formula}.

\begin{prop}
\label{prop:construction-of-mbQ}
    Let \(\vartheta:\Omega\to\states\) be measurable, and fix the density representation from \Cref{prop:instrument-density-representation}.
    Then for each \(\omega\in\Omega\), there exists a unique probability measure
    \[
        \mbQ_{\vartheta;\omega}\in\mcP(\Xi^{\mbN})
    \]
    such that, if \(X_n:\Xi^{\mbN}\to\Xi\) denotes the \(n\)-th coordinate map, then for every \(n\in\mbN_0\) and every \(A\in\mcX\),
    \[
        \mbQ_{\vartheta;\omega}\left(X_{n+1}\in A\mid X_1,\dots,X_n\right)
        =
        \kappa_{n,\vartheta}^\omega(X_1,\dots,X_n;A)
        \qquad
        \mbQ_{\vartheta;\omega}\text{-a.s.}
    \]
    Moreover, for every \(n\ge1\) and every \(A_1,\dots,A_n\in\mcX\),
    \[
        \mbQ_{\vartheta;\omega}
        \left(
            A_1\times\cdots\times A_n\times \Xi\times\Xi\times\cdots
        \right)
        =
        \tr{
            \left(
                \mfI_{\theta^{n-1}\omega}(A_n)
                \circ\cdots\circ
                \mfI_\omega(A_1)
            \right)\left(\vartheta(\omega)\right)
        }.
    \]
    In particular, the finite-dimensional distributions of \(\mbQ_{\vartheta;\omega}\) are given by the canonical cylinder formula of \Cref{prop:instrument-cylinder-formula}.
\end{prop}

\begin{proof}
    Fix \(\omega\in\Omega\). By \Cref{prop:recursive-posterior-measurable}, for  every \(n\in\mbN_0\), the map \(\kappa_{n,\vartheta}^\omega\) is a probability kernel from \((\Xi^n,\mcX^{\otimes n})\) to \((\Xi,\mcX)\).
    Therefore, by the Ionescu--Tulcea theorem (see \cite[Chapter~V]{neveu1965mathematical}), there exists a unique  probability measure \(\mbQ_{\vartheta;\omega}\in\mcP(\Xi^{\mbN})\) such that, for every \(n\in\mbN_0\) and every \(A\in\mcX\),
    \[
        \mbQ_{\vartheta;\omega}\left(X_{n+1}\in A\mid X_1,\dots,X_n\right)
        =
        \kappa_{n,\vartheta}^\omega(X_1,\dots,X_n;A)
        \qquad
        \mbQ_{\vartheta;\omega}\text{-a.s.}
    \]
    Equivalently, for every \(n\ge1\) and every \(A_1,\dots,A_n\in\mcX\),
    \begin{align*}
        &\mbQ_{\vartheta;\omega}
        \left(
            A_1\times\cdots\times A_n\times\Xi\times\Xi\times\cdots
        \right)
        \\
        &\qquad=
        \int_{A_1}\kappa_{0,\vartheta}^\omega(\ast,\dee x_1)
        \int_{A_2}\kappa_{1,\vartheta}^\omega(x_1,\dee x_2)\cdots
        \int_{A_n}\kappa_{n-1,\vartheta}^\omega(x_1,\dots,x_{n-1};\dee x_n).
    \end{align*}

    To identify these marginals with the canonical cylinder formula, define
    recursively the unnormalized branch maps by \(\widehat\sigma_0^{\omega,\vartheta}:=\vartheta(\omega)\), and, for \(n\ge0\),
    \[
        \widehat\sigma_{n+1}^{\omega,\vartheta}(x_1,\dots,x_{n+1})
        :=
        \mcT_{x_{n+1};\theta^n\omega}
        \left(
            \widehat\sigma_n^{\omega,\vartheta}(x_1,\dots,x_n)
        \right).
    \]
    Then
    \[
        \widehat\sigma_n^{\omega,\vartheta}(x_1,\dots,x_n)
        =
        \left(
            \mcT_{x_n;\theta^{n-1}\omega}
            \circ\cdots\circ
            \mcT_{x_1;\omega}
        \right)\left(\vartheta(\omega)\right).
    \]

    A straightforward induction shows that
    \[
        \widehat\sigma_n^{\omega,\vartheta}(x_1,\dots,x_n)
        =
        \left(
            \prod_{j=1}^n
            \tr{
                \mcT_{x_j;\theta^{j-1}\omega}
                \left(
                    \widehat\rho_{j-1}^{\omega,\vartheta}(x_1,\dots,x_{j-1})
                \right)
            }
        \right)
        \widehat\rho_n^{\omega,\vartheta}(x_1,\dots,x_n),
    \]
    where the empty product is interpreted as \(1\). 
    Taking traces gives
    \[
        \prod_{j=1}^n
        \tr{
            \mcT_{x_j;\theta^{j-1}\omega}
            \left(
                \widehat\rho_{j-1}^{\omega,\vartheta}(x_1,\dots,x_{j-1})
            \right)
        }
        =
        \tr{
            \left(
                \mcT_{x_n;\theta^{n-1}\omega}
                \circ\cdots\circ
                \mcT_{x_1;\omega}
            \right)\left(\vartheta(\omega)\right)
        }.
    \]
    Substituting this identity into the Ionescu--Tulcea expansion yields
    \begin{align*}
        &\mbQ_{\vartheta;\omega}
        \left(
            A_1\times\cdots\times A_n\times\Xi\times\Xi\times\cdots
        \right)
        \\
        &\qquad=
        \int_{A_1}\cdots\int_{A_n}
        \tr{
            \left(
                \mcT_{x_n;\theta^{n-1}\omega}
                \circ\cdots\circ
                \mcT_{x_1;\omega}
            \right)\left(\vartheta(\omega)\right)
        }
        \,\nu_\omega(\dee x_1)\cdots \nu_{\theta^{n-1}\omega}(\dee x_n).
    \end{align*}
    Finally, repeated application of
    \[
        \mfI_\eta(A)(X)
        =
        \int_A \mcT_{x;\eta}(X)\,\nu_\eta(\dee x)
    \]
    together with linearity and Fubini's theorem in the finite-dimensional space \(\matrices\) shows that the right-hand side equals
    \[
        \tr{
            \left(
                \mfI_{\theta^{n-1}\omega}(A_n)
                \circ\cdots\circ
                \mfI_\omega(A_1)
            \right)\left(\vartheta(\omega)\right)
        }.
    \]
    This proves the claimed cylinder formula. 
    Uniqueness of  \(\mbQ_{\vartheta;\omega}\) is already part of the Ionescu--Tulcea construction.
\end{proof}

It remains to verify that the quenched law varies measurably with the disorder realization.

\begin{prop}
\label{prop:mbQ-kernel-measurable}
    Let \(\vartheta:\Omega\to\states\) be measurable.
    Then \(\omega\mapsto \mbQ_{\vartheta;\omega}\) is a measurable probability kernel from \((\Omega,\mcF)\) to \((\Xi^{\mbN},\mcX^{\otimes\mbN})\).
    Equivalently, since \((\Xi^{\mbN},\mcX^{\otimes\mbN})\) is standard Borel, the map \(\omega\mapsto \mbQ_{\vartheta;\omega}\) is measurable as a map into \(\left(\mcP(\Xi^{\mbN}),\borel{\mcP(\Xi^{\mbN})}\right)\).
\end{prop}

\begin{proof}
    We apply the parameterized Ionescu--Tulcea theorem \cite[Corollary~2, Section~V.1]{neveu1965mathematical} with \(E_0:=\Omega\), \(\mcF_0:=\mcF\), and, for \(t\ge1\),  \(E_t:=\Xi\), \(\mcF_t:=\mcX\).
    For \(t=0\), define the transition kernel from \(E_0\) to \(E_1\) by
    \[
        P_1^{0}(\omega,A):=\kappa_{0,\vartheta}^\omega(\ast,A)
        =\Gamma_\omega(\vartheta(\omega),A).
    \]
    For \(t\ge1\), define the transition kernel from \(E_0\times\cdots\times E_t=\Omega\times\Xi^t\) to \(E_{t+1}=\Xi\) by
    \[
        P_{t+1}^{0,\dots,t}(\omega,x_1,\dots,x_t;A)
        :=
        \kappa_{t,\vartheta}^\omega\left((x_1,\dots,x_t),A\right).
    \]
    By \Cref{prop:recursive-posterior-measurable}, these kernels are jointly  measurable in \((\omega,x_1,\dots,x_t)\). 
    Therefore, the parameterized  Ionescu--Tulcea theorem yields a probability kernel \(\omega\mapsto \mbQ_{\vartheta;\omega}\) from \((\Omega,\mcF)\) to \((\Xi^{\mbN},\mcX^{\otimes\mbN})\), whose finite-dimensional distributions are determined by the kernels \(P_1^0,P_2^{0,1},\dots\). 
    Moreover, it yields the measurability of
    \[
        \omega\longmapsto \int_{\Xi^{\mbN}} Y(\bar x)\,\mbQ_{\vartheta;\omega}(\dee\bar x)
    \]
    for every positive \(\mcX^{\otimes\mbN}\)-measurable function \(Y\) on \(\Xi^{\mbN}\). Taking \(Y=1_E\) shows that, for every \(E\in\mcX^{\otimes\mbN}\), the map \(\omega\mapsto \mbQ_{\vartheta;\omega}(E)\) is measurable. 
    Thus \(\omega\mapsto\mbQ_{\vartheta;\omega}\) is a measurable probability kernel.
\end{proof}


\subsection{Canonical coding by the instrument process}
\label{subsection:canonical-coding-instrument-process}


To compare our formulation with the canonical shift-space framework of \cite{HW25,YH26}, we now code the environment by the bi-infinite sequence of abstract one-step instruments. Since the measurable density realization from \Cref{prop:instrument-density-representation} is auxiliary, the coding is performed at the level of the instrument family \(\omega\mapsto \mfI_\omega\).

Because \((\Xi,\mcX)\) is standard Borel, the \(\sigma\)-algebra \(\mcX\) is countably generated. Fix once and for all a countable algebra
\[
    \mcC=\{C_m:m\in\mbN\}\subseteq\mcX
\]
such that \(\sigma(\mcC)=\mcX\). 
We use this fixed generating algebra to encode each one-step instrument by the countable family of its Choi matrices on \(\mcC\).
Define
\[
    \mathsf{Instr}
    :=
    \left(\mbM_{d^2}(\mbC)\right)^{\mbN},
    \qquad
    \mcJ
    :=
    \borel{\mbM_{d^2}(\mbC)}^{\otimes\mbN}.
\]
For each \(\omega\in\Omega\), define the current-instrument code
\[
    \mathbf I_\omega
    :=
    \left(
        J(\mfI_\omega(C_m))
    \right)_{m\in\mbN}
    \in \mathsf{Instr}.
\]
By assumption, each coordinate of the map \(\omega\mapsto \mathbf I_\omega\) is measurable; hence
\(\omega\longmapsto \mathbf I_\omega\) is \((\mcF,\mcJ)\)-measurable.

Moreover, \(\mathbf I_\omega\) determines the full instrument \(\mfI_\omega\). 
Indeed, if \(\mathbf I_\omega=\mathbf I_{\omega'}\), then
\[
    J(\mfI_\omega(C_m))
    =
    J(\mfI_{\omega'}(C_m))
    \qquad
    \text{for every }m\in\mbN.
\]
Since \(J\) is injective, the two \(\mcL(\matrices)\)-valued measures \(A\mapsto \mfI_\omega(A)\) and \(A\mapsto \mfI_{\omega'}(A)\) agree on the generating algebra \(\mcC\). 
By the uniqueness theorem for finite complex measures applied entrywise, they therefore agree on all of \(\mcX\).

For \(k\in\mbZ\), define the current-instrument \(\sigma\)-algebra by
\[
    \mcG_k^{\mathrm{inst}}
    :=
    \sigma\!\left(
        \omega\longmapsto \mathbf I_{\theta^k\omega}
    \right).
\]
Equivalently,
\[
    \mcG_k^{\mathrm{inst}}
    =
    \sigma\left(
        \omega\longmapsto J\left(\mfI_{\theta^k\omega}(A)\right)
        : A\in\mcX
    \right),
\]
and it is enough to use the generating algebra \(\mcC\) in place of \(\mcX\).

Define also the past and future \(\sigma\)-algebras generated by the instrument
process:
\[
    \mcF_k^{\mathrm{inst},-}
    :=
    \sigma\left(\mcG_j^{\mathrm{inst}}:j\le k\right),
    \qquad
    \mcF_k^{\mathrm{inst},+}
    :=
    \sigma\left(\mcG_j^{\mathrm{inst}}:j\ge k\right),
    \qquad
    k\in\mbZ,
\]
and let
\[
    \mcG^{\mathrm{inst}}
    :=
    \sigma\left(\mcG_j^{\mathrm{inst}}:j\in\mbZ\right)
\]
be the $\sigma$-algebra generated by the entire instrument process.
\(\phi\)-mixing profile is
\[
    \phi^{\mathrm{inst}}(n)
    :=
    \sup_{k\in\mbZ}
    \phi_{\pr}\!\left(
        \mcF_k^{\mathrm{inst},-},
        \mcF_{k+n}^{\mathrm{inst},+}
    \right),
    \qquad
    n\ge1.
\]

This coding allows us to pass from the original base system to a canonical bi-infinite shift space carrying exactly the same instrument process.
Set
\[
    \widehat\Omega
    :=
    \mathsf{Instr}^{\mbZ},
    \qquad
    \widehat{\mcF}
    :=
    \mcJ^{\otimes\mbZ},
\]
and let
\[
    \widehat\theta:\widehat\Omega\to\widehat\Omega
\]
be the left shift,
\[
    \widehat\theta\left((\hat\omega_n)_{n\in\mbZ}\right)
    :=
    (\hat\omega_{n+1})_{n\in\mbZ}.
\]
Write
\[
    \pi_k:\widehat\Omega\to\mathsf{Instr},
    \qquad
    \pi_k\left((\hat\omega_n)_{n\in\mbZ}\right):=\hat\omega_k,
    \qquad
    k\in\mbZ,
\]
for the coordinate maps. 
Define the coding map
\[
    \iota:\Omega\to\widehat\Omega,
    \qquad
    \iota(\omega)
    :=
    \left(\mathbf I_{\theta^n\omega}\right)_{n\in\mbZ},
\]
and set
\[
    \widehat\pr
    :=
    \pr\circ\iota^{-1}.
\]

\begin{prop}
\label{prop:instrument-coding-factor}
    The coding map \(\iota\) is \((\mcF,\widehat{\mcF})\)-measurable and intertwines the dynamics in the sense that \(\iota\circ\theta = \widehat\theta\circ\iota\).
    Moreover, for every \(k\in\mbZ\), we have that  \(\iota^{-1}\left(\sigma(\pi_k)\right) = \mcG_k^{\mathrm{inst}}\),
    and
    \[
        \iota^{-1}\!\left(
            \sigma(\pi_j:j\le k)
        \right)
        =
        \mcF_k^{\mathrm{inst},-},
        \qquad
        \iota^{-1}\!\left(
            \sigma(\pi_j:j\ge k)
        \right)
        =
        \mcF_k^{\mathrm{inst},+}.
    \]
    In particular, \(\iota^{-1}(\widehat{\mcF})= \mcG^{\mathrm{inst}}\).
    If
    \[
        \widehat\phi(n)
        :=
        \sup_{k\in\mbZ}
        \phi_{\widehat\pr}\!\left(
            \sigma(\pi_j:j\le k),
            \sigma(\pi_j:j\ge k+n)
        \right),
        \qquad
        n\ge1,
    \]
    denotes the \(\phi\)-mixing profile of the canonical coordinate process on \((\widehat\Omega,\widehat{\mcF},\widehat\pr)\), then
    \[
        \widehat\phi(n)=\phi^{\mathrm{inst}}(n)
        \qquad
        \text{for every }n\ge1.
    \]
\end{prop}
\begin{proof}
    Since each coordinate map \(\omega\mapsto \mathbf I_{\theta^n\omega}\) is \((\mcF,\mcJ)\)-measurable, the map \(\iota\) is \((\mcF,\widehat{\mcF})\)-measurable. 
    The identity \(\iota\circ\theta=\widehat\theta\circ\iota\) is immediate from the definition.
    Next,
    \[
        \pi_k(\iota(\omega))
        =
        \mbI_{\theta^k\omega},
    \]
    hence
    \[
        \iota^{-1}\left(\sigma(\pi_k)\right)
        =
        \sigma(\mathbf I_{\theta^k (\,\cdot\,)})
        =
        \mcG_k^{\mathrm{inst}}.
    \]
    Taking joins over \(j\le k\) and \(j\ge k\) gives the identities for \(\mcF_k^{\mathrm{inst},-}\) and \(\mcF_k^{\mathrm{inst},+}\), and the identity \(\iota^{-1}(\widehat{\mcF})=\mcG^{\mathrm{inst}}\) follows by taking the join over all \(k\in\mbZ\).
    Finally, because \(\widehat\pr=\pr\circ\iota^{-1}\), the pullback identities  above imply that for every \(k\in\mbZ\),
    \[
        \phi_{\widehat\pr}\!\left(
            \sigma(\pi_j:j\le k),
            \sigma(\pi_j:j\ge k+n)
        \right)
        =
        \phi_{\pr}\!\left(
            \mcF_k^{\mathrm{inst},-},
            \mcF_{k+n}^{\mathrm{inst},+}
        \right).
    \]
    Taking the supremum over \(k\) yields \(\widehat\phi(n)=\phi^{\mathrm{inst}}(n)\).
\end{proof}

We next note that $\widehat\theta$ is an ergodic $\widehat\pr$-preserving transformation whenever $\theta$ is an ergodic $\pr$-preserving transformation. 

\begin{prop}
\label{prop:coded-shift-ergodic}
    Assume that \(\theta\) is measure-preserving on \((\Omega,\mcF,\pr)\). 
    If \((\Omega,\mcF,\pr,\theta)\) is ergodic, then \((\widehat\Omega,\widehat{\mcF},\widehat\pr,\widehat\theta)\) is ergodic. 
    More generally, it is enough that \(\theta\) be ergodic  on the factor \((\Omega,\mcG^{\mathrm{inst}},\pr)\).
\end{prop}
\begin{proof}
    Let \(A\in\widehat{\mcF}\) be an essentially $\widehat\theta$-invariant set, i.e., \(\widehat\theta^{-1}A=A\) \(\widehat\pr\)-a.s.
    Then, since \(\iota\circ\theta=\widehat\theta\circ\iota\), we have
    \[
        \theta^{-1}\iota^{-1}(A)
        =
        \iota^{-1}(\widehat\theta^{-1}A)
        =
        \iota^{-1}(A)
        \qquad
        \pr\text{-a.s.}
    \]
    Moreover, by \Cref{prop:instrument-coding-factor}, \(\iota^{-1}(A)\in\iota^{-1}(\widehat{\mcF})  =\mcG^{\mathrm{inst}}\subseteq\mcF\).
    If \(\theta\) is ergodic on \((\Omega,\mcF,\pr)\), or just on \((\Omega,\mcG^{\mathrm{inst}},\pr)\), it follows that
    \[
        \pr(\iota^{-1}(A))\in\{0,1\}.
    \]
    Since \(\widehat\pr=\pr\circ\iota^{-1}\),
    \[
        \widehat\pr(A)=\pr(\iota^{-1}(A))\in\{0,1\}.
    \]
    Thus \(\widehat\theta\) is ergodic.
\end{proof}

We now formalize the notion of dependence on the current one-step instrument used in \Cref{section:intro}.

\begin{dfn}
\label{def:depends-only-current-instrument}
    \begin{enumerate}
        \item
            Let \(f:\Omega\times\states\to\mbR\) be bounded and measurable. 
            We say that \(f\) \emph{depends only on the current one-step instrument} if \(f\) is
            \(\mcG_0^{\mathrm{inst}}\otimes\borel{\states}
            \)-measurable.

        \item
            Let \((E,\mcE)\) be a standard Borel space, and let  \(\omega\longmapsto L_\omega\) be a measurable family of probability kernels from \((\states,\borel{\states})\) to \((E,\mcE)\).
            We say that \(\omega\mapsto L_\omega\) \emph{depends only on the current one-step instrument} if, for every \(\rho\in\states\) and every \(C\in\mcE\), the map
            \((\omega,\rho)\mapsto L_\omega(\rho,C)\)  is \(\mcG_0^{\mathrm{inst}}\otimes\borel{\states}\)-measurable.
    
        \item
            Let \(\omega\longmapsto \lambda_\omega\in\mcP(\states)\) be a measurable family of probability measures. 
            We say that \(\omega\mapsto\lambda_\omega\) \emph{depends only on the current one-step  instrument} if the map \(\omega\mapsto \lambda_\omega\) is \((\mcG_0^{\mathrm{inst}},\borel{\mcP(\states)})\)-measurable.
    \end{enumerate}
\end{dfn}

\begin{remark}
\label{rem:current-instrument-dependence}
    The definition above is made at the level of the abstract instrument family \(\omega\mapsto\mfI_\omega\).
    In particular, the statement that a quantity depends only on the current one-step instrument means precisely that it is measurable with respect to the factor generated by the current instrument code \(\mathbf I_\omega\).
    Because the outcome kernel \(\Gamma_\omega\) is canonically determined by \(\mfI_\omega\), its current-instrument dependence is automatic.
    The same is true for the posterior-state kernel \(K_\omega\): although the density \(x\mapsto \mcT_{x;\omega}\) is fixed only up to \(\nu_\omega\)-almost everywhere equality, the integral formula defining \(K_\omega\) is independent of that choice, and hence \(K_\omega\) is canonically determined by \(\mfI_\omega\).
    By contrast, the density representative \((\nu_\omega,\mcT_{x;\omega})\) itself is auxiliary and not canonical as a pointwise object.
\end{remark}

\begin{remark}
    In the finite-outcome case \(\Xi=\mcA\) and \(\mcX=2^{\mcA}\), one may take  \(\mcC=2^{\mcA}\). 
    Then the current instrument code \(\mathbf I_\omega\) is equivalent to the finite family of singleton branches
    \[
        \left(
            \mfI_\omega(\{a\})
        \right)_{a\in\mcA},
    \]
    or, after fixing a Kraus representation, to the usual family of one-step branch maps.
\end{remark}

\begin{prop}
\label{prop:descent-to-coded-shift}
    Assume that \Cref{assumption_doeblin} holds, with constants \(L\in\mbN\), \(\varepsilon>0\), and minorizing family \(\omega\mapsto\lambda_\omega\in\mcP(\states)\). 
    Assume moreover that the minorizing family \(\omega\mapsto\lambda_\omega\) and the bounded measurable function \(f:\Omega\times\states\to\mbR\) depend only on the current one-step instrument in the sense of \Cref{def:depends-only-current-instrument}.
    Let \(K_\omega\) be the canonical posterior-state kernel associated with the current instrument \(\mfI_\omega\). 
    Then there exist a measurable probability kernel 
    \[
         \widehat K:\widehat\Omega\times\states\times\borel{\states}\to[0,1],
    \]
    a bounded measurable function \(\widehat f:\widehat\Omega\times\states\to\mbR\), and a measurable family \(\widehat\lambda:\widehat\Omega\to\mcP(\states)\),  such that for \(\pr\)-almost every \(\omega\),
    \[
        K_\omega(\rho,B)
        =
        \widehat K_{\iota(\omega)}(\rho,B),
        \qquad
        \lambda_\omega
        =
        \widehat\lambda_{\iota(\omega)},
        \qquad
        f(\omega,\rho)
        =
        \widehat f(\iota(\omega),\rho),
    \]
    for all \(\rho\in\states\) and \(B\in\borel{\states}\).
    Moreover,
    \[
        \widehat K_{\hat\omega}^{(L)}(\rho,B)
        \ge
        \varepsilon\,\widehat\lambda_{\widehat\theta^L\hat\omega}(B)
    \]
    for \(\widehat\pr\)-almost every \(\hat\omega\in\widehat\Omega\),
    all \(\rho\in\states\), and all \(B\in\borel{\states}\).
\end{prop}
\begin{proof}
    Let
    \[
        \mathbf I:\Omega\to\mathsf{Instr},
        \qquad
        \mathbf I(\omega):=\mathbf I_\omega.
    \]
    By \Cref{prop:instrument-coding-factor},
    \[
        \sigma(\mathbf I)=\mcG_0^{\mathrm{inst}},
        \qquad
        \pi_0\circ\iota=\mathbf I.
    \]
    By \Cref{rem:current-instrument-dependence}, the canonical posterior-state kernel \(K_\omega\) depends only on the current one-step instrument.
    Since, by assumption, \(f\) and \(\omega\mapsto\lambda_\omega\) depend only on the current one-step instrument, by the Doob--Dynkin factorization lemma for measurable maps (\cite[Lemma~1.13]{Kallenberg_2021}) and the universal disintegration theorem for kernels (\cite[Corollary 1.26]{Kallenberg_2017}) on standard Borel spaces, there exist 
    \[
        \overline K:\mathsf{Instr}\times\states\times\borel{\states}\to[0,1],
        \qquad
        \overline\lambda:\mathsf{Instr}\to\mcP(\states),
        \qquad
        \overline f:\mathsf{Instr}\times\states\to\mbR,
    \]
    such that for \(\pr\)-almost every \(\omega\),
    \[
        K_\omega(\rho,B)=\overline K_{\mathbf I_\omega}(\rho,B),
        \qquad
        \lambda_\omega=\overline\lambda_{\mathbf I_\omega},
        \qquad
        f(\omega,\rho)=\overline f(\mathbf I_\omega,\rho),
    \]
    for all \(\rho\in\states\) and \(B\in\borel{\states}\).

    Now define
    \[
        \widehat K_{\hat\omega}:=\overline K_{\pi_0(\hat\omega)},
        \qquad
        \widehat\lambda_{\hat\omega}:=\overline\lambda_{\pi_0(\hat\omega)},
        \qquad
        \widehat f(\hat\omega,\rho):=\overline f(\pi_0(\hat\omega),\rho).
    \]
    Then, since \(\pi_0(\iota(\omega))=\mathbf I_\omega\), the required pullback identities follow.
    Finally, for \(\pr\)-almost every \(\omega\),
    \[
        \widehat K_{\iota(\omega)}^{(L)}(\rho,B)
        =
        K_\omega^{(L)}(\rho,B)
        \ge
        \varepsilon\,\lambda_{\theta^L\omega}(B)
        =
        \varepsilon\,\widehat\lambda_{\iota(\theta^L\omega)}(B)
        =
        \varepsilon\,\widehat\lambda_{\widehat\theta^L\iota(\omega)}(B).
    \]
    Since \(\widehat\pr=\pr\circ\iota^{-1}\), this implies
    \[
        \widehat K_{\hat\omega}^{(L)}(\rho,B)
        \ge
        \varepsilon\,\widehat\lambda_{\widehat\theta^L\hat\omega}(B)
    \]
    for \(\widehat\pr\)-almost every \(\hat\omega\in\widehat\Omega\), all \(\rho\in\states\), and all
    \(B\in\borel{\states}\).
\end{proof}

We will later use this coding to formulate the additional mixing assumptions and prove the Berry–Esseen estimate with deterministic normalization. 


\subsection{Posterior states: outcome-path and state-path viewpoints}
\label{subsection:posterior-state-viewpoints}


Having constructed the quenched law on the outcome records, we now identify the induced posterior-state process. 
One may view it either as a measurable function of the observed outcome path or as a time-inhomogeneous Markov chain on state-path space. 
The next three results make this correspondence precise and identify the two descriptions when the initial condition is a point mass.

\begin{prop}
\label{prop:posterior-markov-under-mbQ}
    Let \(\vartheta:\Omega\to\states\) be measurable, and fix \(\omega\in\Omega\).
    Then under \(\mbQ_{\vartheta;\omega}\), the posterior-state process \(\left(\rho_n^{\omega,\vartheta}\right)_{n\ge0} \)
    is a time-inhomogeneous Markov chain on \(\states\) with initial state
    \[
        \rho_0^{\omega,\vartheta}=\vartheta(\omega)
        \qquad
        \mbQ_{\vartheta;\omega}\text{-a.s.},
    \]
    and transition kernels
    \[
        K_\omega,\ K_{\theta\omega},\ K_{\theta^2\omega},\ \dots.
    \]
    That is, for every \(n\in\mbN_0\) and every \(B\in\borel{\states}\),
    \[
        \mbQ_{\vartheta;\omega}\!\left(
            \rho_{n+1}^{\omega,\vartheta}\in B
            \,\middle|\,
            \rho_0^{\omega,\vartheta},\dots,\rho_n^{\omega,\vartheta}
        \right)
        =
        K_{\theta^n\omega}\left(\rho_n^{\omega,\vartheta},B\right)
        \qquad
        \mbQ_{\vartheta;\omega}\text{-a.s.}
    \]
\end{prop}

\begin{proof}
    Let \(X_n:\Xi^{\mbN}\to\Xi\) denote the coordinate process on the outcome path space, and set
    \[
        \mcG_n:=\sigma(X_1,\dots,X_n),
        \qquad
        \mcH_n:=\sigma(\rho_0^{\omega,\vartheta},\dots,\rho_n^{\omega,\vartheta}).
    \]
    By construction, \(\rho_n^{\omega,\vartheta}\) is \(\mcG_n\)-measurable, and therefore \(\mcH_n\subseteq\mcG_n\).

    The initial condition is immediate from the definition:
    \[
        \rho_0^{\omega,\vartheta}=\vartheta(\omega)
        \qquad
        \mbQ_{\vartheta;\omega}\text{-a.s.}
    \]

    Now fix \(n\in\mbN_0\) and \(B\in\borel{\states}\). Since
    \[
        \rho_{n+1}^{\omega,\vartheta}
        =
        \mcT_{X_{n+1};\theta^n\omega}\proj \rho_n^{\omega,\vartheta},
    \]
    we have
    \[
        1_{\{\rho_{n+1}^{\omega,\vartheta}\in B\}}
        =
        1_B\!\left(\mcT_{X_{n+1};\theta^n\omega}\proj \rho_n^{\omega,\vartheta}\right).
    \]
    Using the recursive description of \(\mbQ_{\vartheta;\omega}\) from the Ionescu--Tulcea construction, we obtain
    \begin{equation*}
        \mbE_{\mbQ_{\vartheta;\omega}}\!\left[
            1_{\{\rho_{n+1}^{\omega,\vartheta}\in B\}}
            \,\middle|\,
            \mcG_n
        \right]
        =
        \int_\Xi
        1_B\!\left(\mcT_{x;\theta^n\omega}\proj \rho_n^{\omega,\vartheta}\right)\,
        \kappa_{n,\vartheta}^\omega(X_1,\dots,X_n;\dee x).
    \end{equation*}
    By the definition of \(\kappa_{n,\vartheta}^\omega\),
    \[
        \kappa_{n,\vartheta}^\omega(X_1,\dots,X_n;\dee x)
        =
        \Gamma_{\theta^n\omega}\left(\rho_n^{\omega,\vartheta},\dee x\right).
    \]
    Hence, the right-hand side equals
    \[
        \int_\Xi
        1_B\!\left(\mcT_{x;\theta^n\omega}\proj \rho_n^{\omega,\vartheta}\right)\,
        \Gamma_{\theta^n\omega}\left(\rho_n^{\omega,\vartheta},\dee x\right)
        =
        K_{\theta^n\omega}\!\left(\rho_n^{\omega,\vartheta},B\right).
    \]
    Therefore
    \[
        \mbE_{\mbQ_{\vartheta;\omega}}\!\left[
            1_{\{\rho_{n+1}^{\omega,\vartheta}\in B\}}
            \,\middle|\,
            \mcG_n
        \right]
        =
        K_{\theta^n\omega}\!\left(\rho_n^{\omega,\vartheta},B\right).
    \]

    Since \(\rho_n^{\omega,\vartheta}\) is \(\mcH_n\)-measurable and \(\rho\mapsto K_{\theta^n\omega}(\rho,B)\) is measurable, the random variable \(K_{\theta^n\omega}\!\left(\rho_n^{\omega,\vartheta},B\right)\) is \(\mcH_n\)-measurable. 
    Taking conditional expectation with respect to \(\mcH_n\) and using the tower property, we obtain
    \begin{align*}
        \mbE_{\mbQ_{\vartheta;\omega}}\!\left[
            1_{\{\rho_{n+1}^{\omega,\vartheta}\in B\}}
            \,\middle|\,
            \mcH_n
        \right]
        &\qquad=
        \mbE_{\mbQ_{\vartheta;\omega}}\!\left[
            \mbE_{\mbQ_{\vartheta;\omega}}\!\left[
                1_{\{\rho_{n+1}^{\omega,\vartheta}\in B\}}
                \,\middle|\,
                \mcG_n
            \right]
            \,\middle|\,
            \mcH_n
        \right]
        \\
        &\qquad=
        \mbE_{\mbQ_{\vartheta;\omega}}\!\left[
            K_{\theta^n\omega}\!\left(\rho_n^{\omega,\vartheta},B\right)
            \,\middle|\,
            \mcH_n
        \right]
        \\
        &\qquad=
        K_{\theta^n\omega}\!\left(\rho_n^{\omega,\vartheta},B\right).
    \end{align*}
    This is exactly the Markov property with transition kernel
    \(K_{\theta^n\omega}\).
\end{proof}

We now pass to the corresponding law on posterior-state path space.
Let
\[
    e_n:\states^{\mbN_0}\to\states,
    \qquad
    e_n((\rho_0,\rho_1,\rho_2,\dots)):=\rho_n,
    \qquad n\in\mbN_0,
\]
denote the coordinate maps, and let
\[
    \mcE_n:=\sigma(e_0,\dots,e_n),
    \qquad n\in\mbN_0,
\]
be the natural filtration on the state-path space
\[
    (\states^{\mbN_0},\borel{\states}^{\otimes\mbN_0}).
\]
The next proposition constructs this state-path law directly, for arbitrary measurable initial distributions on the state space.

\begin{prop}
\label{prop:construction-of-bbK}
    Let \(\omega\mapsto \zeta_\omega\in\mcP(\states)\)  be a measurable family of initial laws.
    Then there exists a unique measurable probability kernel \(\omega\longmapsto \mbK_{\zeta;\omega}\) from \((\Omega,\mcF)\) to \((\states^{\mbN_0},\borel{\states}^{\otimes\mbN_0})\)  such that the coordinate map \(e_0\) has law \(\zeta_\omega\) and, for every \(n\ge0\),
    \[
        \mbK_{\zeta;\omega}\!\left(
            e_{n+1}\in B \,\middle|\, e_0,\dots,e_n
        \right)
        =
        K_{\theta^n\omega}(e_n,B)
        \qquad
        \mbK_{\zeta;\omega}\text{-a.s.}
    \]
    for every \(B\in\borel{\states}\).
    Equivalently, for every \(n\ge0\) and every
    \(B_0,\dots,B_n\in\borel{\states}\),
    \begin{align*}
        &\mbK_{\zeta;\omega}
        \left(
            B_0\times\cdots\times B_n\times
            \states\times\states\times\cdots
        \right)
        \\
        &\qquad=
        \int_{B_0}\zeta_\omega(\dee\rho_0)
        \int_{B_1}K_\omega(\rho_0,\dee\rho_1)\cdots
        \int_{B_n}K_{\theta^{n-1}\omega}(\rho_{n-1},\dee\rho_n).
    \end{align*}
\end{prop}
\begin{proof}
    Since \(\states\) is a Borel subset of the finite-dimensional space \(\matrices\), it is a standard Borel space. 
    Hence \((\states^{\mbN_0},\borel{\states}^{\otimes\mbN_0})\) is also a standard Borel space.
    Define a kernel from \((\Omega,\mcF)\) to \((\states,\borel{\states})\) by \(P_0(\omega,B):=\zeta_\omega(B)\).
    For each \(n\ge0\), define a kernel from
    \[
        \left(\Omega\times\states^{n+1},
        \mcF\otimes\borel{\states}^{\otimes(n+1)}\right)
    \]
    to \((\states,\borel{\states})\) by
    \[
        P_{n+1}(\omega,\rho_0,\dots,\rho_n;B)
        :=
        K_{\theta^n\omega}(\rho_n,B),
        \qquad
        B\in\borel{\states}.
    \]
    By assumption, \(\omega\mapsto \zeta_\omega\) is measurable, so \(P_0\) is a measurable kernel. By \Cref{prop:Gamma-K-are-kernels}, for every \(B\in\borel{\states}\), the map \((\omega,\rho)\mapsto K_\omega(\rho,B)\) is \((\mcF\otimes\borel{\states})\)-measurable. Since \(\theta\) is measurable, it follows that for every \(n\ge0\),
    \[
        (\omega,\rho_0,\dots,\rho_n)
        \longmapsto
        K_{\theta^n\omega}(\rho_n,B)
    \]
    is \((\mcF\otimes\borel{\states}^{\otimes(n+1)})\)-measurable. 
    Thus each \(P_{n+1}\) is a measurable probability kernel.

    We now apply the parameterized Ionescu--Tulcea theorem with
    \[
        E_0:=\Omega,\qquad E_n:=\states\ \text{ for }n\ge1,
    \]
    and transition kernels
    \[
        P_0,\ P_1,\ P_2,\dots.
    \]
    This yields a unique measurable probability kernel \(\omega\mapsto \mbK_{\zeta;\omega}\) from \((\Omega,\mcF)\) to  \((\states^{\mbN_0},\borel{\states}^{\otimes\mbN_0})\)  whose finite-dimensional marginals are given by
    \begin{align*}
        &\mbK_{\zeta;\omega}
        \left(
            B_0\times\cdots\times B_n\times
            \states\times\states\times\cdots
        \right)
        \\
        &\qquad=
        \int_{B_0}P_0(\omega,\dee\rho_0)
        \int_{B_1}P_1(\omega,\rho_0;\dee\rho_1)\cdots
        \int_{B_n}P_n(\omega,\rho_0,\dots,\rho_{n-1};\dee\rho_n),
    \end{align*}
    which is exactly the stated formula.

    The conditional-distribution identity
    \[
        \mbK_{\zeta;\omega}\!\left(
            e_{n+1}\in B \,\middle|\, e_0,\dots,e_n
        \right)
        =
        K_{\theta^n\omega}(e_n,B)
        \qquad
        \mbK_{\zeta;\omega}\text{-a.s.}
    \]
    is the corresponding Markov property furnished by the Ionescu--Tulcea construction. Uniqueness is part of the same theorem.
\end{proof}

For point-mass initial data, the state-path law just constructed coincides with the pushforward of the quenched outcome law under the posterior-state map.

\begin{prop}
\label{prop:bbK-is-pushforward-of-mbQ}
    Let \(\vartheta:\Omega\to\states\) be a measurable initial state.
    Then, for \(\pr\)-almost every \(\omega\),
    \[
        \mbK_{\vartheta;\omega}
        =
        \mbQ_{\vartheta;\omega}
        \circ
        \left(
            \rho_0^{\omega,\vartheta},
            \rho_1^{\omega,\vartheta},
            \rho_2^{\omega,\vartheta},
            \dots
        \right)^{-1}.
    \]
\end{prop}
\begin{proof}
    Fix \(\omega\) in the full-\(\pr\)-measure set on which  \Cref{prop:posterior-markov-under-mbQ} holds, and define a probability measure on \((\states^{\mbN_0},\borel{\states}^{\otimes\mbN_0})\) by
    \[
        \widetilde{\mbK}_{\vartheta;\omega}
        :=
        \mbQ_{\vartheta;\omega}
        \circ
        \left(
            \rho_0^{\omega,\vartheta},
            \rho_1^{\omega,\vartheta},
            \rho_2^{\omega,\vartheta},
            \dots
        \right)^{-1}.
    \]
    Since each map
    \[
        \rho_n^{\omega,\vartheta}:\Xi^{\mbN}\to\states
    \]
    is measurable by \Cref{prop:recursive-posterior-measurable}, this  pushforward is well defined.

    By construction,
    \[
        e_n\circ
        \left(
            \rho_0^{\omega,\vartheta},
            \rho_1^{\omega,\vartheta},
            \rho_2^{\omega,\vartheta},
            \dots
        \right)
        =
        \rho_n^{\omega,\vartheta},
        \qquad n\in\mbN_0.
    \]
    Hence the coordinate process \((e_n)_{n\ge0}\) under  \(\widetilde{\mbK}_{\vartheta;\omega}\) has the same law as the  posterior-state process \(\left(\rho_n^{\omega,\vartheta}\right)_{n\ge0}\) under \(\mbQ_{\vartheta;\omega}\).

    By \Cref{prop:posterior-markov-under-mbQ}, under \(\mbQ_{\vartheta;\omega}\)  the process \(\left(\rho_n^{\omega,\vartheta}\right)_{n\ge0}\) is a time-inhomogeneous Markov chain on \(\states\) with initial state
    \[
        \rho_0^{\omega,\vartheta}=\vartheta(\omega)
        \qquad
        \mbQ_{\vartheta;\omega}\text{-a.s.},
    \]
    and transition kernels
    \[
        K_\omega,\ K_{\theta\omega},\ K_{\theta^2\omega},\dots.
    \]
    Therefore, under \(\widetilde{\mbK}_{\vartheta;\omega}\), the  coordinate process \((e_n)_{n\ge0}\) satisfies
    \[
        \widetilde{\mbK}_{\vartheta;\omega}(e_0\in B)
        =
        \delta_{\vartheta(\omega)}(B),
        \qquad
        B\in\borel{\states},
    \]
    and, for every \(n\ge0\),
    \[
        \widetilde{\mbK}_{\vartheta;\omega}\!\left(
            e_{n+1}\in B \,\middle|\, e_0,\dots,e_n
        \right)
        =
        K_{\theta^n\omega}(e_n,B)
        \qquad
        \widetilde{\mbK}_{\vartheta;\omega}\text{-a.s.}
    \]
    for every \(B\in\borel{\states}\).

    Thus \(\widetilde{\mbK}_{\vartheta;\omega}\) satisfies the defining  property of \(\mbK_{\vartheta;\omega}\) in \Cref{prop:construction-of-bbK}. 
    By uniqueness in that proposition, we conclude that
    \[
        \widetilde{\mbK}_{\vartheta;\omega}
        =
        \mbK_{\vartheta;\omega}.
    \]
    This is exactly the claimed identity.
\end{proof}

\begin{remark}
\label{rem:outcome-vs-state-path}
    For a measurable initial state \(\vartheta:\Omega\to\states\) and a fixed \(\omega\in\Omega\), the posterior trajectory may be viewed either as the sequence of measurable maps \(\rho_n^{\omega,\vartheta}:\Xi^{\mbN}\to\states\) on the outcome-path space, or as the coordinate process \((e_n)_{n\ge0}\) on the state-path space \((\states^{\mbN_0},\mbK_{\vartheta;\omega})\).
    By \Cref{prop:bbK-is-pushforward-of-mbQ}, these two descriptions are equivalent.
    For a general measurable family of initial laws
    \[
        \omega\longmapsto \zeta_\omega\in\mcP(\states),
    \]
    the state-path law \(\mbK_{\zeta;\omega}\) from \Cref{prop:construction-of-bbK} is the natural object. 
    By contrast, if one defines an outcome law by mixing \(\mbQ_{\rho;\omega}\) against
    \(\zeta_\omega\), then the resulting law depends only on the barycenter
    \[
        \bar\rho_\zeta(\omega)
        :=
        \int_{\states}\rho\,\zeta_\omega(\dee\rho),
    \]
    by linearity of the cylinder formula. Thus allowing a general initial law gives a genuine extension for the posterior-state chain, but not for the observed outcome law.
\end{remark}


\section{Proofs of \texorpdfstring{\Cref{thm:s_exists}}{Theorem~1} and \texorpdfstring{\Cref{thm:birkhoff_ergodic}}{Theorem~2}}


We provide the proofs of our first two results in this section.
    

\subsection{Proof of \texorpdfstring{\Cref{thm:s_exists}}{Theorem~1}}


\rhosexists*
\begin{proof}
    We work throughout modulo \(\pr\)-almost sure equality.
    We also continue to work with the measurable density realization \((\omega,x)\mapsto \mcT_{x;\omega}\), \(\omega\mapsto \nu_\omega\) fixed once and for all in \Cref{prop:instrument-density-representation}, and with the associated posterior-state kernels \(K_\omega\) and non-selective channels \(\Phi_\omega\).

    Let \(\mfM\) be the set of measurable families
    \[
        \zeta=\{\zeta_\omega\}_{\omega\in\Omega},
        \qquad
        \zeta_\omega\in\mcP(\states),
    \]
    modulo \(\pr\)-almost sure equality, endowed with the metric
    \[
        d(\zeta,\widetilde\zeta)
        :=
        \operatorname*{ess\,sup}_{\omega\in\Omega}
        \norm{
            \zeta_\omega-\widetilde\zeta_\omega
        }_{\mathrm{TV}}.
    \]
    Since \(\mcP(\states)\) is complete for \(\norm{\,\cdot\,}_{\mathrm{TV}}\), the metric space \((\mfM,d)\) is complete.
    To see that \((\mfM,d)\) is complete, let \((\zeta^{(m)})_{m\ge1}\) be \(d\)-Cauchy. 
        Passing to a subsequence if necessary, we may assume
    \[
        d(\zeta^{(m+1)},\zeta^{(m)})\le 2^{-m}.
    \]
    Then for each \(m\) there exists a null set \(N_m\) such that
    \[
        \norm{
            \zeta^{(m+1)}_\omega-\zeta^{(m)}_\omega
        }_{\mathrm{TV}}
        \le 2^{-m}
        \qquad\text{for all }\omega\notin N_m.
    \]
    Hence, on \(\Omega\setminus N\), where \(N:=\bigcup_{m\ge1}N_m\), the sequence \((\zeta^{(m)}_\omega)_{m\ge1}\) is Cauchy in \((\mcP(\states),\norm{\,\cdot\,}_{\mathrm{TV}})\), and therefore converges to some \(\mu_\omega\in\mcP(\states)\). 
    Set \(\mu_\omega:=\delta_{\rho_\ast}\) for \(\omega\in N\), where \(\rho_\ast\in\states\) is the fixed reference state.
    For every \(B\in\borel{\states}\), the map \(\omega\mapsto\mu_\omega(B)\) is measurable on \(\Omega\setminus N\) as a pointwise limit of measurable functions, and its chosen extension is constant on \(N\).
    Thus \(\omega\mapsto\mu_\omega\) is measurable for the evaluation sigma-algebra on \(\mcP(\states)\).
     Finally, along the chosen subsequence,
    \[
        d(\zeta^{(m)},\mu)
        \le \sum_{k=m}^{\infty}2^{-k}
        =2^{1-m}\longrightarrow0.
    \]
    Since the original sequence is \(d\)-Cauchy, it also converges to \(\mu\).
    Thus \((\mfM,d)\) is complete.
   
    Define the pullback operator \(T:\mfM\to\mfM\) by
    \[
        (T\zeta)_\omega
        :=
        \zeta_{\theta^{-1}\omega}K_{\theta^{-1}\omega},
    \]
    i.e.
    \[
        (T\zeta)_\omega(B)
        =
        \int_{\states}
        K_{\theta^{-1}\omega}(\rho,B)\,
        \zeta_{\theta^{-1}\omega}(\dee\rho),
        \qquad
        B\in\borel{\states}.
    \]
    This is well defined on equivalence classes, because \(\theta\) preserves null sets. 
    Moreover, \(T\zeta\) is again measurable: for every \(B\in\borel{\states}\), the map \((\omega,\rho)\mapsto K_{\theta^{-1}\omega}(\rho,B)\) is measurable by \Cref{prop:Gamma-K-are-kernels} and the measurability of \(\theta^{-1}\), and therefore
    \[
        \omega\longmapsto
        \int_{\states}
        K_{\theta^{-1}\omega}(\rho,B)\,
        \zeta_{\theta^{-1}\omega}(\dee\rho)
    \]
    is measurable by the standard theorem on integration against kernels.
    By induction, for every \(n\ge1\), we also have that
    \[
        (T^n\zeta)_\omega
        =
        \zeta_{\theta^{-n}\omega}K_{\theta^{-n}\omega}^{(n)}.
    \]

    We first note that \(T\) is non-expansive. 
    Indeed, if \(Q\) is any Markov kernel on \((\states,\borel{\states})\) and \(\alpha, \beta\in\mcP(\states)\),
    then
    \[
        \norm{\alpha Q-\beta Q}_{\mathrm{TV}}
        \le
        \norm{\alpha-\beta}_{\mathrm{TV}}.
    \]
    To see this, let \(f:\states\to\mbR\) be bounded and measurable with
    \(\norm{f}_\infty\le1\). Then
    \[
        \int_{\states} f(\eta)\,(\alpha Q-\beta Q)(\dee\eta)
        =
        \int_{\states} (Qf)(\rho)\,(\alpha-\beta)(\dee\rho),
    \]
    where
    \[
        (Qf)(\rho):=\int_{\states} f(\eta)\,Q(\rho,\dee\eta)
    \]
    also satisfies \(\norm{Qf}_\infty\le1\). 
    Taking the supremum over all such \(f\) yields the claim. 
    Consequently,
    \[
        d(T\zeta,T\widetilde\zeta)\le d(\zeta,\widetilde\zeta).
    \]

    Next we use the Doeblin decomposition from \Cref{assumption_doeblin}. 
    For
    \(\pr\)-almost every \(\omega\),
    \[
        K_{\theta^{-L}\omega}^{(L)}(\rho,\,\cdot\,)
        =
        \varepsilon\,\lambda_\omega(\,\cdot\,)
        +
        (1-\varepsilon)\,R_{\theta^{-L}\omega}(\rho,\,\cdot\,).
    \]
    Hence, for \(\pr\)-almost every \(\omega\),
    \[
        (T^L\zeta)_\omega
        =
        \zeta_{\theta^{-L}\omega}K_{\theta^{-L}\omega}^{(L)}
        =
        \varepsilon\,\lambda_\omega
        +
        (1-\varepsilon)\,
        \zeta_{\theta^{-L}\omega}R_{\theta^{-L}\omega}.
    \]
    Therefore
    \begin{align*}
        \norm{
            (T^L\zeta)_\omega-(T^L\widetilde\zeta)_\omega
        }_{\mathrm{TV}}
        =
        (1-\varepsilon)
        \norm{
            \zeta_{\theta^{-L}\omega}R_{\theta^{-L}\omega}
            -
            \widetilde\zeta_{\theta^{-L}\omega}R_{\theta^{-L}\omega}
        }_{\mathrm{TV}} 
        \le
        (1-\varepsilon)
        \norm{
            \zeta_{\theta^{-L}\omega}
            -
            \widetilde\zeta_{\theta^{-L}\omega}
        }_{\mathrm{TV}}.
    \end{align*}
    Taking essential suprema gives
    \[
        d(T^L\zeta,T^L\widetilde\zeta)
        \le
        (1-\varepsilon)\,d(\zeta,\widetilde\zeta).
    \]
    Thus \(T^L\) is a strict contraction on the complete metric space \((\mfM,d)\).

    By the Banach fixed-point theorem, there exists a unique element of \(\mfM\), represented by a measurable family still denoted \(\mu=\{\mu_\omega\}_{\omega\in\Omega}\), such that \(T^L\mu=\mu\).
    Since
    \[
        T^L(T\mu)=T(T^L\mu)=T\mu,
    \]
    the family \(T\mu\) is also a fixed point of \(T^L\). 
    By uniqueness of the fixed point of \(T^L\) in \(\mfM\), we conclude that
    \[
        T\mu=\mu
        \qquad
        \text{in }\mfM.
    \]
    Equivalently, for \(\pr\)-almost every \(\omega\), we get \( \mu_\omega=\mu_{\theta^{-1}\omega}K_{\theta^{-1}\omega}\).
    Replacing \(\omega\) by \(\theta\omega\), we obtain \((K_\omega)^*\mu_\omega=\mu_{\theta\omega}\) for \(\pr\)-almost every \(\omega\).
    This proves existence and uniqueness of the equivariant family, up to \(\pr\)-almost sure equality.

    The quantitative pullback estimate now follows immediately. Let \(n\ge0\), and write
    \[
        n=qL+r,
        \qquad
        q=\lfloor n/L\rfloor,
        \qquad
        0\le r<L.
    \]
    Then
    \begin{equation*}
        d(T^n\zeta,\mu)
        =
        d(T^rT^{qL}\zeta,T^r\mu) 
        \le
        d(T^{qL}\zeta,\mu) 
        \le
        (1-\varepsilon)^q\,d(\zeta,\mu),
    \end{equation*}
    because \(T\) is non-expansive and \(T^r\mu=\mu\). Using
    \[
        (T^n\zeta)_\omega
        =
        \zeta_{\theta^{-n}\omega}K_{\theta^{-n}\omega}^{(n)},
    \]
    we obtain
    \[
        \esssup_{\omega}\,
        \norm{
            \zeta_{\theta^{-n}\omega}K_{\theta^{-n}\omega}^{(n)}
            -
            \mu_\omega
        }_{\mathrm{TV}}
        \le
        (1-\varepsilon)^{\lfloor n/L\rfloor}
        \operatorname*{ess\,sup}_{\omega}
        \norm{\zeta_\omega-\mu_\omega}_{\mathrm{TV}}.
    \]

    For \(m\in\mcP(\states)\), define its barycenter by
    \[
        b(m)
        :=
        \int_{\states}\eta\,m(\dee\eta)\in\states,
    \]
    where the integral is taken in the finite-dimensional vector space \(\matrices\).

    We claim that for every \(\omega\in\Omega\) and every \(\alpha\in\mcP(\states)\),
    \[
        b(\alpha K_\omega)=\Phi_\omega\left(b(\alpha)\right).
    \]
    Indeed, since all integrals are taken in a finite-dimensional vector space, Fubini's theorem applies coordinatewise, and
    \begin{equation*}
        b(\alpha K_\omega)
        =
        \int_{\states}\eta\,(\alpha K_\omega)(\dee\eta) 
        =
        \int_{\states}
        \left(
            \int_{\states}\eta\,K_\omega(\rho,\dee\eta)
        \right)\alpha(\dee\rho).
    \end{equation*}
    Now, by definition of \(K_\omega\),
    \[
        \int_{\states}\eta\,K_\omega(\rho,\dee\eta)
        =
        \int_\Xi
        \left(\mcT_{x;\omega}\proj\rho\right)\,
        \tr{\mcT_{x;\omega}(\rho)}\,
        \nu_\omega(\dee x).
    \]
    If \(\tr{\mcT_{x;\omega}(\rho)}>0\), then
    \[
        \left(\mcT_{x;\omega}\proj\rho\right)\,
        \tr{\mcT_{x;\omega}(\rho)}
        =
        \mcT_{x;\omega}(\rho).
    \]
    If \(\tr{\mcT_{x;\omega}(\rho)}=0\), then \(\mcT_{x;\omega}(\rho)\ge0\) has trace \(0\), hence
    \(\mcT_{x;\omega}(\rho)=0\), and the same identity still holds. Therefore
    \[
        \int_{\states}\eta\,K_\omega(\rho,\dee\eta)
        =
        \int_\Xi \mcT_{x;\omega}(\rho)\,\nu_\omega(\dee x)
        =
        \mfI_\omega(\Xi)(\rho)
        =
        \Phi_\omega(\rho).
    \]
    Hence
    \begin{equation*}
        b(\alpha K_\omega)
        =
        \int_{\states}\Phi_\omega(\rho)\,\alpha(\dee\rho) 
        =
        \Phi_\omega\!\left(
            \int_{\states}\rho\,\alpha(\dee\rho)
        \right) 
        =
        \Phi_\omega\left(b(\alpha)\right),
    \end{equation*}
    since \(\Phi_\omega\) is linear.

    Iterating this identity yields, for every \(n\ge0\),
    \[
        b(\alpha K_\omega^{(n)})
        =
        \Phi_\omega^{(n)}\left(b(\alpha)\right).
    \]

    Set
    \[
        \s(\omega)
        :=
        b(\mu_\omega)
        =
        \int_{\states}\eta\,\mu_\omega(\dee\eta).
    \]
    Since \(\eta\mapsto\eta\) is a bounded Borel map on the compact space \(\states\) and \(\omega\mapsto\mu_\omega\) is measurable, the map \(\omega\mapsto \s(\omega)\) is measurable. 
    Using the equivariance of \(\mu\), we obtain, for \(\pr\)-almost every \(\omega\),
    \[
        \s(\theta\omega)
        =
        b(\mu_{\theta\omega})
        =
        b(\mu_\omega K_\omega)
        =
        \Phi_\omega\left(b(\mu_\omega)\right)
        =
        \Phi_\omega(\s(\omega)).
    \]
    Hence \(\s\) is dynamically stationary.

    It remains to prove uniqueness of the dynamically stationary state, up to \(\pr\)-almost sure equality. 
    Let \(\widetilde\s:\Omega\to\states\) be any dynamically stationary state, and define
    \[
        \zeta_\omega:=\delta_{\widetilde\s(\omega)}.
    \]
    Then \(\zeta\in\mfM\). For every \(n\ge0\), we have, for \(\pr\)-almost every \(\omega\),
    \begin{equation*}
        b\left((T^n\zeta)_\omega\right)
        =
        b\left(
            \delta_{\widetilde\s(\theta^{-n}\omega)}
            K_{\theta^{-n}\omega}^{(n)}
        \right) 
        =
        \Phi_{\theta^{-n}\omega}^{(n)}
        \left(\widetilde\s(\theta^{-n}\omega)\right) 
        =
        \widetilde\s(\omega),
    \end{equation*}
    where the last equality follows by iterating the relation
    \[
        \widetilde\s(\theta\eta)=\Phi_\eta\left(\widetilde\s(\eta)\right).
    \]

    We next note that for any \(\alpha,\beta\in\mcP(\states)\),
    \[
        \norm{b(\alpha)-b(\beta)}_1
        \le
        \norm{\alpha-\beta}_{\mathrm{TV}}.
    \]
    Indeed, by duality of the trace norm,
    \begin{align*}
        \norm{b(\alpha)-b(\beta)}_1
        &=
        \sup_{\norm{H}_\infty\le1}
        \left|
            \tr{
                H\adj\int_{\states}\eta\,(\alpha-\beta)(\dee\eta)
            }
        \right| \\
        &=
        \sup_{\norm{H}_\infty\le1}
        \left|
            \int_{\states}\tr{H\adj\eta}\,(\alpha-\beta)(\dee\eta)
        \right| \\
        &\le
        \norm{\alpha-\beta}_{\mathrm{TV}},
    \end{align*}
    since \(|\tr{H\adj\eta}|\le \norm{H\adj}_\infty\norm{\eta}_1 = \norm{H}_\infty\norm{\eta}_1\le1\) for every \(\eta\in\states\).

    Therefore, for every \(n\ge0\),
    \begin{equation*}
        \esssup_{\omega}
        \norm{\widetilde\s(\omega)-\s(\omega)}_1
        \le
        \esssup_{\omega}
        \norm{
            (T^n\zeta)_\omega-\mu_\omega
        }_{\mathrm{TV}} 
        \le
        (1-\varepsilon)^{\lfloor n/L\rfloor}\,d(\zeta,\mu).
    \end{equation*}
    Since \(d(\zeta,\mu)\le2\), we obtain
    \[
        \esssup_{\omega}
        \norm{\widetilde\s(\omega)-\s(\omega)}_1
        \le
        2(1-\varepsilon)^{\lfloor n/L\rfloor}
        \qquad
        \text{for every }n\ge0.
    \]
    Letting \(n\to\infty\) yields
    \[
        \widetilde\s(\omega)=\s(\omega)
        \qquad
        \text{for \(\pr\)-almost every }\omega.
    \]
    Thus, the dynamically stationary state is unique up to \(\pr\)-almost sure equality.

    Finally, we prove the last estimate. Fix \(n\ge0\). 
    For each \(\rho\in\states\), define the constant family
    \[
        \zeta^\rho_\omega:=\delta_\rho,
        \qquad \omega\in\Omega.
    \]
    Then, for \(\pr\)-almost every \(\omega\),
    \[
        b\left((T^n\zeta^\rho)_\omega\right)
        =
        \Phi_{\theta^{-n}\omega}^{(n)}(\rho).
    \]
    Hence, for each fixed \(\rho\in\states\),
    \[
        \esssup_{\omega}
        \norm{
            \Phi_{\theta^{-n}\omega}^{(n)}(\rho)-\s(\omega)
        }_1
        \le
        \esssup_{\omega}
        \norm{
            (T^n\zeta^\rho)_\omega-\mu_\omega
        }_{\mathrm{TV}}
        \le
        2(1-\varepsilon)^{\lfloor n/L\rfloor}.
    \]

    Since \(\states\) is a compact metric space, let \(D\subseteq\states\) be a countable dense subset. For each \(\rho\in D\), choose a full-measure set \(\Omega_{n,\rho}\subseteq\Omega\) such that
    \[
        \norm{
            \Phi_{\theta^{-n}\omega}^{(n)}(\rho)-\s(\omega)
        }_1
        \le
        2(1-\varepsilon)^{\lfloor n/L\rfloor}
        \qquad
        \text{for all }\omega\in\Omega_{n,\rho}.
    \]
    Set
    \[
        \Omega_n:=\bigcap_{\rho\in D}\Omega_{n,\rho}.
    \]
    Then \(\pr(\Omega_n)=1\), and for every \(\omega\in\Omega_n\) and every
    \(\rho\in D\),
    \[
        \norm{
            \Phi_{\theta^{-n}\omega}^{(n)}(\rho)-\s(\omega)
        }_1
        \le
        2(1-\varepsilon)^{\lfloor n/L\rfloor}.
    \]
    For fixed \(\omega\in\Omega_n\), the map
    \[
        \rho\longmapsto
        \norm{
            \Phi_{\theta^{-n}\omega}^{(n)}(\rho)-\s(\omega)
        }_1
    \]
    is continuous on \(\states\), since \(\Phi_{\theta^{-n}\omega}^{(n)}:\matrices\to\matrices\) is linear on the finite-dimensional space \(\matrices\). Therefore the same bound extends from \(D\) to all \(\rho\in\states\). Hence
    \[
        \sup_{\rho\in\states}
        \norm{
            \Phi_{\theta^{-n}\omega}^{(n)}(\rho)-\s(\omega)
        }_1
        \le
        2(1-\varepsilon)^{\lfloor n/L\rfloor}
        \qquad
        \text{for every }\omega\in\Omega_n.
    \]
    Taking the essential supremum over \(\omega\) proves
    \[
        \esssup_{\omega}\,
        \sup_{\rho\in\states}
        \norm{
            \Phi_{\theta^{-n}\omega}^{(n)}(\rho)-\s(\omega)
        }_1
        \le
        2(1-\varepsilon)^{\lfloor n/L\rfloor}.
    \]
\end{proof}


\subsection{A pathwise ergodic theorem: proof of \texorpdfstring{\Cref{thm:birkhoff_ergodic}}{Theorem~2}}
\label{section:lln_for_trajectories}


Before we provide the proof of \Cref{thm:birkhoff_ergodic}, we record a corollary of the proof of \Cref{thm:s_exists} that will be useful later. 
The proof of the following corollary is contained in the proof of \Cref{thm:s_exists} above and is therefore omitted. 

\begin{cor}
\label{cor:doeblin-forgetting}
    Under \Cref{assumption_1,assumption_doeblin}, let \(\s:\Omega\to\states\) be the unique dynamically stationary state given by \Cref{thm:s_exists}, and set
    \[
        r_n:=2(1-\varepsilon)^{\lfloor n/L\rfloor},
        \qquad n\ge1.
    \]
    Then there exists a \(\theta\)-invariant full-measure set \(\Omega_\ast\subseteq\Omega\) such that for every \(\omega\in\Omega_\ast\), every \(n\ge1\), and every \(\rho\in\states\),
    \[
        \norm{
            \Phi_\omega^{(n)}(\rho)-\s(\theta^n\omega)
        }_1
        \le
        r_n.
    \]
    Equivalently, for every \(\omega\in\Omega_\ast\), every \(m\in\mbN_0\), every \(\ell\ge1\), and every \(\rho\in\states\),
    \[
        \norm{
            \Phi_{\theta^m\omega}^{(\ell)}(\rho)-\s(\theta^{m+\ell}\omega)
        }_1
        \le
        r_\ell.
    \]
\end{cor}

We also need the following lemma. 

\begin{lemma}
\label{lem:conditional-mean-posterior}
Fix a measurable initial state \(\vartheta:\Omega\to\states\) and an environment \(\omega\in\Omega\).
Let
\[
    \mcG_n:=\sigma(X_1,\dots,X_n),
    \qquad n\ge0,
\]
be the natural outcome filtration on \(\Xi^{\mbN}\), where \(X_n:\Xi^{\mbN}\to\Xi\) denotes the \(n\)-th coordinate map. 
Then for all \(n\ge0\) and \(\ell\ge1\),
\[
    \mbE_{\mbQ_{\vartheta;\omega}}
    \left[
        \rho_{n+\ell}^{\omega,\vartheta}
        \,\middle|\,
        \mcG_n
    \right]
    =
    \Phi_{\theta^n\omega}^{(\ell)}
    \left(
        \rho_n^{\omega,\vartheta}
    \right)
    \qquad
    \mbQ_{\vartheta;\omega}\text{-a.s.}
\]
\end{lemma}
\begin{proof}
    We continue to work with the measurable density realization \((\omega,x)\mapsto \mcT_{x;\omega}\), \(\omega\mapsto \nu_\omega\) fixed once and for all in \Cref{prop:instrument-density-representation}, together with the associated kernels \(\Gamma_\omega\), \(K_\omega\), quenched laws \(\mbQ_{\vartheta;\omega}\), and posterior processes \((\rho_n^{\omega,\vartheta})_{n\ge0}\).
    For brevity, write
    \[
        \rho_n:=\rho_n^{\omega,\vartheta},
        \qquad
        \kappa_n:=\kappa_{n,\vartheta}^\omega.
    \]
    Since \(\rho_n(\bar x)\in\states\) for every \(n\in\mbN\) and every \(\bar x\), we have
    \[
        \norm{\rho_n}_1=1
        \qquad
        \mbQ_{\vartheta;\omega}\text{-a.s.},
    \]
    so each \(\rho_n\) is Bochner integrable as an \(M_d(\mbC)\)-valued random variable.
    
    We first prove the one-step identity
    \begin{equation}
    \label{eq:one-step-cond-mean}
        \mbE_{\mbQ_{\vartheta;\omega}}[\rho_{n+1}\mid \mcG_n]
        =
        \Phi_{\theta^n\omega}(\rho_n)
        \qquad
        \mbQ_{\vartheta;\omega}\text{-a.s.}
    \end{equation}
    
    Since \(M_d(\mbC)\) is finite-dimensional, it suffices to test this identity
    against an arbitrary linear functional
    \[
        \Lambda:M_d(\mbC)\to\mbC.
    \]
    Let \(H\) be a bounded \(\mcG_n\)-measurable scalar random variable. Then
    \begin{align*}
        \mbE_{\mbQ_{\vartheta;\omega}}\!\left[H\,\Lambda(\rho_{n+1})\right]
        &=
        \mbE_{\mbQ_{\vartheta;\omega}}
        \left[
            H\,
            \Lambda\!\left(
                \mcT_{X_{n+1};\theta^n\omega}\proj\rho_n
            \right)
        \right].
    \end{align*}
    Using the defining conditional law of \(X_{n+1}\) under \(\mbQ_{\vartheta;\omega}\), we obtain
    \begin{align*}
        \mbE_{\mbQ_{\vartheta;\omega}}\!\left[H\,\Lambda(\rho_{n+1})\right]
        &=
        \mbE_{\mbQ_{\vartheta;\omega}}
        \left[
            H
            \int_\Xi
            \Lambda\!\left(
                \mcT_{x;\theta^n\omega}\proj\rho_n
            \right)\,
            \kappa_n(X_1,\dots,X_n;\dee x)
        \right].
    \end{align*}
    By definition of \(\kappa_n\),
    \[
        \kappa_n(X_1,\dots,X_n;\dee x)
        =
        \Gamma_{\theta^n\omega}(\rho_n,\dee x)
        =
        \tr{\mcT_{x;\theta^n\omega}(\rho_n)}\,\nu_{\theta^n\omega}(\dee x).
    \]
    Hence
    \begin{align*}
        \mbE_{\mbQ_{\vartheta;\omega}}\!\left[H\,\Lambda(\rho_{n+1})\right]
        &=
        \mbE_{\mbQ_{\vartheta;\omega}}
        \left[
            H
            \int_\Xi
            \Lambda\!\left(
                \mcT_{x;\theta^n\omega}\proj\rho_n
            \right)\,
            \tr{\mcT_{x;\theta^n\omega}(\rho_n)}\,
            \nu_{\theta^n\omega}(\dee x)
        \right].
    \end{align*}
    For fixed \(x\in\Xi\), set
    \[
        a_x:=\tr{\mcT_{x;\theta^n\omega}(\rho_n)}.
    \]
    If \(a_x>0\), then by definition of the projective update,
    \[
        \mcT_{x;\theta^n\omega}\proj\rho_n
        =
        \frac{\mcT_{x;\theta^n\omega}(\rho_n)}{a_x},
    \]
    and therefore
    \[
        \Lambda\!\left(\mcT_{x;\theta^n\omega}\proj\rho_n\right)\,a_x
        =
        \Lambda\!\left(\mcT_{x;\theta^n\omega}(\rho_n)\right).
    \]
    If \(a_x=0\), then \(\mcT_{x;\theta^n\omega}(\rho_n)\ge0\) has trace \(0\), hence it is the zero matrix. 
    Thus, the same identity still holds. 
    Therefore, for every \(x\in\Xi\),
    \[
        \Lambda\!\left(\mcT_{x;\theta^n\omega}\proj\rho_n\right)\,
        \tr{\mcT_{x;\theta^n\omega}(\rho_n)}
        =
        \Lambda\!\left(\mcT_{x;\theta^n\omega}(\rho_n)\right).
    \]
    Substituting this into the previous display gives
    \begin{align*}
        \mbE_{\mbQ_{\vartheta;\omega}}\!\left[H\,\Lambda(\rho_{n+1})\right]
        &=
        \mbE_{\mbQ_{\vartheta;\omega}}
        \left[
            H
            \int_\Xi
            \Lambda\!\left(\mcT_{x;\theta^n\omega}(\rho_n)\right)\,
            \nu_{\theta^n\omega}(\dee x)
        \right].
    \end{align*}
    Since \(\Lambda\) is linear and continuous, it commutes with integration:
    \[
        \int_\Xi
        \Lambda\!\left(\mcT_{x;\theta^n\omega}(\rho_n)\right)\,
        \nu_{\theta^n\omega}(\dee x)
        =
        \Lambda\!\left(
            \int_\Xi \mcT_{x;\theta^n\omega}(\rho_n)\,\nu_{\theta^n\omega}(\dee x)
        \right)
        =
        \Lambda\!\left(
            \mfI_{\theta^n\omega}(\Xi)(\rho_n)
        \right)
        =
        \Lambda\!\left(\Phi_{\theta^n\omega}(\rho_n)\right).
    \]
    Hence
    \[
        \mbE_{\mbQ_{\vartheta;\omega}}\!\left[H\,\Lambda(\rho_{n+1})\right]
        =
        \mbE_{\mbQ_{\vartheta;\omega}}
        \!\left[
            H\,\Lambda(\Phi_{\theta^n\omega}(\rho_n))
        \right].
    \]
    Since this holds for every bounded \(\mcG_n\)-measurable \(H\) and every linear functional \(\Lambda\), we obtain \eqref{eq:one-step-cond-mean}.

    \smallskip
    We now prove the general \(\ell\)-step identity by induction on \(\ell\).
    The case \(\ell=1\) is exactly \eqref{eq:one-step-cond-mean}. Assume the claim holds for some \(\ell\ge1\). 
    Then
    \begin{align*}
        \mbE_{\mbQ_{\vartheta;\omega}}
        \left[
            \rho_{n+\ell+1}
            \,\middle|\,
            \mcG_n
        \right]
        &=
        \mbE_{\mbQ_{\vartheta;\omega}}
        \left[
            \mbE_{\mbQ_{\vartheta;\omega}}
            \left[
                \rho_{n+\ell+1}
                \,\middle|\,
                \mcG_{n+\ell}
            \right]
            \,\middle|\,
            \mcG_n
        \right] \\
        &=
        \mbE_{\mbQ_{\vartheta;\omega}}
        \left[
            \Phi_{\theta^{n+\ell}\omega}(\rho_{n+\ell})
            \,\middle|\,
            \mcG_n
        \right].
    \end{align*}
    Now \(\Phi_{\theta^{n+\ell}\omega}:M_d(\mbC)\to M_d(\mbC)\) is a bounded linear map, depending only on the fixed environment \(\omega\). 
    Since \(M_d(\mbC)\) is finite-dimensional, bounded linear maps commute with Bochner conditional expectation. 
    Thus
    \[
        \mbE_{\mbQ_{\vartheta;\omega}}
        \left[
            \Phi_{\theta^{n+\ell}\omega}(\rho_{n+\ell})
            \,\middle|\,
            \mcG_n
        \right]
        =
        \Phi_{\theta^{n+\ell}\omega}
        \left(
            \mbE_{\mbQ_{\vartheta;\omega}}
            \left[
                \rho_{n+\ell}
                \,\middle|\,
                \mcG_n
            \right]
        \right).
    \]
    Applying the induction hypothesis,
    \[
        \mbE_{\mbQ_{\vartheta;\omega}}
        \left[
            \rho_{n+\ell}
            \,\middle|\,
            \mcG_n
        \right]
        =
        \Phi_{\theta^n\omega}^{(\ell)}(\rho_n),
    \]
    we conclude that
    \begin{align*}
        \mbE_{\mbQ_{\vartheta;\omega}}
        \left[
            \rho_{n+\ell+1}
            \,\middle|\,
            \mcG_n
        \right]
        &=
        \Phi_{\theta^{n+\ell}\omega}
        \left(
            \Phi_{\theta^n\omega}^{(\ell)}(\rho_n)
        \right)
        =
        \Phi_{\theta^n\omega}^{(\ell+1)}(\rho_n).
    \end{align*}
    This completes the induction.
\end{proof}

We are now ready to prove \Cref{thm:birkhoff_ergodic}.

\pathbirkhoffergodic*

\begin{proof}
    As always, we work with the fixed measurable density realization \((\omega,x)\mapsto \mcT_{x;\omega}\), \(\omega\mapsto \nu_\omega\) furnished by \Cref{prop:instrument-density-representation}.
    Fix a measurable initial state \(\vartheta:\Omega\to\states\), and let \(\rho_{\mathrm{ss}}\) be the unique dynamically stationary state given by \Cref{thm:s_exists}.
    Set
    \[
        r_n:=2(1-\varepsilon)^{\lfloor n/L\rfloor},
        \qquad n\ge1.
    \]
    
    Since \(\norm{\rho_{\mathrm{ss}}(\omega)}_1=1\) for all \(\omega\), the map \(\rho_{\mathrm{ss}}:\Omega\to\matrices\) is Bochner integrable.
    Fix a basis \(E_1,\dots,E_M\) of \(\matrices\), where \(M=d^2\), and let \(\ell_1,\dots,\ell_M:\matrices\to\mbC\) be the corresponding coordinate functionals. 
    Each \(\ell_j\) is continuous, so there exists \(C_j<\infty\) such that
    \[
        |\ell_j(X)|\le C_j\norm{X}_1
        \qquad\text{for all }X\in\matrices.
    \]
    Therefore
    \[
        |\ell_j(\rho_{\mathrm{ss}}(\omega))|
        \le C_j
        \qquad\text{for all }\omega\in\Omega,
    \]
    so each scalar function \(\omega\mapsto \ell_j(\rho_{\mathrm{ss}}(\omega))\) belongs to \(L^1(\pr)\). 
    Applying the scalar Birkhoff ergodic theorem to the real and imaginary parts of these finitely many coordinate functions, and using the ergodicity of \((\Omega,\mcF,\pr,\theta)\), we obtain a full-measure set \(\Omega_0'\subseteq\Omega\) such that for every \(\omega\in\Omega_0'\) and every \(j\),
    \[
        \frac1N\sum_{n=1}^N \ell_j(\rho_{\mathrm{ss}}(\theta^n\omega))
        \longrightarrow
        \int_\Omega \ell_j(\rho_{\mathrm{ss}}(\eta))\,\pr(\dee\eta)
        =
        \ell_j(\bar\rho_{\mathrm{ss}}).
    \]
    Hence
    \[
        \ell_j\!\left(
            \frac1N\sum_{n=1}^N \rho_{\mathrm{ss}}(\theta^n\omega)
            -
            \bar\rho_{\mathrm{ss}}
        \right)
        \longrightarrow 0
        \qquad
        (j=1,\dots,M).
    \]
    Since all norms on the finite-dimensional space \(\matrices\) are equivalent, coordinatewise convergence implies convergence in trace norm. 
    Thus
    \begin{equation}
    \label{eq:birkhoff-rhoss-short}
        \frac1N\sum_{n=1}^N \rho_{\mathrm{ss}}(\theta^n\omega)
        \longrightarrow
        \bar\rho_{\mathrm{ss}}
        \qquad\text{in trace norm}
    \end{equation}
    for every \(\omega\in\Omega_0'\).
    
    Let \(\Omega_\ast\) be the full-measure set from \Cref{cor:doeblin-forgetting}, and define
    \[
        \Omega_0:=\Omega_0'\cap\Omega_\ast.
    \]
    Then \(\Omega_0\) still has full \(\pr\)-measure. 
    Fix \(\omega\in\Omega_0\), and work on the probability space
    \[
        \left(\Xi^{\mbN},\mcX^{\otimes\mbN},\mbQ_{\vartheta;\omega}\right).
    \]
    For brevity, write
    \[
        \rho_n:=\rho_n^{\omega,\vartheta},
        \qquad
        \mcG_n:=\sigma(X_1,\dots,X_n).
    \]
    
    For each \(m\in\mbN\) and \(n\ge0\), define
    \[
        B_n^{(m)}
        :=
        \frac1m\sum_{\ell=1}^m \rho_{n+\ell},
    \qquad
        C_n^{(m)}
        :=
        \frac1m\sum_{\ell=1}^m
        \Phi_{\theta^n\omega}^{(\ell)}(\rho_n),
    \]
    and
    \[
        D_n^{(m)}
        :=
        B_n^{(m)}-C_n^{(m)}.
    \]
    By \Cref{lem:conditional-mean-posterior}, \(C_n^{(m)}\) is a version of
    \[
        \mbE_{\mbQ_{\vartheta;\omega}}
        \left[
            B_n^{(m)}
            \,\middle|\,
            \mcG_n
        \right].
    \]
    Also define the block average of the dynamically stationary states
    \[
        R_n^{(m)}(\omega)
        :=
        \frac1m\sum_{\ell=1}^m \rho_{\mathrm{ss}}(\theta^{n+\ell}\omega).
    \]
    Since \(\omega\in\Omega_\ast\), \Cref{cor:doeblin-forgetting} gives
    \[
        \norm{
            \Phi_{\theta^n\omega}^{(\ell)}(\rho_n)
            -
            \rho_{\mathrm{ss}}(\theta^{n+\ell}\omega)
        }_1
        \le
        r_\ell
        \qquad
        \mbQ_{\vartheta;\omega}\text{-a.s.}
    \]
    for every \(n\ge0\) and \(\ell\ge1\). 
    Hence, for every \(m\in\mbN\),
    \begin{equation}
    \label{eq:block-error-short-rewrite}
        \norm{C_n^{(m)}-R_n^{(m)}(\omega)}_1
        \le
        \delta_m
        :=
        \frac1m\sum_{\ell=1}^m r_\ell
        \qquad
        \mbQ_{\vartheta;\omega}\text{-a.s.}
    \end{equation}
    Since \(r_\ell\to0\), Cesàro's lemma yields
    \begin{equation}
    \label{eq:delta-m-short-rewrite}
        \delta_m\longrightarrow 0
        \qquad (m\to\infty).
    \end{equation}
    Moreover, \(B_n^{(m)}\) and \(C_n^{(m)}\) are convex combinations of states, so
    \[
        \norm{D_n^{(m)}}_1\le 2
        \qquad
        \mbQ_{\vartheta;\omega}\text{-a.s.}
    \]

    Choose a complex basis \(E_1,\dots,E_M\) of \(M_d(\mbC)\), where \(M=d^2\), and let \(\Lambda_1,\dots,\Lambda_M:M_d(\mbC)\to\mbC\) be the dual basis functionals, so that every \(A\in M_d(\mbC)\) admits the expansion
    \[
        A=\sum_{i=1}^M \Lambda_i(A)\,E_i.
    \]
    Since all norms are equivalent on \(M_d(\mbC)\), there exists \(C_0>0\) such that
    \[
        \norm{A}_1\le C_0\max_{1\le i\le M} |\Lambda_i(A)|
        \qquad
        \text{for all }A\in M_d(\mbC).
    \]
    
    For each \(m\in\mbN\) and each \(r\in\{0,\dots,m-1\}\), define
    \[
        X_j^{(m,r)}:=D_{(j-1)m+r}^{(m)},
        \qquad j\ge1,
    \]
    and
    \[
        \scrH_j^{(m,r)}:=\mcG_{jm+r},
        \qquad j\ge0.
    \]
    Then \(X_j^{(m,r)}\) is \(\scrH_j^{(m,r)}\)-measurable, and
    \[
        \scrH_{j-1}^{(m,r)}
        =
        \mcG_{(j-1)m+r}.
    \]
    Therefore
    \[
        \mbE_{\mbQ_{\vartheta;\omega}}
        \left[
            X_j^{(m,r)}
            \,\middle|\,
            \scrH_{j-1}^{(m,r)}
        \right]
        =0
        \qquad
        \mbQ_{\vartheta;\omega}\text{-a.s.}
    \]
    and so, for every \(i\in\{1,\dots,M\}\), the scalar sequence
    \[
        \Lambda_i\!\left(X_j^{(m,r)}\right)
    \]
    is a bounded martingale difference sequence. By the strong law for uniformly bounded martingale difference sequences, for each fixed \(m\), each \(r\in\{0,\dots,m-1\}\), and each \(i\in\{1,\dots,M\}\),
    \[
        \frac1J\sum_{j=1}^J \Lambda_i\!\left(X_j^{(m,r)}\right)
        \longrightarrow 0
        \qquad
        \mbQ_{\vartheta;\omega}\text{-a.s.}
    \]
    Intersecting over the finitely many pairs \((i,r)\), we obtain a full \(\mbQ_{\vartheta;\omega}\)-measure set \(E_{\omega,m}\subseteq\Xi^{\mbN}\) such that for every \(\bar x\in E_{\omega,m}\) and every \(r\in\{0,\dots,m-1\}\),
    \[
        \frac1J\sum_{j=1}^J X_j^{(m,r)}(\bar x)
        \longrightarrow 0
        \qquad
        \text{in trace norm.}
    \]
    
    For such \(\bar x\), if
    \[
        J_{N,r}
        :=
        \#\{\,0\le n\le N-1:\ n\equiv r \pmod m\,\},
    \]
    then
    \[
        \sum_{n=0}^{N-1} D_n^{(m)}(\bar x)
        =
        \sum_{r=0}^{m-1}\sum_{j=1}^{J_{N,r}} X_j^{(m,r)}(\bar x),
    \]
    and therefore, for \(N\ge m\),
    \begin{equation}
    \label{eq:D-cesaro-short-rewrite}
        \norm{
            \frac1N\sum_{n=0}^{N-1} D_n^{(m)}(\bar x)
        }_1
        \le
        \sum_{r=0}^{m-1}
        \frac{J_{N,r}}{N}
        \norm{
            \frac1{J_{N,r}}
            \sum_{j=1}^{J_{N,r}} X_j^{(m,r)}(\bar x)
        }_1
        \longrightarrow 0.
    \end{equation}
    
    Now define
    \[
        E_\omega:=\bigcap_{m=1}^\infty E_{\omega,m}.
    \]
    Then \(E_\omega\) still has full \(\mbQ_{\vartheta;\omega}\)-measure, and for every \(\bar x\in E_\omega\) and every \(m\in\mbN\), \eqref{eq:D-cesaro-short-rewrite} holds.
    
    Fix \(\bar x\in E_\omega\) and \(m\in\mbN\). By
    \eqref{eq:block-error-short-rewrite},
    \[
        \norm{
            \frac1N\sum_{n=0}^{N-1}
            \left(
                C_n^{(m)}(\bar x)-R_n^{(m)}(\omega)
            \right)
        }_1
        \le
        \delta_m
        \qquad (N\ge1).
    \]
    Combining this with \eqref{eq:D-cesaro-short-rewrite}, we obtain
    \begin{equation}
    \label{eq:block-to-ref-short-rewrite}
        \limsup_{N\to\infty}
        \norm{
            \frac1N\sum_{n=0}^{N-1} B_n^{(m)}(\bar x)
            -
            \frac1N\sum_{n=0}^{N-1} R_n^{(m)}(\omega)
        }_1
        \le
        \delta_m.
    \end{equation}
    
    For this fixed \(m\), we also have
    \[
        \frac1N\sum_{n=0}^{N-1} R_n^{(m)}(\omega)
        =
        \frac1m\sum_{\ell=1}^m
        \frac1N\sum_{n=0}^{N-1}
        \rho_{\mathrm{ss}}(\theta^{n+\ell}\omega).
    \]
    For each fixed \(\ell\ge1\),
    \[
        \norm{
            \frac1N\sum_{n=0}^{N-1}\rho_{\mathrm{ss}}(\theta^{n+\ell}\omega)
            -
            \frac1N\sum_{n=1}^{N}\rho_{\mathrm{ss}}(\theta^n\omega)
        }_1
        \le
        \frac{2\ell}{N},
    \]
    because every state has trace norm \(1\). Together with \eqref{eq:birkhoff-rhoss-short}, this yields
    \begin{equation}
    \label{eq:R-conv-short-rewrite}
        \frac1N\sum_{n=0}^{N-1} R_n^{(m)}(\omega)
        \longrightarrow
        \bar\rho_{\mathrm{ss}}
        \qquad
        \text{in trace norm.}
    \end{equation}
    
    Furthermore,
    \[
        \frac1N\sum_{n=0}^{N-1} B_n^{(m)}(\bar x)
        =
        \frac1m\sum_{\ell=1}^m
        \frac1N\sum_{n=0}^{N-1}\rho_{n+\ell}(\bar x).
    \]
    Hence
    \begin{align}
    \label{eq:orig-vs-block-short-rewrite}
        \norm{
            \frac1N\sum_{j=1}^{N}\rho_j(\bar x)
            -
            \frac1N\sum_{n=0}^{N-1} B_n^{(m)}(\bar x)
        }_1
        \le
        \frac1m\sum_{\ell=1}^m
        \norm{
            \frac1N\sum_{j=1}^{N}\rho_j(\bar x)
            -
            \frac1N\sum_{n=0}^{N-1}\rho_{n+\ell}(\bar x)
        }_1 
        \le
        \frac{2m}{N},
    \end{align}
    again because every posterior state has trace norm \(1\).
    
    Since \(\bar x\in E_\omega\), the bounds \eqref{eq:block-to-ref-short-rewrite}, \eqref{eq:R-conv-short-rewrite}, and \eqref{eq:orig-vs-block-short-rewrite} imply that for every \(m\in\mbN\),
    \[
        \limsup_{N\to\infty}
        \norm{
            \frac1N\sum_{n=1}^{N}\rho_n(\bar x)
            -
            \bar\rho_{\mathrm{ss}}
        }_1
        \le
        \delta_m.
    \]
    Letting \(m\to\infty\) and using \eqref{eq:delta-m-short-rewrite}, we conclude that
    \[
        \lim_{N\to\infty}
        \norm{
            \frac1N\sum_{n=1}^{N}\rho_n^{\omega,\vartheta}(\bar x)
            -
            \bar\rho_{\mathrm{ss}}
        }_1
        =0
    \]
    for every \(\bar x\in E_\omega\). Since \(E_\omega\) has full \(\mbQ_{\vartheta;\omega}\)-measure, the proof is complete.
\end{proof}


\section{Quenched CLT and Berry--Esseen theorems}
\label{section:CLT_with_rates}


We use the state-path viewpoint from \Cref{section:prelim}. 
For a measurable initial state \(\vartheta:\Omega\to\states\), the quenched law \(\mbK_{\vartheta;\omega}\in\mcP(\states^{\mbN_0})\) is the law of the posterior-state chain with initial state \(\vartheta(\omega)\) and one-step transition kernels
\[
    K_\omega,\ K_{\theta\omega},\ K_{\theta^2\omega},\dots.
\]
Let
\[
    e_n:\states^{\mbN_0}\to\states,
    \qquad
    e_n\left((\rho_0,\rho_1,\rho_2,\dots)\right):=\rho_n,
    \qquad n\in\mbN_0,
\]
denote the coordinate maps.
Let \(f:\Omega\times\states\to\mbR\) be bounded and measurable, and write \(f_\omega(\rho):=f(\omega,\rho)\).
On \(\states^{\mbN_0}\), define
\[
    S_N^\omega f
    :=
    \sum_{n=0}^{N-1} f_{\theta^n\omega}(e_n),
    \qquad
    \bar S_N^\omega f
    :=
    S_N^\omega f
    -
    \mbE_{\mbK_{\vartheta;\omega}}[S_N^\omega f].
\]


\subsection{Proof of \texorpdfstring{\Cref{thm:quenched_clt_posterior}}{Theorem~3}}
\label{section:Theorem_CLT}


For the proof of \Cref{thm:quenched_clt_posterior}, we work on the measurable $\theta$-invariant set $\Omega_*$ of full $\pr$-measure constructed in \Cref{lem:HW-hypotheses-posterior} below.
Fix $\omega\in\Omega_*$.
Under \(\mbK_{\vartheta;\omega}\), the coordinate process \((e_n)_{n\ge0}\) is a time-inhomogeneous Markov chain on \(\states\) with successive transition kernels
\[
    K_\omega,\ K_{\theta\omega},\ K_{\theta^2\omega},\dots.
\]
We therefore place ourselves in the sequential framework of \cite[Theorem~2.2]{HW25}
by setting
\[
    X_j:=\states,
    \qquad
    P_j:=K_{\theta^j\omega},
    \qquad
    \tilde f_j:=f_{\theta^j\omega},
    \qquad j\ge0.
\]

\begin{lemma}
\label{lem:HW-hypotheses-posterior}
    Make \Cref{assumption_1,assumption_doeblin}.
    Let $f:\Omega\times\states\to\mbR$ be bounded and measurable, and let $\vartheta:\Omega\to\states$ be a measurable initial state.
    There exists a measurable $\theta$-invariant set $\Omega_*\subseteq\Omega$ of full $\pr$-measure such that, for every $\omega\in\Omega_*$, the posterior-state chain under $\mbK_{\vartheta;\omega}$ satisfies the boundedness and exponential \(\phi\)-mixing hypotheses of \cite[Theorem~2.2]{HW25}. 
    Consequently, \cite[Theorem~2.2]{HW25} applies whenever
    \[
        \sigma_N(\omega)
        :=\sqrt{\Var_{\mbK_{\vartheta;\omega}}
            \left(S_N^\omega f\right)}
        \longrightarrow\infty.
    \]
\end{lemma}
\begin{proof}
    Let $\Omega_D\in\mcF$ be a set of full $\pr$-measure on which the minorization in \Cref{assumption_doeblin} holds simultaneously for every $\rho\in\states$ and $B\in\borel{\states}$.
    Set
    \[
        \Omega_*:=\bigcap_{k\in\mbZ}\theta^{-k}\Omega_D.
    \]
    By \Cref{assumption_1}, this set is measurable, $\theta$-invariant, and of full $\pr$-measure.
    Fix $\omega\in\Omega_*$.

    By construction, under $\mbK_{\vartheta;\omega}$ the coordinate process is a Markov chain on $\states$ with successive transition kernels $P_j=K_{\theta^j\omega}$, $j\ge0$.
    Its $L$-step transition kernel from time $j$ is $K_{\theta^j\omega}^{(L)}$.
    Since $\theta^j\omega\in\Omega_D$, for every $j\ge0$,
    \[
        K_{\theta^j\omega}^{(L)}(\rho,B)
        \ge\varepsilon\lambda_{\theta^{j+L}\omega}(B),
        \qquad
        \rho\in\states,\quad B\in\borel{\states}.
    \]
    This is the uniform minorization in \cite[Section~2.1]{HW25}, with $n_0=L$, $\gamma=\varepsilon$, and  $m_j=\lambda_{\theta^j\omega}$.
    Moreover, the functions $\tilde f_j(\rho)=f(\theta^j\omega,\rho)$ are real and measurable, and satisfy
    \[
        |\tilde f_j(\rho)|\le\|f\|_\infty,
        \qquad j\ge0,\quad\rho\in\states.
    \]
    
    Let $\phi_\omega(n)$ denote the $\phi$-mixing coefficients of the coordinate process under $\mbK_{\vartheta;\omega}$.
    By \cite[Lemma~2.4]{HW25}, the uniform $L$-step minorization established above implies
    \[
        \phi_\omega(n)\le C\delta^n,
        \qquad
        \delta=(1-\varepsilon)^{1/L}\in(0,1),
    \]
    where $C$ depends only on $L$ and $\varepsilon$.
    Thus the exponential $\phi$-mixing hypothesis of \cite[Theorem~2.2]{HW25} is satisfied.

    The additional hypothesis of \cite[Theorem~2.2]{HW25} is precisely $\sigma_N(\omega)\to\infty$.
\end{proof}

\quenchedcltposterior*

\begin{proof}
    We first establish a uniform estimate for the posterior kernels' convergence to the equivariant family.
    Let $(\mu_\omega)$ be the equivariant family from \Cref{thm:s_exists}.
    Intersecting the invariant full-measure set $\Omega_*$ from \Cref{lem:HW-hypotheses-posterior} with all integer translates of a measurable full-measure set on which equivariance holds, we may assume that \(\mu_{\theta^j\omega}K_{\theta^j\omega} = \mu_{\theta^{j+1}\omega}\) for all \(\omega\in\Omega_*\) and \(j\in\mbZ\). 
 
    For a probability kernel $Q$ on $\states$ and a bounded measurable function $g:\states\to\mbC$, write
    \[
        (Qg)(\rho):=\int_{\states}g(\eta)\,Q(\rho,\dee\eta).
    \]
    We claim that, for every $\omega\in\Omega_*$, every $n\ge1$, and every such $g$,
    \begin{equation}
    \label{eq:quenched-forgetting-composition}
        \left\|
            \left(K_{\theta^{n-1}\omega}\circ\cdots\circ K_\omega\right)g
            -\mu_{\theta^n\omega}(g)\mathbf1
        \right\|_\infty
        \le
        2(1-\varepsilon)^{\lfloor n/L\rfloor}\|g\|_\infty.
    \end{equation}
    Indeed, for every $\xi\in\Omega_*$, the Doeblin minorization gives a probability kernel $R_\xi$ such that
    \[
        K_\xi^{(L)}(\rho,\,\cdot\,)
        =\varepsilon\lambda_{\theta^L\xi}
         +(1-\varepsilon)R_\xi(\rho,\,\cdot\,).
    \]
    Consequently, for probability measures $\alpha,\beta$ on $\states$, the common minorizing term cancels and
    \[
        \|\alpha K_\xi^{(L)}-\beta K_\xi^{(L)}\|_{\mathrm{TV}}
        =(1-\varepsilon)\|\alpha R_\xi-\beta R_\xi\|_{\mathrm{TV}}
        \le(1-\varepsilon)\|\alpha-\beta\|_{\mathrm{TV}}.
    \]
    Write $n=qL+r$, where $q=\lfloor n/L\rfloor$ and $0\le r<L$.
    Applying this contraction to the $q$ complete blocks and using the nonexpansiveness of the remaining $r$ Markov transitions yields
    \[
        \|\alpha K_\omega^{(n)}-\beta K_\omega^{(n)}\|_{\mathrm{TV}}
        \le(1-\varepsilon)^q\|\alpha-\beta\|_{\mathrm{TV}}.
    \]
    By equivariance, \(\mu_\omega K_\omega^{(n)}=\mu_{\theta^n\omega}\), and taking $\alpha=\delta_\rho$ and $\beta=\mu_\omega$, we obtain
    \[
        \|\delta_\rho K_\omega^{(n)}
           -\mu_{\theta^n\omega}\|_{\mathrm{TV}}
        \le2(1-\varepsilon)^q.
    \]
    It follows that
    \[
    \begin{aligned}
        \left|(K_\omega^{(n)}g)(\rho)
              -\mu_{\theta^n\omega}(g)\right|
        \le\|g\|_\infty
           \|\delta_\rho K_\omega^{(n)}
              -\mu_{\theta^n\omega}\|_{\mathrm{TV}}
        \le2(1-\varepsilon)^q\|g\|_\infty.
    \end{aligned}
    \]
    Taking the supremum over $\rho\in\states$ proves \eqref{eq:quenched-forgetting-composition}.
    In particular, this gives the geometric estimate
    \[
        \left\|
            \left(K_{\theta^{n-1}\omega}\circ\cdots\circ K_\omega\right)g
            -\mu_{\theta^n\omega}(g)\mathbf1
        \right\|_\infty
        \le C\|g\|_\infty\delta^n,
    \]
    where we take
    \[
        C=\frac{2}{1-\varepsilon},
        \qquad
        \delta=(1-\varepsilon)^{1/L}\in(0,1).
    \]

    \medskip

    We next apply \cite[Theorem~D.2]{DH_PTRF}.
    Let $B_b(\states)$ denote the Banach space of bounded Borel measurable complex functions, equipped with the supremum norm.
    Define the finite complex kernel 
    \begin{equation}
    \label{eq:new_kernel}
        K_{\omega,z}(\rho,\dee\eta):= e^{zf_{\theta \omega}}(\eta)K_\omega(\rho,\dee\eta)
    \end{equation}
    so that \(K_{\omega,z}g = K_\omega(e^{zf_{\theta\omega}}g)\). 
    Denote
    \[
        K^{(n)}_{\omega,z}:=
            K_{\theta^{n-1}\omega,z} \circ \cdots \circ K_{\omega,z},
            \qquad
            K_{\omega,z}^{(0)}:=\mathrm{Id}.
    \]
    
    Now, let $B_b(\states)$ denote the Banach space of bounded Borel measurable complex functions, equipped with the supremum norm.
    Fix $\omega\in \Omega_*$ and for each $j\in\mbZ$ take
    \[
        B_j:=B_b(\states),
        \qquad
        A_j(z):=K_{\theta^{-j-1}\omega,z},
        \qquad 
        m_j:=\mu_{\theta^{-j}\omega}.
    \]
    We now verify the assumptions D1 in \cite[Appendix~D]{DH_PTRF}. 
    We first note that in the notation of \cite[Appendix~D]{DH_PTRF}, $m_j(g) = \nu_j(g)$, but we have use $\nu$ elsewhere. 
    First, note that each $B_j$ has dimension greater than $1$.
    Since $K$ are kernels, we also have that
    \[
        \norm{A_j}_{\infty \to\infty} = \norm{m_j}_{B^*_j}
        = 1\,
        \qquad m_j(\mathbf 1) = 1.
    \]
    Moreover,
    \[
        A_j\mathbf1=\mathbf1=\lambda_jh_{j+1}.
    \]
    Furthermore, using equivariance property, for each $g\in B_b(\states)$ we see that 
    \[
        m_{j+1}(A_jg)
        =
        \int_{\states}
        (K_{\theta^{-j-1}\omega}g)(\rho)\,
        \mu_{\theta^{-j-1}\omega}(\dee\rho)
        =
        \mu_{\theta^{-j}\omega}(g) = m_j(g).
    \]
    This verifies that \(A_j^*m_{j+1}=\lambda_jm_j\).
    The Condition~(i) in \cite[Assumption~D1]{DH_PTRF} holds trivially (in the notion of \cite[Appendix~D]{DH_PTRF} $\lambda_j = 1$ for each $j$). 
    For Condition~(ii), suppose that $m_j(g)=0$, then
    \[
        A_{j+n-1}\cdots A_jg
        =K_{\theta^{-j-n}\omega}^{(n)}g.
    \]
    Therefore, by the estimate above
    \[
        \|A_{j+n-1}\cdots A_jg\|_\infty
        =
         \|K_{\theta^{-j-n}\omega}^{(n)}g
      -\mu_{\theta^{-j}\omega}(g)\mathbf1\|_\infty
      \le 
      C\delta^n\|g\|_\infty. 
    \]

    As for condition~(iii), take the parameter space to be $\mbC$ and write
    \[
        M:=\sup_{\xi\in\Omega,\,\rho\in\states}
              |f(\xi,\rho)|<\infty.
    \]
    Then using 
    \[
        K_{\xi,z}
        =\sum_{k=0}^{\infty}\frac{z^k}{k!}
           K_\xi(f_{\theta\xi}^{\,k}\,\,\cdot\,),\qquad
        \|K_\xi(f_{\theta\xi}^{\,k}\,\,\cdot\,)\|_{\infty\to\infty}
        \le M^k,
    \]
    we see that the series converges in operator norm, uniformly in $\xi\in\Omega_*$ and on every bounded disc in $z$.
    Therefore $z\mapsto K_{\xi,z}$ is entire in operator norm, and
    \[
        \|K_{\xi,z}\|_{\infty\to\infty}\le e^{|z|M},
        \qquad
        \|K_{\xi,z}-K_\xi\|_{\infty\to\infty}
        \le e^{|z|M}-1.
    \]
    For any fixed $r>0$, the common neighborhood $U:=\{z\in\mbC:|z|<r\}$ consequently satisfies
    \[
        \sup_{j\in\mbZ}\sup_{z\in U}
        \|A_j(z)\|_{\infty\to\infty}
        \le e^{rM}<\infty.
    \]
    If $m_j(g) =0$, the estimate above gives us that 
    \[
        \|A_{j+n-1}(0)\cdots A_j(0)g\|_\infty
        = 
        \|K_{\theta^{-j-n}\omega}^{(n)}g\|_\infty
        \le
        C \delta^n \norm{g}_\infty.
    \]
    The exponential series shows that $z\mapsto K_{\omega,z}$ is entire in operator norm, and
    \[
        \|K_{\omega,z}\|_{\infty\to\infty}\le e^{|z|M},
        \qquad
        \|K_{\omega,z}-K_\omega\|_{\infty\to\infty}
        \le e^{|z|M}-1,
    \]
    uniformly in $\omega\in\Omega_*$, where
    \[
        M:=\sup_{\omega,\rho}|f(\omega,\rho)|<\infty.
    \]
    Since all these bounds are independent of the fixed $\omega\in\Omega_*$, all three conditions in \cite[Assumption~D1]{DH_PTRF} are satisfied. 

     \smallskip

     Now we apply \cite[Theorem~D.2]{DH_PTRF}.
     Choose $\delta_0 \in (\delta,1)$, by \cite[Theorem~D.2]{DH_PTRF} there are constants  $r_0,C_0>0$ and analytic families
     \[
        \lambda_j(z)\in\mbC\setminus\{0\},
        \qquad
        h_j^{(z)}\in B_b(\states),
        \qquad
        \nu_j^{(z)}\in B_b(\states)^*,
        \qquad j\in\mbZ,
    \]
    that are analytic on a common neighborhood of $\{z:|z|\le r_0\}$ and uniformly bounded:
    \[
        \sup_{j\in\mbZ,\,|z|\le r_0}
        \left(
            |\lambda_j(z)|
            +\|h_j^{(z)}\|_\infty
            +\|\nu_j^{(z)}\|_{B_b(\states)^*}
        \right)<\infty.
    \]
    We note that these constants and bounds can be chosen independently of $\omega\in\Omega_*$, since the bounds entering the perturbation argument are uniform in $\omega$.
    At $z = 0$ they satisfy
    \[
        \lambda_j(0)=1,\qquad
        h_j^{(0)}=\mathbf1,\qquad
        \nu_j^{(0)}(g)=m_j(g), \qquad \nu_j^{(z)}(\mathbf1) =\nu_j^{(z)}(h_j^{(z)})=1.
    \]
    Moreover, for every $g\in B_b(\states)$,
    \[
        A_j(z)h_j^{(z)}
        =\lambda_j(z)h_{j+1}^{(z)},
    \]
    and
    \[
        \nu_{j+1}^{(z)}(A_j(z)g)
        =\lambda_j(z)\nu_j^{(z)}(g).
    \]
    For every $j\in\mbZ$, $n\ge1$, $|z|\le r_0$, and $g\in B_b(\states)$, we have
    \[
        \left\|
            A_{j+n-1}(z)\cdots A_j(z)g
            -\left(\prod_{k=j}^{j+n-1}\lambda_k(z)\right)
             \nu_j^{(z)}(g)h_{j+n}^{(z)}
        \right\|_\infty
        \le C_0\delta_0^n\|g\|_\infty.
    \]
    In particular, taking $j=-n$ and recalling the definition
    of $A_j(z)$ gives
    \[
        \left\|
            K_{\omega,z}^{(n)}g
            -\left(\prod_{k=-n}^{-1}\lambda_k(z)\right)
             \nu_{-n}^{(z)}(g)h_0^{(z)}
        \right\|_\infty
        \le C_0\delta_0^n\|g\|_\infty.
    \]

    \smallskip

    We next identify the scalar factors as measurable functions of the environment.
    Choose $a\in(\delta_0,1)$. By uniform analyticity and the values at zero, we may decrease $r_0$ so that
    \[
        |\lambda_j(z)|\ge a,
        \qquad
        \|h_j^{(z)}-\mathbf1\|_\infty\le\tfrac12
    \]
    uniformly in $\omega\in\Omega_*$, $j\in\mbZ$,
    and $|z|\le r_0$.
    Fix $\rho_*\in\states$. The preceding spectral estimate, applied with $g=\mathbf1$, gives
    \begin{equation}
    \label{eq:quenched-measurable-eigenvalue}
        \lambda_\omega(z)
        :=\lim_{m\to\infty}
        \frac{[K_{\theta^{-m}\omega,z}^{(m+1)}
                    \mathbf1](\rho_*)}
             {[K_{\theta^{-m}\omega,z}^{(m)}
                    \mathbf1](\rho_*)}
        =\lambda_{-1}(z).
    \end{equation}
    Indeed, the two leading terms have the same eigenfunction factor, the denominator has modulus at least $a^m/4$ for all sufficiently large $m$, and the ratio converges uniformly with error $O((\delta_0/a)^m)$.
    Each ratio is measurable in $\omega$ and analytic in $z$.
    Thus $\lambda_\omega(z)$ is measurable in $\omega$ and analytic in $z$, with uniform bounds.
    Applying the same formula at $\theta^k\omega$ gives
    \[
        \lambda_{\theta^k\omega}(z)=\lambda_{-k-1}(z),
        \qquad k\in\mbZ,
    \]
    where the right-hand side uses the spectral data along the fixed orbit of $\omega$.
    Take
    \[
        \Lambda_{\omega,n}(z)
        :=\prod_{k=0}^{n-1}\lambda_{\theta^k\omega}(z),
        \qquad \Lambda_{\omega,0}(z):=1.
    \]
    Since $\nu_{-n}^{(z)}(\mathbf1)=1$, the spectral estimate  also gives the uniform limit
    \[
        h_\omega^{(z)}
        :=\lim_{n\to\infty}
          \Lambda_{\omega,n}(z)^{-1}
          K_{\omega,z}^{(n)}\mathbf1
        =h_0^{(z)}.
    \]
    Consequently, $h_\omega^{(z)}$ is uniformly bounded and analytic in $z$, and
    \begin{equation}
    \label{eq:quenched-twisted-one}
        \left\|K_{\omega,z}^{(n)}\mathbf1
             -\Lambda_{\omega,n}(z)h_\omega^{(z)}
        \right\|_\infty
        \le C_0\delta_0^n,
        \qquad
        |\Lambda_{\omega,n}(z)|\ge a^n.
    \end{equation}

    \medskip

    Having observed above, we now follow the second-moment argument in \cite[Section~3, proof of Theorem~2.2(ii)]{YH_JSP}.
    Take $\rho_0:=\vartheta(\omega)$ and let $M_{\omega,N}(z) = \mbE_{\mbK_{\vartheta;\omega}}[e^{zS_N^\omega f}]$. 
    Now, by \eqref{eq:new_kernel} and by \Cref{prop:construction-of-bbK} we have that 
     \[
        M_{\omega,N}(z)
        =e^{zf_\omega(\rho_0)}
          [K_{\omega,z}^{(N-1)}\mathbf1](\rho_0).
    \]
    Define
    \[
        F_{\omega,N}(z)
        :=\frac{M_{\omega,N}(z)}{\Lambda_{\omega,N}(z)}.
    \]
    For $N\ge2$, \eqref{eq:quenched-twisted-one} gives
    \[
        F_{\omega,N}(z)
            =\frac{e^{zf_\omega(\rho_0)}}
                   {\lambda_{\theta^{N-1}\omega}(z)}
              \left(h_\omega^{(z)}(\rho_0)
                    +r_{\omega,N}(z)\right),
                \qquad
                |r_{\omega,N}(z)|
                \le C_0(\delta_0/a)^{N-1},
    \]
    with an analytic remainder and all bounds are uniform in $\omega$, $N$ and in $\rho_0$ on the common disc. 
    Therefore, $F_{\omega,N}$ is uniformly bounded.
    If $N=1$, the same conclusions also follow directly from 
    \[
        F_{\omega,1}(z)
        =\frac{e^{zf_\omega(\rho_0)}}{\lambda_\omega(z)}.
    \]
    Put $\pi_\omega(z):=\log\lambda_\omega(z)$, choosing the analytic branch with $\pi_\omega(0)=0$.
    Since $F_{\omega,N}(0)=1$, Cauchy's estimates give $F_{\omega,N}'(0),F_{\omega,N}''(0)=O(1)$ uniformly.
    Differentiating $M_{\omega,N}=\Lambda_{\omega,N}F_{\omega,N}$ twice at zero therefore yields
    \[
        \begin{aligned}
            \Var_{\mbK_{\vartheta;\omega}}(S_N^\omega f)
            &=M_{\omega,N}''(0)-\left(M_{\omega,N}'(0)\right)^2\\
            &=\sum_{j=0}^{N-1}\pi_{\theta^j\omega}''(0)
              +F_{\omega,N}''(0)-\left(F_{\omega,N}'(0)\right)^2\\
            &=\sum_{j=0}^{N-1}\pi_{\theta^j\omega}''(0)+O(1).
        \end{aligned}
    \]
    Since $\omega\mapsto\pi_\omega''(0)$ is measurable and uniformly bounded, Birkhoff's ergodic theorem gives
    \[
        \lim_{N\to\infty}\frac1N
        \Var_{\mbK_{\vartheta;\omega}}(S_N^\omega f)
        =\int_\Omega\pi_\omega''(0)\,\dee\pr(\omega)
        =:\Sigma^2
    \]
    for $\pr$-almost every $\omega$.
    This constant is real and nonnegative, being a limit of normalized variances.
    By ergodicity, $\Sigma$ is independent of $\omega$.
    Moreover, the scalar factors $\lambda_\omega(z)$ were constructed independently of $\vartheta$, and the error in the variance identity is uniform in the initial state $\rho_0$. 
    Hence, the same $\Sigma$ applies to every measurable initial state $\vartheta$.

    Finally, suppose $\Sigma>0$ and put $\sigma_N(\omega):= \|\bar S_N^\omega f\|_{L^2(\mbK_{\vartheta;\omega})}$.
    Then $\sigma_N(\omega)/\sqrt N\to\Sigma$. By \Cref{lem:HW-hypotheses-posterior} and \cite[Theorem~2.2(i)]{HW25}, the standardized  Berry--Esseen error is
    $O(\sigma_N(\omega)^{-1})=O(N^{-1/2})$.
    The stated CLT follows from this bound and $\sigma_N(\omega)/\sqrt N\to\Sigma$.
    
\end{proof}


\subsection{Proof of \texorpdfstring{\Cref{thm:quenched_berry_esseen_posterior}}{Theorem~4}}


Throughout this subsection, we work under the hypotheses of \Cref{thm:quenched_berry_esseen_posterior}.
We use the instrument process $(\mathbf I_{\theta^j\omega})_{j\in\mbZ}$ and its past and
future $\sigma$-algebras from \Cref{subsection:canonical-coding-instrument-process}.
In particular, \Cref{assumption_mix} gives
\[
    \sum_{n=1}^{\infty}\phi^{\mathrm{inst}}(n)<\infty,
\]
and $K_\omega$ and $f_\omega$ depend only on the current one-step instrument in the sense of \Cref{def:depends-only-current-instrument}.

First, by applying the results in \cite[Section 3.1]{YH_JSP}, we get the following result.

\begin{lemma}
\label{lem:quenched-variance-rate}
In the circumstances of \Cref{thm:quenched_berry_esseen_posterior}, there exists a  deterministic constant $\Sigma\ge0$ such that, for every measurable initial state $\vartheta$, for $\pr$-almost every $\omega$,
\[
    \left|
        \frac1N\Var_{\mbK_{\vartheta;\omega}}(S_N^\omega f)
        -\Sigma^2
        \right|
    =O(N^{-1/2}\log N).
\]
\end{lemma}

    \begin{proof}
        The spectral argument in the preceding proof of \Cref{thm:quenched_clt_posterior}, before its application of Birkhoff's theorem, does not require ergodicity.
        It gives a bounded real function $q(\omega):=\pi_\omega''(0)$ such that
        \[
            \Var_{\mbK_{\vartheta;\omega}}(S_N^\omega f)
            =\sum_{j=0}^{N-1}q(\theta^j\omega)+O(1),
        \]
        uniformly in the initial state.
        
        Now we adapt the approximation and martingale argument of \cite[Section~3.1]{YH_JSP}.
        Let $R_m(\omega,z)$ denote the finite ratio in \eqref{eq:quenched-measurable-eigenvalue}, and put $q_m(\omega):=(\log R_m(\omega,\cdot))''(0)$.
        By current-instrument dependence, $q_m$ is measurable with respect to $\sigma(\mcG_j^{\mathrm{inst}}:-m\le j\le1)$.
        The uniform convergence of the ratios and Cauchy's estimates give, for some $\tau\in(0,1)$,
        \[
            \|q-q_m\|_{L^\infty(\pr)}\le C\tau^m.
        \]
        
        Choose the $\mcF_1^{\mathrm{inst},-}$-measurable version of $q$ obtained from the approximations $q_m$.
        Put $Y_j:=q\circ\theta^{j-1}-\mbE_{\pr} q$, for $j\in\mbZ$.
        Then $Y_j$ is $\mcF_j^{\mathrm{inst},-}$-measurable.
        For $k\ge3$, take $m=\lfloor(k-1)/2\rfloor$.
        The preceding approximation and the conditional-expectation bound for $\phi$-mixing give
        \[
            \norm{
                \mbE_{\pr}[Y_k\mid\mcF_0^{\mathrm{inst},-}]
                }_{L^\infty(\pr)}
            \le C\left(\tau^m+\phi^{\mathrm{inst}}(k-1-m)\right).
        \]
        Hence these norms are summable by \Cref{assumption_mix}.
        Consequently,
        \[
            U_j:=\sum_{k\ge1}
                \mbE_{\pr}[Y_{j+k}\mid\mcF_j^{\mathrm{inst},-}],
            \qquad
            D_j:=Y_j+U_j-U_{j-1}
        \]
        are uniformly bounded, and $(D_j)$ is a martingale-difference sequence with respect to $(\mcF_j^{\mathrm{inst},-})_{j\in\mbZ}$.
        Since
        \[
            \sum_{j=1}^N Y_j
            =\sum_{j=1}^N D_j+U_0-U_N,
        \]
        Azuma--Hoeffding's inequality and the Borel--Cantelli lemma yield
        \[
            \sum_{j=0}^{N-1}q(\theta^j\omega)-N\mbE_{\pr} q
            =O(\sqrt N\log N)
        \]
        almost surely.
        Combining this with the variance identity proves the claim with $\Sigma^2:=\mbE_{\pr} q$.
        This constant is independent of the initial state and is nonnegative, being a limit of normalized variances.
    \end{proof}

Next, we need the following lemma. 

\begin{lemma}
\label{lem:change-of-normalization}
    Let $(S_N)$ be a centered sequence of random variables such that $\sigma_N:=\norm{S_N}_{L^2(\mbP)}\to \infty$. Suppose that $|\sigma_N^2/N-1|\leq CN^{-1/2}\log(N)$ and that 
    $$
    \sup_{t\in\mbR}\left|F_N(t)-\Phi(t)\right|=O(N^{-1/2})
    $$
    where $F_N(t)=\mbP(S_N/\sigma_N\leq t)$. Set $G_N(t)=\mbP(S_N/\sqrt N\leq t)$. Then 
    $$
    \sup_{t\in\mbR}\left|G_N(t)-\Phi(t)\right|=O(N^{-1/2}(\log(N))^{3/2}).
    $$
\end{lemma}
    \begin{proof}
    We have 
    $$
    \Phi(t\sqrt N/\sigma_N)-CN^{-1/2}
    \leq G_N(t)
    =F_N(t\sqrt N/\sigma_N)\leq \Phi(t\sqrt N/\sigma_N)+CN^{-1/2}.
    $$
    Thus, it is enough to show that 
    $$
    \sup_{t\in\mbR}|\Phi(t\sqrt N/\sigma_N)-\Phi(t)|=O(N^{-1/2}(\log(N))^{3/2}).
    $$

    Since $\sigma_N^2/N\to1$, we have $\sqrt N/\sigma_N\geq1/2$ for all sufficiently large $N$.
    If $|t|\geq c(\log(N))^{1/2}$, then
    \[
        |\Phi(t\sqrt N/\sigma_N)-\Phi(t)|
        \leq 1-\Phi(|t|/2)
        \leq C e^{-t^2/8}
        \leq C N^{-c^2/8}.
    \]
    Hence the required estimate holds if $c$ is sufficiently large.
    If $|t|< c(\log(N))^{1/2}$ then we take $u_N=\frac{N}{\sigma_N^2}$.
    Since $\sigma_N^2/N\to1$, the assumed variance estimate gives $|u_N-1|\leq CN^{-1/2}\log N$ for all sufficiently large $N$.
    Therefore, 
    $$
    |\Phi(t\sqrt N/\sigma_N)-\Phi(t)|\leq |t||\sqrt N/\sigma_N-1|=|t|\left|\frac{u_N-1}{\sqrt{u_N}+1}\right|\leq C|t|N^{-1/2}\log(N).
    $$
    This completes the proof.
\end{proof}

\Cref{thm:quenched_berry_esseen_posterior} follows by combining the above results.

\quenchedberryesseenposterior*
\begin{proof}
    By \Cref{lem:quenched-variance-rate}, the variance limit exists and equals a deterministic constant $\Sigma^2$,
    independent of the initial state.
    Fix a measurable initial state $\vartheta$ and work on  a common measurable set $\Omega_0$ of full $\pr$-measure
    on which \Cref{lem:quenched-variance-rate,lem:HW-hypotheses-posterior} hold.
    Suppose $\Sigma>0$, and put $T_N:=\bar S_N^\omega f/\Sigma$.
    The first lemma above gives
    \[
        \left|
        \frac{\norm{T_N}_{L^2(\mbK_{\vartheta;\omega})}^2}{N}-1
        \right|
        =O(N^{-1/2}\log N),
    \]
    and hence $\norm{T_N}_{L^2(\mbK_{\vartheta;\omega})}\to\infty$.
    By \Cref{lem:HW-hypotheses-posterior} and \cite[Theorem~2.2(i)]{HW25}, the standardized
    Berry--Esseen error for $T_N$ is $O(N^{-1/2})$.
    Applying \Cref{lem:change-of-normalization} to $T_N$  under $\mbK_{\vartheta;\omega}$ therefore gives the asserted bound.

    Whenever \Cref{thm:quenched_clt_posterior} applies,  the constants $\Sigma$ coincide by uniqueness of the variance limit.
\end{proof}


\section{Annealed limit theorems: proof of \texorpdfstring{\Cref{thm:annealed_alpha_mixing}}{Theorem~5}}

    Let $(\mu_\omega)_{\omega\in\Omega}$ be the measurable equivariant family of probability measures from \Cref{thm:s_exists}. Thus $\mu_\omega K_\omega=\mu_{\theta\omega}$, $\pr$-almost surely. 
    Recall that $\kappa=\bar\kappa\circ\Pi^{-1}$, where $\Pi(\omega,y)=(\iota(\omega),y)$.
    We therefore compute probabilities of coded events using the original base and the quenched laws $\bar\kappa_\omega$.

    By \eqref{eq:quenched-forgetting-composition}, there exist constants $C>0$ and $\delta\in(0,1)$ such that
    \[
        \left\|
            K_\omega^{(n)}g-\mu_{\theta^n\omega}(g)\mathbf1
        \right\|_\infty
        \le C\delta^n\|g\|_\infty
    \]
    for every $n\ge1$, every $\omega\in\Omega_*$, and every bounded measurable function $g:\states\to\mbC$.

    Fix $\rho_*\in\states$ and put
    \[
        \mu_\omega^{(r)}
        :=\delta_{\rho_*}K_{\theta^{-r}\omega}^{(r)}.
    \]
    Applying the preceding estimate at $\theta^{-r}\omega$ gives
    \begin{equation}
    \label{mu approx}
        \|\mu_\omega^{(r)}-\mu_\omega\|_{\mathrm{TV}}
        \le C\delta^r.
    \end{equation}
    By the current-instrument dependence of $K_\omega$, the probability measure $\mu_\omega^{(r)}$ depends only on the instruments at times $-r,\ldots,-1$.
    Thus \eqref{mu approx} approximates $\mu_\omega$ by a function of a finite block of past instruments.

    For $k\le\ell$, denote by $\kappa_{\omega,k,\ell}$ the marginal of $\bar\kappa_\omega$ on state coordinates $k,\ldots,\ell$.
    Its finite-dimensional distribution is
    \[
        \kappa_{\omega,k,\ell}
           (\dee\rho_k,\ldots,\dee\rho_\ell)
        =
        \mu_{\theta^k\omega}(\dee\rho_k)
        \prod_{j=k}^{\ell-1}
           K_{\theta^j\omega}(\rho_j,\dee\rho_{j+1}).
    \]
    In particular, under the natural identification of coordinates,
    \[
        \kappa_{\omega,k,\ell}
        =\kappa_{\theta^k\omega,0,\ell-k}.
    \]

    For $A\in\mcH_{k,\ell}$, put
    \[
        A_\omega:=\{y\in\mcY:(\iota(\omega),y)\in A\},
        \qquad
        p_A(\omega):=\bar\kappa_\omega(A_\omega).
    \]
    Then $0\le p_A\le1$ and
    $\kappa(A)=\mbE_{\pr}p_A$.
    For integers $a\le b$, also write
    \[
        \mcF_{a,b}^{\mathrm{inst}}
        :=\sigma(\mcG_j^{\mathrm{inst}}:a\le j\le b).
    \]
    
    Using \eqref{mu approx}, we obtain the following result.

    \begin{lemma}
    \label{ApLemma}
        There exist constants $C>0$ and $\eta\in(0,1)$ such that, for every $r\ge1$,
        \[
            \sup_{k\le\ell}\sup_{A\in\mcH_{k,\ell}}
            \left\|
                p_A-\mbE_{\pr}\!\left[
                    p_A\mid\mcF_{k-r,\ell+r}^{\mathrm{inst}}
                \right]
            \right\|_{L^\infty(\pr)}
            \le C\eta^r.
        \]
\end{lemma}

\begin{proof}
    Fix  $k\le\ell$, $A\in\mcH_{k,\ell}$, and $r\ge1$.
    Define the probability measure
    \[
        \kappa_{\omega,k,\ell}^{(r)}
            (\dee\rho_k,\ldots,\dee\rho_\ell)
        :=
        \mu_{\theta^k\omega}^{(r)}(\dee\rho_k)
        \prod_{j=k}^{\ell-1}
            K_{\theta^j\omega}(\rho_j,\dee\rho_{j+1}),
    \]
    where $\mu_{\theta^k\omega}^{(r)}=\delta_{\rho_*}K_{\theta^{k-r}\omega}^{(r)}$.
    Regarding $A_\omega$, defined above, as an event on the state coordinates $k,\ldots ,\ell$ we define $p_{A,r}(\omega) :=\kappa_{\omega,k,\ell}^{(r)}(A_\omega)$. 

    For $\omega\in\Omega$, set
    \[
        h_{A,\omega}(\rho_k)
            :=
            \int_{\states^{\ell-k}}
                \mathbf1_{A_\omega}(\rho_k,\ldots,\rho_\ell)
                \prod_{j=k}^{\ell-1}
                    K_{\theta^j\omega}(\rho_j,\dee\rho_{j+1}).
    \]
    Then we have that $0\le h_{A,\omega}\le 1$ and that 
    \[
        p_A(\omega)
            =\int_{\states}h_{A,\omega}\,\dee\mu_{\theta^k\omega},
            \qquad
            \text{ and }
            \qquad 
         p_{A,r}(\omega)
            =\int_{\states}h_{A,\omega}\,
                \dee\mu_{\theta^k\omega}^{(r)}.
    \]
    Thus, by \eqref{mu approx} we see that 
    \begin{align*}
        \left|p_A(\omega)-p_{A,r}(\omega)\right|
        & = \left|
                \int_{\states}h_{A,\omega}\,
                    \dee\left(\mu_{\theta^k\omega}
                  -\mu_{\theta^k\omega}^{(r)}\right)
            \right|
        \\
        &\le 
            \norm{
                \mu_{\theta^k\omega}-\mu_{\theta^k\omega}^{(r)}
            }_{\mathrm{TV}}\le C\delta^r,
    \end{align*}
    for $\pr$-almost every $\omega$.
    Therefore, 
    \[
        \norm{p_A-p_{A,r}}_{L^\infty(\pr)}
        \le C\delta^r.
    \]
    This bound is uniform in the block length and in $A$, because $0\le h_{A,\omega}\le1$ in every case.

    By current-instrument dependence, $p_{A,r}$ is measurable with respect to $\mcF_{k-r,\ell}^{\mathrm{inst}}$:
    the transitions use instruments from $k-r$ through $\ell-1$, and the event $A$ uses instrument coordinates from $k$ through $\ell$.
    In particular, $p_{A,r}$ is $\mcF_{k-r,\ell+r}^{\mathrm{inst}}$-measurable.
    Hence
    \[
    \begin{aligned}
        \left\|
            p_A-\mbE_{\pr}\!\left[
                p_A\mid\mcF_{k-r,\ell+r}^{\mathrm{inst}}
            \right]
         \right\|_{L^\infty(\pr)}
        &\le
        \|p_A-p_{A,r}\|_{L^\infty(\pr)}
        +
        \left\|
            \mbE_{\pr}\!\left[
                p_{A,r}-p_A
                \mid\mcF_{k-r,\ell+r}^{\mathrm{inst}}
            \right]
        \right\|_{L^\infty(\pr)}
        \\
        &\le 2C\delta^r.
    \end{aligned}
    \]
    This proves the claim, with $\eta=\delta$ and a suitable constant $C$.
\end{proof}

We now prove \Cref{thm:annealed_alpha_mixing}.

\annealedalphamixing*

\begin{proof}
    First, take finite-block events $A\in\mcH_{k,\ell}$ and $B\in\mcH_{k_1,\ell_1}$, where
    $k_1\ge\ell+m$.
    Integrating \eqref{eq:quenched-forgetting-composition} over  $\omega$, we obtain
    \[
        \left|
        \kappa(A\cap B)-\mbE_{\pr}[p_Ap_B]
        \right|
        \le C\delta^m\kappa(A).
    \]
    On the other hand, for $m\ge3$ and taking $r=\lfloor m/3\rfloor$. 
    Applying \Cref{ApLemma} twice, and using the definition of $\alpha^{\mathrm{inst}}$, we obtain
    \[
        \left|
            \mbE_{\pr}[p_Ap_B]-\kappa(A)\kappa(B)
        \right|
        \le 2C\eta^r+\alpha^{\mathrm{inst}}(r).
    \]
    Combining these estimates gives
    \[
        |\kappa(A\cap B)-\kappa(A)\kappa(B)|
        \le
        C\delta^m+2C\eta^{\lfloor m/3\rfloor}
        +\alpha^{\mathrm{inst}}(\lfloor m/3\rfloor).
    \]
    
    Now take $\zeta:=\max\{\delta,\eta^{1/3}\}\in(0,1)$.
    The first two terms are bounded by $C_1\zeta^m$  for a suitable constant $C_1$.
    A monotone-class argument extends the above estimate to $A\in\mcH_{-\infty,\ell}$ and $B\in\mcH_{\ell+m,\infty}$.
    Taking suprema proves the assertion for $m\ge3$, and enlarging $C_1$ handles $m=1,2$.
\end{proof}


\section*{Acknowledgments}
 LP acknowledges support from Villum Fonden Grant No. 25452 and Grant No. 60842, as well as QMATH Center of Excellence Grant No. 10059. LP was also supported by the Danish e-infrastructure Consortium (DeiC) grant 5260-00014B. 


\begin{appendix}


\section{Examples}
\label{Section:examples}


    In this section, we provide examples of disordered instruments satisfying \Cref{assumption_doeblin}. 
    The conditions in \Cref{assumption_1,assumption_mix} hold, for example, in a two-sided i.i.d.\ environment when each instrument is a measurable function of the current coordinate. 
    We therefore focus on verifying \Cref{assumption_doeblin}. 
    The periodic regeneration models are treated separately and do not satisfy \Cref{assumption_mix}.
    We organize the examples into broad classes and present special cases of each. 
    Some examples are adapted from \cite{luba_clt}.

    Throughout, we write the Doeblin minorization as
    \[
        K_\omega^{(L)}(\rho,\,\cdot\,)\ge b\,\widetilde\lambda_\omega.
    \]
    Here $L\ge1$ is the block length, $b\in(0,1]$ is the minorization mass, and $\widetilde\lambda_\omega$ is a probability measure independent of the initial state $\rho$.
    Both $L$ and $b$ are independent of $\omega$.
    This gives \Cref{assumption_doeblin} by setting \(\lambda_\omega=\widetilde\lambda_{\theta^{-L}\omega}\) and choosing $0<\varepsilon<\min\{b,1\}$.
    All parameter families, including the family $\widetilde\lambda_\omega$, are assumed measurable, and environmental conditions are understood to hold almost surely.


\subsection{Basis transition dynamics}


    Recall that $d\ge2$ and let $(e_i)_{i=1}^d$ be an orthonormal basis.
    Take $\rho^{(i)} = \ket{e_i}\bra{e_i}$ and define $\rho_{i} = \bra{e_i}\rho \ket{e_i}$ for any $\rho\in\mbS_d$. 
    For a random column-stochastic matrix \(\Omega\ni\omega\mapsto Q_\omega = (q_{ki}(\omega))\), consider either of the instruments 
    \begin{equation}
    \label{eq:examples-fine}
        \mcT_{(k,i);\omega}(X)= \left(\sqrt{q_{ki}(\omega)}\,\ket{e_k}\bra{e_i}\right) \, X \, \left(\sqrt{q_{ki}(\omega)}\,\ket{e_k}\bra{e_i}\right)\adj
    \end{equation}
    or
    \begin{equation}
    \label{eq:examples-coarse}
        \widehat{\mcT}_{k;\omega}(X)
        =
        \sum_i\mcT_{(k,i);\omega}(X)
        =
        \left(\sum_iq_{ki}(\omega)X_{ii}\right)\rho^{(k)}.
    \end{equation}

    The first instrument is perfect and records the pair $(k,i)$, with outcome space $\mcA=\{1,\ldots,d\}^2$.
    The second records only $k$ and is obtained by summing over the unobserved index $i$. 
    It is not necessarily perfect, and its outcome space is $\widehat{\mcA}=\{1,\ldots,d\}$. Although their visible records differ, both instruments have the same posterior-state kernel:
    \begin{equation}
    \label{eq:examples-classical-kernel}
         K_\omega(\rho,\,\cdot\,)
         =\sum_k\left(Q_\omega p(\rho)\right)_k\delta_{\rho^{(k)}},
         \qquad p(\rho)=(\rho_{11},\ldots,\rho_{dd})^{\mathsf T}.
    \end{equation}

    The following proposition supports the claims made in the examples in this subsection. 

    \begin{prop}
    \label{prop:posterior-finite-state-reduction}
        Suppose \eqref{eq:examples-classical-kernel} holds with the basis states specified above. 
        For an integer $L\ge1$, set
        \[
         M_\omega=Q_{\theta^{L-1}\omega}\cdots Q_\omega,
         \qquad h_\omega(k)=\min_i M_\omega(k,i),
         \qquad b_\omega=\sum_k h_\omega(k).
        \]
        If $b_\omega\ge b_*>0$ uniformly, then
        \[
         K_\omega^{(L)}(\rho,\,\cdot\,)\ge b_*\widetilde\lambda_\omega,
         \qquad
         \widetilde\lambda_\omega
         =\frac1{b_\omega}\sum_k h_\omega(k)\delta_{\rho^{(k)}}.
        \]
        More generally, if $M_\omega(k,i)\ge b\ell_\omega(k)$ for all $k,i$, where $b>0$ and $\ell_\omega$ is a probability vector, then
        \[
            K_\omega^{(L)}(\rho,\,\cdot\,)
            \ge b\sum_k\ell_\omega(k)\delta_{\rho^{(k)}}
            \qquad\text{for every }\rho\in\mbS_d.
        \]
    \end{prop}
    \begin{proof}
        Since $(\rho^{(i)})_{jj} = \bra{e_j}\rho^{(i)}\ket{e_j} = |\bra{e_j}\ket{e_i}|^2 = \delta_{i,j}$, we have that $p(\rho^{(i)})=\mathbf e_i$, the $i$-th standard coordinate vector in $\mbR^d$. 
        Therefore 
        \[
            K_\omega^{(L)}(\rho,\,\cdot\,)
            =
            \sum_k\left(M_\omega p(\rho)\right)_k\delta_{\rho^{(k)}}.
        \]
       By the hypothesis above, each coefficient is at least $h_\omega(k)$ in the first case and at least $b\ell_\omega(k)$ in the second, since the coordinates of $p(\rho)$ sum to one: $\sum_i\rho_{ii}=1$.
    \end{proof}

    Applying the criterion in \Cref{prop:posterior-finite-state-reduction} to specific choices of the transition matrix $Q_\omega$ yields the following families of instruments satisfying \Cref{assumption_doeblin}.

    \paragraph{Special cases.}
        Applying \Cref{prop:posterior-finite-state-reduction} to the transition matrices in \Cref{tab:examples-basis-transitions} gives instruments satisfying \Cref{assumption_doeblin}. 
        Each choice admits both realizations \eqref{eq:examples-fine} and \eqref{eq:examples-coarse}.
        The first four rows in the table below use the family
        \begin{equation}
        \label{eq:examples-reset-mixture}
            Q_\omega=(1-\alpha_\omega)I
            +\alpha_\omega v_\omega\mathbf1^{\mathsf T},
            \qquad
            0<\alpha_*\le\alpha_\omega\le1,
        \end{equation}
        where $v_\omega$ is a probability vector.
        Write $\mathbf e_i$ for the $i$-th standard coordinate vector, and let $C$ be the cyclic permutation matrix defined by
        \[
            C\mathbf e_i=\mathbf e_{i+1}\quad(1\le i<d),
            \qquad
            C\mathbf e_d= \mathbf e_1.
        \]
        All parameter and transition-probability bounds below hold uniformly in $\omega$.
        For each row, the displayed lower bound verifies
        \[
            M_\omega(k,i)\ge b\,\ell_\omega(k)
            \qquad\text{for every }k,i.
        \]
        The proposition therefore gives
        \[
            K_\omega^{(L)}(\rho,\,\cdot\,)
            \ge b\,\widetilde\lambda_\omega,
            \qquad
            \widetilde\lambda_\omega
            =\sum_k\ell_\omega(k)\delta_{\rho^{(k)}},
            \qquad \rho\in\mbS_d.
        \]
        
    \begingroup
        \setlength{\tabcolsep}{4pt}
        \renewcommand{\arraystretch}{1.2}
        
        \begin{xltabular}{\linewidth}{
            @{}>{\raggedright\arraybackslash}X c c c c@{}
        }
        \caption{Doeblin bounds for basis transition instruments.}
        \label{tab:examples-basis-transitions}\\
                \toprule
            Model and assumptions
            & $L$
            & \shortstack{Lower bound\\for $M_\omega(k,i)$}
            & $b$
            & $\ell_\omega(k)$\\
            \midrule
            
            \textit{Reset mixture.}
            \eqref{eq:examples-reset-mixture}, with arbitrary $v_\omega$.
            & $1$
            & $\alpha_*v_\omega(k)$
            & $\alpha_*$
            & $v_\omega(k)$\\
            \addlinespace
            
            \textit{Uniform label noise.}
            \eqref{eq:examples-reset-mixture}, with
            $v_\omega=\mathbf1/d$.
            & $1$
            & $\alpha_*/d$
            & $\alpha_*$
            & $1/d$\\
            \addlinespace
            
            \textit{Absorption.}
            \eqref{eq:examples-reset-mixture}, with
            $v_\omega=\mathbf e_1$.
            & $1$
            & $\alpha_*\delta_{k1}$
            & $\alpha_*$
            & $\delta_{k1}$\\
            \addlinespace
            
            \textit{Exact replacement.}
            \eqref{eq:examples-reset-mixture}, with
            $\alpha_\omega=1$ and $v_\omega=\mathbf e_{r(\omega)}$.
            & $1$
            & $\delta_{k,r(\omega)}$
            & $1$
            & $\delta_{k,r(\omega)}$\\
            \midrule
            
            \textit{Complete transitions.}
            $q_{ki}(\omega)\ge\delta>0$ for every $k,i$.
            & $1$
            & $\delta$
            & $d\delta$
            & $1/d$\\
            \addlinespace
            
            \textit{Cyclic keep--switch.}
            $Q_\omega=a_\omega I+(1-a_\omega)C$,
            $a_\omega\in[\eta,1-\eta]$, $0<\eta\le1/2$.
            & $d-1$
            & $\eta^{d-1}$
            & $d\eta^{d-1}$
            & $1/d$\\
            \addlinespace
            
            \textit{Fair three-state cycle.}
            $d=3$ and $Q_\omega=(I+C)/2$.
            & $2$
            & $1/4$
            & $3/4$
            & $1/3$\\
       
            \bottomrule
        \end{xltabular}
    \endgroup

    The first five entrywise bounds follow directly from $Q_\omega$.
    In the cyclic case, every displacement is achieved by $j$ switches and $d-1-j$ stays for some $0\le j\le d-1$.
    Each such walk has probability at least $\eta^{d-1}$, giving the displayed bound.
    For the fair three-state cycle, $Q^2=(I+2C+C^2)/4$, so every entry is at least $1/4$.
    The cyclic length $d-1$ is minimal: for $1\le n<d-1$, every row of the $n$-step transition matrix contains a zero.
    Testing the basis-state inputs therefore forces the mass of any common minorant at each basis state to vanish.

    The same finite-state construction also gives instruments whose transition probabilities are specified through a directed graph.
    We present a general graph construction and then specialize it to Cayley graphs.

    \paragraph{Instruments from directed graphs.}
        Let $\mcG$ be a fixed directed graph on $\{1,\ldots,d\}$.
        Assume that there exists an integer $r\ge1$ such that, for every ordered pair of vertices $i,k$, including $i=k$, there is a directed walk from $i$ to $k$ consisting of exactly $r$ edges, where vertices and edges may be repeated.
        Choose a measurable family of column-stochastic matrices $Q_\omega=(q_{ki}(\omega))$ such that
        \[
            q_{ki}(\omega)\ge\eta>0
            \qquad\text{whenever }i\longrightarrow k
            \text{ is an edge of }\mcG,
        \]
        uniformly in $\omega$.
        As in \eqref{eq:examples-fine} and
        \eqref{eq:examples-coarse}, define
        \begin{align*}
            \mcT_{(k,i);\omega}(X)
            &=q_{ki}(\omega)X_{ii}\rho^{(k)},\\
            \widehat{\mcT}_{k;\omega}(X)
            &=\left(\sum_iq_{ki}(\omega)X_{ii}\right)\rho^{(k)}.
        \end{align*}
        The first instrument records the transition $(i\to k)$; the second records only its destination $k$.
        Both have the posterior-state kernel \eqref{eq:examples-classical-kernel}.
        For any initial vertex $i$ and terminal vertex $k$, choose an $r$-step walk
        \[
            i=i_0\longrightarrow i_1\longrightarrow\cdots
            \longrightarrow i_r=k.
        \]
        Its contribution to the transition product gives
        \[
            M_\omega(k,i)
            \ge
            \prod_{j=0}^{r-1}
            q_{i_{j+1},i_j}(\theta^j\omega)
            \ge\eta^r.
        \]
        Hence \Cref{prop:posterior-finite-state-reduction} yields
        \[
            K_\omega^{(r)}(\rho,\,\cdot\,)
            \ge d\eta^r\,\widetilde\lambda,
            \qquad
            \widetilde\lambda
            =\frac1d\sum_{k=1}^d\delta_{\rho^{(k)}}.
        \]
        Thus both realizations satisfy  \Cref{assumption_doeblin} with block length $L=r$  and minorization mass $b=d\eta^r$.

    \paragraph{Instruments from finite Cayley graphs.}
        Let $G$ be a finite group with $d=|G|\ge2$, and index the orthonormal basis by $g\in G$, writing
        $\rho^{(g)}=\ket{e_g}\bra{e_g}$.
        Let $S\subseteq G$ generate $G$ and contain its identity.
        The moves
        \[
            g\longrightarrow sg,\qquad s\in S,
        \]
        define the directed left Cayley graph, with a loop at each vertex supplied by the identity.
        For each $g\in G$, choose a measurable probability vector $(p_\omega(s\mid g))_{s\in S}$ satisfying
        \[
            p_\omega(s\mid g)\ge\eta>0,
            \qquad
            \sum_{s\in S}p_\omega(s\mid g)=1,
        \]
        uniformly in $\omega$ and $g$.
        Define the transition matrix by
        \[
            q_{h,g}(\omega)
            =\sum_{s\in S}
              p_\omega(s\mid g)\mathbf1_{\{h=sg\}}.
        \]
        The outcome spaces are $\mcA=G\times G$ for the perfect realization and $\widehat{\mcA}=G$ for the coarse realization. Set $\mcT_{(h,g);\omega}=0$ whenever $hg^{-1}\notin S$.
        The remaining perfect branches and the coarse branches are, respectively,
        \[
            \mcT_{(sg,g);\omega}(X)
            =p_\omega(s\mid g)X_{gg}\rho^{(sg)},
            \qquad s\in S,\quad g\in G,
        \]
       and
        \[
            \widehat{\mcT}_{h;\omega}(X)
            =\left(\sum_{g\in G}q_{h,g}(\omega)X_{gg}\right)
              \rho^{(h)}.
        \]
        Since $G$ is finite and $S$ generates $G$, the directed Cayley graph is strongly connected.
        Let $r$ be the directed diameter of this Cayley graph.
        Then every ordered pair of vertices is connected by a directed walk of length at most $r$.
        Since the identity belongs to $S$, shorter walks can be extended to exactly $r$ steps by holding at their terminal vertex.
        The preceding graph argument therefore gives
        \[
            K_\omega^{(r)}(\rho,\,\cdot\,)
            \ge |G|\eta^r\,\widetilde\lambda,
            \qquad
            \widetilde\lambda
            =\frac1{|G|}\sum_{h\in G}\delta_{\rho^{(h)}}.
        \]
        Thus both realizations satisfy \Cref{assumption_doeblin} with block length $L=r$ and minorization mass $b=|G|\eta^r$.

    \medskip
        
    We next show that the preceding constructions admit continuum outcome realizations with the same posterior-state kernels, and hence the same block lengths and minorization masses.

    \paragraph{Continuum outcomes.}
        The same criterion applies without changing the posterior state space.
        Let $(\Xi,\mcX)$ be a standard Borel outcome space, and let $\nu_\omega$ be a reference probability measure in the density representation of the instrument, as in \Cref{prop:instrument-density-representation}.
        Let measurable sets $E_{ki}\subseteq\Omega\times\Xi$ have sections
        \[
            E_{ki,\omega}:=\{x\in\Xi:(\omega,x)\in E_{ki}\}
        \]
        that partition $\Xi$ up to $\nu_\omega$-null sets.
        Suppose that
        \[
            \mcT_{x;\omega}(X)
            =w_{ki}(\omega,x)X_{ii}\rho^{(k)}
            \qquad (x\in E_{ki,\omega}),
        \]
        where the weights $w_{ki}$ are nonnegative and jointly measurable.
        Set
        \[
            q_{ki}(\omega)
            =\int_{E_{ki,\omega}}
              w_{ki}(\omega,x)\,\nu_\omega(\dee x).
        \]
        Trace preservation gives $\sum_kq_{ki}(\omega)=1$.
        Each positive-probability outcome in $E_{ki,\omega}$ has posterior $\rho^{(k)}$, so
        \[
            K_\omega(\rho, \,\cdot\,)
            =\sum_{k,i}\rho_{ii}
              \left(
                  \int_{E_{ki,\omega}}
                  w_{ki}(\omega,x)\,\nu_\omega(\dee x)
              \right)\delta_{\rho^{(k)}}
            =\sum_k(Q_\omega p(\rho))_k\delta_{\rho^{(k)}}.
        \]
        Thus \eqref{eq:examples-classical-kernel} holds, and \Cref{prop:posterior-finite-state-reduction} applies.
        These branches admit the perfect realization
        \[
            V_{x;\omega}
            =\sqrt{w_{ki}(\omega,x)}\,\ket{e_k}\bra{e_i}
            \qquad (x\in E_{ki,\omega}).
        \]
        In particular, every preceding choice of $Q_\omega$ has a continuum realization with the same block length $L$ and minorization mass $b$.
        For this construction, take
        \[
            \Xi=\{1,\ldots,d\}^2\times[0,1],
            \qquad
            \mcX=2^{\{1,\ldots,d\}^2}
                 \otimes\mcB([0,1]),
        \]
        with
        \[
            E_{ki,\omega}=\{(k,i)\}\times[0,1].
        \]
        Choose the reference probability measure
        \[
            \nu_\omega=\nu
            :=\left(\frac1{d^2}\sum_{k,i=1}^d\delta_{(k,i)}\right)
              \otimes\operatorname{Leb}_{[0,1]},
        \]
        independently of $\omega$.
        For $x=(k,i,t)$, set
        \[
            V_{x;\omega}
            =d\sqrt{q_{ki}(\omega)}\,\ket{e_k}\bra{e_i},
            \qquad
            \mcT_{x;\omega}(X)
            =V_{x;\omega}XV_{x;\omega}\adj.
        \]
        Then $w_{ki}(\omega,x)=d^2q_{ki}(\omega)$, and
        \[
            \int_{E_{ki,\omega}}
            w_{ki}(\omega,x)\,\nu_\omega(\dee x)
            =q_{ki}(\omega),
        \]
        as required.


\subsection{Common posterior criterion}
\label{subsec:examples-preparation}


    Several outcome records can lead to the same posterior state.
    Their probabilities can therefore be added before taking a uniform lower bound. 
    For a finite alphabet $\mcA$ and a word $b=(b_1,\ldots,b_L)\in\mcA^L$, write
    \[
         \mcT_{b;\omega}^{(L)}
         :=\mcT_{b_L;\theta^{L-1}\omega}\circ\cdots\circ\mcT_{b_1;\omega}.
    \]

    \begin{prop}
    \label{prop:posterior-block-atom}
        Fix $L\ge1$ and pairwise disjoint sets $U_1,\ldots,U_s\subseteq\mcA^L$.
        Suppose that, for almost every $\omega$, there are states
        $\tau_{j;\omega}\in\states$, positive semidefinite matrices
        $F_{b;\omega}$, and numbers $c_{j;\omega}\ge0$ such that
        \[
            \mcT_{b;\omega}^{(L)}(X)
            =\tr{F_{b;\omega}X}\tau_{j;\omega}
            \qquad (X\in\matrices,\ b\in U_j),
        \]
        and
        \[
            A_{j;\omega}:=\sum_{b\in U_j}F_{b;\omega}
            \ge c_{j;\omega}I.
        \]
        Assume that
        \[
            b_\omega:=\sum_{j=1}^s c_{j;\omega}\ge b_* > 0,
        \]
        where $b_*$ is deterministic, and define
        \[
             \widetilde\lambda_\omega
             :=\frac{1}{b_\omega}
               \sum_{j=1}^s c_{j;\omega}\delta_{\tau_{j;\omega}}
             \in\mcP(\states).
        \]
        Then
        \[
             K_\omega^{(L)}(\rho,B)
             \ge b_* \widetilde\lambda_\omega(B),
             \qquad \rho\in\states,\quad B\in\borel{\states}.
        \]
        In particular, \Cref{assumption_doeblin} holds with
        $\lambda_\omega=\widetilde\lambda_{\theta^{-L}\omega}$.
    \end{prop}

    A related criterion applies to general outcome spaces.
    Let $(\Xi,\mcX)$ be a standard Borel outcome space, and use the density representation from \Cref{prop:instrument-density-representation}.
    For $b=(b_1,\ldots,b_L)\in\Xi^L$, use the same composition with the density maps $\mcT_{b_\ell;\theta^{\ell-1}\omega}$, and put
    \[
        \nu_\omega^{(L)}
        :=\nu_\omega\otimes\nu_{\theta\omega}\otimes\cdots
          \otimes\nu_{\theta^{L-1}\omega}.
    \]
    Let $E\subseteq\Omega\times\Xi^L$ be measurable and write
    \[
        E_\omega:=\{b\in\Xi^L:(\omega,b)\in E\}.
    \]
    Suppose that, for almost every $\omega$ and $\nu_\omega^{(L)}$-almost every $b\in E_\omega$,
    \[
        \mcT_{b;\omega}^{(L)}(X)
        =\tr{F_{b;\omega}X}\tau_{b;\omega}
        \qquad (X\in\matrices),
        \qquad
        F_{b;\omega}\ge c_{b;\omega}I\ge0,
    \]
    where $\tau_{b;\omega}\in\states$ and the families $F_{b;\omega}$, $\tau_{b;\omega}$, and $c_{b;\omega}$ are jointly measurable in $(\omega,b)$.
    Assume that
    \[
        b_\omega
        :=\int_{E_\omega}c_{b;\omega}\,
          \nu_\omega^{(L)}(\dee b)
        \ge b_*>0,
    \]
    where $b_*$ is deterministic, and define
    \[
        \widetilde\lambda_\omega
        :=\frac1{b_\omega}
          \int_{E_\omega}
          c_{b;\omega}\delta_{\tau_{b;\omega}}\,
          \nu_\omega^{(L)}(\dee b).
    \]
    Then
    \[
        K_\omega^{(L)}(\rho,B)
        \ge b_*\widetilde\lambda_\omega(B),
        \qquad
        \rho\in\states,\quad B\in\borel{\states}.
    \]
    In particular, \Cref{assumption_doeblin} holds with $\lambda_\omega=\widetilde\lambda_{\theta^{-L}\omega}$.

    We prove the finite-alphabet criterion and the general-outcome variant together.
    
    \begin{proof}
        Fix $\rho\in\states$ and $B\in\borel{\states}$. 
        In the finite-alphabet case, every word $b\in U_j$ with positive probability ends at $\tau_{j;\omega}$.
        The total probability of the words in $U_j$ is
        \[
            \sum_{b\in U_j}\tr{F_{b;\omega}\rho}
            =\tr{A_{j;\omega}\rho}
            \ge c_{j;\omega}.
        \]
        Since the groups are disjoint,
        \[
            K_\omega^{(L)}(\rho,B)
            \ge\sum_{j=1}^s
                c_{j;\omega}\mathbf1_B(\tau_{j;\omega})
            =b_\omega\widetilde\lambda_\omega(B)
            \ge b_*\widetilde\lambda_\omega(B).
        \]
        For general outcomes, the same argument gives
        \[
            \begin{aligned}
            K_\omega^{(L)}(\rho,B)
            &\ge\int_{E_\omega}
                \mathbf1_B(\tau_{b;\omega})
                \tr{F_{b;\omega}\rho}\,
                \nu_\omega^{(L)}(\dee b)\\
            &\ge\int_{E_\omega}
                \mathbf1_B(\tau_{b;\omega})c_{b;\omega}\,
                \nu_\omega^{(L)}(\dee b)\\
            &=b_\omega\widetilde\lambda_\omega(B)
            \ge b_*\widetilde\lambda_\omega(B).
            \end{aligned}
        \]
        In both cases, the normalized measures have mass one, and their measurability follows from the assumed measurability of the data.
    \end{proof}
    
    \paragraph{Special cases.}
        \Cref{tab:posterior-preparation} applies \Cref{prop:posterior-block-atom} with $L=1$, taking each selected outcome as a separate group.
        In the first three rows, select the reset outcomes; in the last row, select all outcomes.
        Thus $A_{a;\omega}=F_{a;\omega}$, and the required bound is $F_{a;\omega}\ge c_{a;\omega}I$.
        Write $\mcR_\tau(X)=\tr{X}\tau$ and suppress environmental dependence within the table.
        Each row uses its own finite outcome alphabet.
        In the first three rows, $\mcB$ denotes a finite nonempty set of additional outcome labels, disjoint from the reset labels.
        All prepared states belong to $\states$; unlisted branches may be any completely positive maps completing the instrument.
        The last column gives the total minorization mass $b_\omega$ and the probability measure $\widetilde\lambda_\omega$.
        Whenever $b_\omega\ge b_*>0$ uniformly, the proposition gives
        \[
            K_\omega(\rho,\,\cdot\,)
            \ge b_*\widetilde\lambda_\omega
            \qquad\text{for every }\rho\in\states,
        \]
        and hence \Cref{assumption_doeblin}.
        
    \begingroup
        \setlength{\tabcolsep}{4pt}
        \renewcommand{\arraystretch}{1.2}
        
        \begin{xltabular}{\linewidth}{
            @{}>{\raggedright\arraybackslash}X
            >{\raggedright\arraybackslash}p{0.19\linewidth}
            >{\raggedright\arraybackslash}p{0.19\linewidth}@{}
        }
        \caption{One-step applications of the common posterior criterion.}
        \label{tab:posterior-preparation}\\
        \toprule
        Special case and branches
        & Effect bound
        & $b_\omega$ and $\widetilde\lambda_\omega$\\
        \midrule
        \endfirsthead
            
            \textit{One or several visible resets.}\par
            $\mcA=J\sqcup\mcB$, where $J$ is a finite nonempty set.\par
            $\mcT_j=c_j\mcR_{\tau_j}$ for $j\in J$,
            $c_j\ge0$, $\sum_{j\in J}c_j\le1$.\par
            A singleton $J$ gives a single recorded reset outcome.
            & $F_j=c_jI$.
            & $\displaystyle b_\omega=\sum_{j\in J}c_j$,
              \par $\displaystyle
              \widetilde\lambda_\omega
              =\frac{1}{b_\omega}\sum_{j\in J}c_j\delta_{\tau_j}$.\\
            \addlinespace
            
            \textit{Reset with unitary motion or damping.}\par
            $\mcA=\{r\}\sqcup\mcB$.\par
            $\mcT_r=c\mcR_\tau$,
            $\mcT_a=(1-c)\mcJ_a$ for $a\in \mcB$, $0<c\le1$,
            where $(\mcJ_a)_{a\in \mcB}$ is an instrument.
            Take $\mcJ_a(X)=q_aU_aXU_a^*$ for a probability
            vector $(q_a)_a$ and unitaries $U_a$, or use a resolved
            amplitude-damping instrument.
            & $F_r=cI$.
            & $b_\omega=c$,
              \par $\widetilde\lambda_\omega=\delta_\tau$.\\
            \addlinespace
            
            \textit{Heralded Gibbs preparation.}\par
            $\mcA=\{r\}\sqcup\mcB$.\par
            $\mcT_r=c\mcR_\tau$, $0<c\le1$,
            $\tau=e^{-\beta H}/\tr{e^{-\beta H}}$,
            with $H=H^*$ and $\beta\ge0$.
            The reset is recorded.
            & $F_r=cI$.
            & $b_\omega=c$,
              \par $\widetilde\lambda_\omega=\delta_\tau$.\\
            \addlinespace
            
            \textit{Noisy POVM followed by preparation.}\par
            Take $\mcA=\{1,\ldots,N\}$, with $N\ge1$.\par
            $\mcT_a(X)=\tr{F_aX}\tau_a$,
            $F_a=\beta q_aI+(1-\beta)G_a$,
            where $(G_a)_a$ is a POVM, $(q_a)_a$ is a probability
            vector, and $0<\beta<1$ is fixed.
            & $F_a\ge\beta q_aI$;
              take $c_a=\beta q_a$.
            & $b_\omega=\beta$,
              \par $\displaystyle
              \widetilde\lambda_\omega=\sum_aq_a\delta_{\tau_a}$.\\
            
            \bottomrule
        \end{xltabular}
    \endgroup

    We next apply \Cref{prop:posterior-block-atom} to finite preparation instruments, grouping outcome records with a common final posterior.

    \paragraph{Finite preparation.}
        Take the outcome alphabet $\mcA=\{1,\ldots,N\}$, where $N\ge1$.
        Let $\mcT_{a;\omega}(X)=\tr{F_{a;\omega}X}\tau_{a;\omega}$, $1\le a\le N$, for any POVM $(F_{a;\omega})_a$ and states $\tau_{a;\omega}$. 
        Write $\Phi_\omega=\sum_a\mcT_{a;\omega}$. 
        The prepared states $\tau_{a;\omega}$ may have overlapping supports.
         For each final outcome $b\in\{1,\ldots,N\}$, define
        \[
            U_b
            :=\{u=(u_1,\ldots,u_L)\in\{1,\ldots,N\}^L:u_L=b\}.
        \]
        Every positive-probability word in $U_b$ ends at $\tau_{b;\theta^{L-1}\omega}$, independently of the initial state and the preceding outcomes.
        We therefore add the probabilities of all words in $U_b$ to obtain their contribution to the posterior-state kernel.
        Thus the total probability of the words in $U_b$ is
        \[
            \tr{
                F_{b;\theta^{L-1}\omega}
                \left(
                    \Phi_{\theta^{L-2}\omega}\circ\cdots\circ\Phi_\omega
                \right)(\rho)
            }
            =\tr{A_{b;\omega}^{(L)}\rho},
        \]
        where
        \[
            A_{b;\omega}^{(L)}
            :=\left(
                \Phi_{\theta^{L-2}\omega}\circ\cdots\circ\Phi_\omega
            \right)^*
            (F_{b;\theta^{L-1}\omega}).
        \]
        The channel product is taken as the identity for $L=1$.
        Consequently,
        \[
            K_\omega^{(L)}(\rho,\,\cdot\,)
            =\sum_{b=1}^N
              \tr{A_{b;\omega}^{(L)}\rho}\,
              \delta_{\tau_{b;\theta^{L-1}\omega}}.
        \]
        In \Cref{prop:posterior-block-atom}, we may therefore take
        \[
            c_{b;\omega}
            =\lambda_{\min}(A_{b;\omega}^{(L)}).
        \]

        This class includes noisy detector readout followed by feedback preparation, with $F_{a;\omega}=\sum_iR_\omega(a|i)\ket{e_i}\bra{e_i}$, where $R_\omega(a|i)\ge0$ and $\sum_aR_\omega(a|i)=1$.
        To apply \Cref{prop:posterior-block-atom}, it remains to establish a uniform positive lower bound on
        \[
            b_\omega
            :=\sum_{b=1}^N\lambda_{\min}(A_{b;\omega}^{(L)}).
        \]
        The following contraction condition guarantees such a bound for a sufficiently large block length $L$.
        
        If the channels satisfy
        $\|\Phi_\omega(\rho)-\Phi_\omega(\sigma)\|_1
        \le q\|\rho-\sigma\|_1$ uniformly for some $0\le q<1$, then
        \[
             b_\omega
             =
             \sum_b\lambda_{\min}(A_{b;\omega}^{(L)})
             \ge 1-Nq^{L-1}.
        \]
        Indeed, compare each final-label probability with its value at a fixed reference input; its variation is at most $q^{L-1}$, since $|\tr{F_b(\rho-\sigma)}|\le\frac12\|\rho-\sigma\|_1$
        for $0\le F_b\le I$.
        Choose $L\ge2$ such that $Nq^{L-1}<1$ and set $b_*:=1-Nq^{L-1}>0$.
        Then \Cref{prop:posterior-block-atom} applies with $c_{b;\omega}=\lambda_{\min}(A_{b;\omega}^{(L)})$, and hence \Cref{assumption_doeblin} holds.
        This implication uses the finite set of prepared posteriors.
        For example, a common preparation component $\tau_{a;\omega}=h_\omega\zeta_\omega+(1-h_\omega)\sigma_{a;\omega}$, with $0<h_*\le h_\omega\le1$ and states $\zeta_\omega,\sigma_{a;\omega}$, gives $q=1-h_*$ even for singular effects.

    \medskip
    
    The following example applies the finite preparation construction to a tetrahedral qubit measurement, giving an explicit two-step minorization bound for arbitrary prepared states.
    
    \paragraph{Tetrahedral measurement and feedback.}
    \label{eg:posterior-tetrahedral}
        We now work in dimension $d=2$, with outcome alphabet $\mcA=\{1,2,3,4\}$.
        Let $n_1,\ldots,n_4$ be the unit vertices of a regular tetrahedron, let $O_\omega$ be measurable rotations, and let $\boldsymbol\sigma = (\sigma_1,\sigma_2,\sigma_3)$ denote the Pauli vector.
        Choose arbitrary measurable maps $\omega\mapsto s_{a;\omega}\in\mbR^3$, $a=1,\ldots,4$, with $\norm{s_{a;\omega}}_{\mbR^3}\le1$.
        Let $m_{a;\omega}=O_\omega n_a$ be the rotated tetrahedral directions. 
        The vectors $s_{a;\omega}$ specify the states prepared after each outcome and need not be aligned with the measurement directions $m_{a;\omega}$.
        Set
        \[
             F_{a;\omega}=\tfrac14(I+m_{a;\omega}\cdot\boldsymbol\sigma),
             \qquad
             \tau_{a;\omega}=\tfrac12(I+s_{a;\omega}\cdot\boldsymbol\sigma),
        \]
        where
        \[
            m\cdot\boldsymbol\sigma
            :=\sum_{j=1}^3m_j\sigma_j,
            \qquad m\in\mbR^3.
        \]
        Define the preparation instrument by
        \[
            \mcT_{a;\omega}(X)
            =\tr{F_{a;\omega}X}\tau_{a;\omega},
            \qquad a=1,\ldots,4.
        \]
        For every such choice of the rotations $O_\omega$ and prepared states $\tau_{a;\omega}$, this instrument satisfies \Cref{assumption_doeblin} with block length $L=2$ and minorization mass $b_*=2/3$.
        
        To see this, write
        \[
             v_\omega=\tfrac14\sum_a s_{a;\omega},\qquad
             M_\omega=\tfrac14\sum_a s_{a;\omega}m_{a;\omega}^{\mathsf T}.
        \]
        The total channel acts on Bloch vectors by $r\mapsto v_\omega+M_\omega r$.
        Consequently, the effects obtained by grouping two-step words by their last outcome are
        \[
             A_{b;\omega}^{(2)}
             =\tfrac14\left[
               (1+m_{b;\theta\omega}\cdot v_\omega)I
               +(M_\omega^{\mathsf T}m_{b;\theta\omega})
                                 \cdot\boldsymbol\sigma\right].
        \]
        The tetrahedral inner products give the Frobenius-norm estimate
        \[
             \|M_\omega\|_{\mathrm F}^2
             =\tfrac1{12}\sum_a\norm{s_{a;\omega}}_{\mbR^3}^2
                   -\tfrac1{48}\norm{\sum_a s_{a;\omega}}_{\mbR^3}^2
             \le\tfrac13.
        \]
        Using the tetrahedral identities at $\theta\omega$ and the Cauchy--Schwarz inequality, we obtain
        \begin{align*}
             \sum_b\lambda_{\min}(A_{b;\omega}^{(2)})
             =1-\tfrac14\sum_b\norm{M_\omega^{\mathsf T}m_{b;\theta\omega}}_{\mbR^3}
             \ge1-\sqrt{\tfrac13\|M_\omega\|_{\mathrm F}^2}
             \ge\tfrac23.
        \end{align*}
        We now apply \Cref{prop:posterior-block-atom} with $L=2$, the disjoint groups
        \[
            U_b=\{(a,b):a=1,\ldots,4\},
            \qquad b=1,\ldots,4,
        \]
        the common posteriors $\tau_{b;\theta\omega}$, and
        \[
            c_{b;\omega}
            =\lambda_{\min}(A_{b;\omega}^{(2)}).
        \]
        Since
        $A_{b;\omega}^{(2)}\ge c_{b;\omega}I$
        and $\sum_b c_{b;\omega}\ge2/3$, the proposition gives
        \[
            K_\omega^{(2)}(\rho,\,\cdot\,)
            \ge \tfrac23\,\widetilde\lambda_\omega,
            \qquad
            \widetilde\lambda_\omega
            =\frac{\sum_b c_{b;\omega}\delta_{\tau_{b;\theta\omega}}}
                   {\sum_b c_{b;\omega}}.
        \]
        If the four prepared states are distinct almost surely, no one-step minorization is possible: each atom has  the singular effect $F_{a;\omega}$, whose probability has infimum zero.
        In this case the shortest length is exactly two. 
        In particular, choosing $s_{a;\omega}=m_{a;\omega}$ gives the L\"uders instrument of the tetrahedral POVM:
        \[
            \mcT_{a;\omega}(X)
            =\sqrt{F_{a;\omega}}X\sqrt{F_{a;\omega}}.
        \]
        Indeed, $\tau_{a;\omega}$ is then a rank-one projection and $F_{a;\omega}=\tau_{a;\omega}/2$.
        This instrument is perfect, and its four prepared states are distinct, so its minimal block length is two.

        \begin{remark}
            Individual perfect branches with input-independent posteriors have rank-one effects, which cannot dominate $cI$ with $c>0$ when $d\ge2$.
            Grouping records with the same posterior can nevertheless produce a positive definite effect, as this example illustrates.
        \end{remark}    


\subsection{Resolved-jump instruments}
\label{subsec:examples-resolved-jump-network}


    Let $d\ge2$ and fix an orthonormal basis $(e_i)_{i=1}^d$ of $\mbC^d$.
    As before, write \(\rho^{(i)}=\ket{e_i}\bra{e_i}\) and \(X_{ii}=\inner*{e_i}{Xe_i}\).
    Let
    \[
        \mcA_{\mathrm J}
        :=\{(k,i):1\le k,i\le d,\ k\ne i\},
    \]
    and let $\mcA_0$ be a finite set disjoint from $\mcA_{\mathrm J}$.
    The outcome alphabet is $\mcA=\mcA_0\cup\mcA_{\mathrm J}$.

    For $\alpha\in\mcA_0$ and $(k,i)\in\mcA_{\mathrm J}$, define the Kraus operators
    \[
        D_{\alpha;\omega}
        :=\operatorname{diag}
           (d_{\alpha1;\omega},\ldots,d_{\alpha d;\omega}),
        \qquad
        J_{ki;\omega}
        :=\sqrt{w_{ki;\omega}}\,\ket{e_k}\bra{e_i}.
    \]
    Assume that the coefficients $d_{\alpha i;\omega}\in\mbC$ and $w_{ki;\omega}\ge0$ are measurable in $\omega$ and that, almost surely,
    \[
        \sum_{\alpha\in\mcA_0}|d_{\alpha i;\omega}|^2
        +\sum_{k\ne i}w_{ki;\omega}=1,
        \qquad i=1,\ldots,d.
    \]
    We shall call $(k,i)$ a jump outcome: the operator $J_{ki;\omega}$ maps $e_i$ to a multiple of $e_k$ and annihilates the other basis vectors.
    Each jump is individually resolved, meaning that its outcome records both the source $i$ and the destination $k$.
    The diagonal outcomes $\alpha\in\mcA_0$ are called no-jump outcomes, since
    \[
        D_{\alpha;\omega}e_i=d_{\alpha i;\omega}e_i
    \]
    and thus they do not transfer a basis vector to a different basis direction.

    Define the instrument branches by
    \[
        \mcT_{a;\omega}(X)
        :=
        \begin{cases}
            D_{\alpha;\omega}XD_{\alpha;\omega}\adj,
            & a=\alpha\in\mcA_0,\\[2pt]
            J_{ki;\omega}XJ_{ki;\omega}\adj
            =w_{ki;\omega}X_{ii}\rho^{(k)},
            & a=(k,i)\in\mcA_{\mathrm J}.
        \end{cases}
    \]
    The normalization gives
    \[
        \sum_{\alpha\in\mcA_0}
            D_{\alpha;\omega}\adj D_{\alpha;\omega}
        +\sum_{(k,i)\in\mcA_{\mathrm J}}
            J_{ki;\omega}\adj J_{ki;\omega}
        =I.
    \]
    Consequently, these branches form a perfect instrument, whose total channel is
    \[
        \Phi_\omega:=\sum_{a\in\mcA}\mcT_{a;\omega}.
    \]

    Set $w_{ii;\omega}=0$ and define
    \[
        W_\omega:=(w_{ki;\omega})_{k,i=1}^d,
        \qquad
        N_\omega
        :=\operatorname{diag}\left(
            \sum_{\alpha\in\mcA_0}|d_{\alpha1;\omega}|^2,
            \ldots,
            \sum_{\alpha\in\mcA_0}|d_{\alpha d;\omega}|^2
        \right).
    \]
    Then $Q_\omega:=N_\omega+W_\omega$ is column-stochastic and describes the population evolution:
    \[
        p(\Phi_\omega(\rho))=Q_\omega p(\rho),
        \qquad
        p(\rho):=(\rho_{11},\ldots,\rho_{dd})^{\mathsf T},
        \qquad \rho\in\states.
    \]
    Every positive-probability jump into $k$ has posterior  $\rho^{(k)}$, independently of the initial state.
    A no-jump outcome, however, may change the conditional populations, retain coherences, and have a posterior that depends on the initial state.

    \paragraph{Terminal-jump criterion.}
        We apply \Cref{prop:posterior-block-atom} by grouping words whose last outcome is a jump according to that jump’s destination. 
        Fix $L\ge1$ and define the pairwise disjoint sets
        \[
            U_k
            :=\{u\in\mcA^L:
                 u_L=(k,i)\text{ for some }i\ne k\},
            \qquad k=1,\ldots,d.
        \]
        Every positive-probability word in $U_k$ has posterior $\rho^{(k)}$.
        More precisely, writing \(F_{u;\omega}:=\left(\mcT_{u;\omega}^{(L)}\right)^*(I)\ge0\), we have
        \[
            \mcT_{u;\omega}^{(L)}(X)
            =\tr{F_{u;\omega}X}\rho^{(k)},
            \qquad X\in\matrices,\quad u\in U_k.
        \]
        Set
        \[
            R_\omega^{(L)}
            :=W_{\theta^{L-1}\omega}
              Q_{\theta^{L-2}\omega}\cdots Q_\omega,
        \]
        where the product of the $Q$ matrices is the identity when $L=1$.
        Summing over all outcomes preceding the final jump gives
        \[
            \sum_{u\in U_k}\tr{F_{u;\omega}\rho}
            =
            \sum_{i=1}^d w_{ki;\theta^{L-1}\omega}
                \tr{\rho^{(i)}\Phi_\omega^{(L-1)}(\rho)}
            =
            \sum_{j=1}^d R_\omega^{(L)}(k,j)\rho_{jj},
            \qquad \rho\in\states.
        \]
        Thus the grouped effect is
        \[
            A_{k;\omega}
            :=\sum_{u\in U_k}F_{u;\omega}
            =\sum_{j=1}^d R_\omega^{(L)}(k,j)\rho^{(j)}.
        \]
        Since this matrix is diagonal, we may take
        \[
            c_{k;\omega}
            =h_\omega(k)
            :=\min_{1\le j\le d}R_\omega^{(L)}(k,j).
        \]
        Consequently, if
        \[
            b_\omega:=\sum_{k=1}^d h_\omega(k)\ge b_*>0
        \]
        almost surely for some deterministic $b_*$, then  \Cref{prop:posterior-block-atom} gives
        \[
            K_\omega^{(L)}(\rho,\cdot)
            \ge b_*\widetilde\lambda_\omega,
            \qquad
            \widetilde\lambda_\omega
            :=\frac{1}{b_\omega}
              \sum_{k=1}^d h_\omega(k)\delta_{\rho^{(k)}},
            \qquad \rho\in\states.
        \]
        In particular, \Cref{assumption_doeblin} holds with $\lambda_\omega=\widetilde\lambda_{\theta^{-L}\omega}$ and any $0<\varepsilon<\min\{b_*,1\}$.

    We next apply the terminal-jump criterion to a monitored qubit relaxation model, obtaining a two-step minorization and showing that the block length is minimal.
    
    \paragraph{Monitored generalized amplitude damping.}
        We specialize the preceding construction to $d=2$, with $e_1$ and $e_2$ representing the ground and excited states, respectively.
        Let $\mcA_0=\{\mathrm g,\mathrm e\}$ and take
        \[
            \begin{aligned}
            D_{\mathrm g;\omega}
            &=\sqrt{p_\omega}
              \operatorname{diag}(1,\sqrt{1-\gamma_\omega}),
            &
            J_{12;\omega}
            &=\sqrt{p_\omega\gamma_\omega}
              \ket{e_1}\bra{e_2},\\
            D_{\mathrm e;\omega}
            &=\sqrt{1-p_\omega}
              \operatorname{diag}(\sqrt{1-\gamma_\omega},1),
            &
            J_{21;\omega}
            &=\sqrt{(1-p_\omega)\gamma_\omega}
              \ket{e_2}\bra{e_1}.
            \end{aligned}
        \]
    Its total channel is the generalized amplitude-damping channel.
    A monitoring realization records the initial and final energies of a fresh qubit ancilla during an excitation-exchange interaction.
    Here $p_\omega$ is the ancilla ground-state population; the usual thermal interpretation with nonnegative inverse temperature corresponds to $p_\omega\ge1/2$ (\cite{Khatri_2020}).

    Assume that $p_\omega$ and $\gamma_\omega$ are measurable and satisfy, almost surely,
    \[
        0<p_*\le\tfrac12,
        \qquad p_\omega\in[p_*,1-p_*],
        \qquad
        0<\gamma_*\le\gamma_\omega<1,
    \]
    where $p_*$ and $\gamma_*$ are deterministic.
    The population and jump matrices are, respectively, 
    \[
        Q_\omega
        =
        \begin{pmatrix}
            1-(1-p_\omega)\gamma_\omega
            & p_\omega\gamma_\omega\\
            (1-p_\omega)\gamma_\omega
            & 1-p_\omega\gamma_\omega
        \end{pmatrix},
        \qquad
        W_\omega
        =
        \begin{pmatrix}
            0 & p_\omega\gamma_\omega\\
            (1-p_\omega)\gamma_\omega & 0
        \end{pmatrix}.
    \]
    We apply the terminal-jump criterion with $L=2$. Write $p_j=p_{\theta^j\omega}$ and $\gamma_j=\gamma_{\theta^j\omega}$.
    Since $\gamma_0<1$, the row minima of $R_\omega^{(2)}=W_{\theta\omega}Q_\omega$ are
    \[
        h_\omega(1)
        =p_1\gamma_1(1-p_0)\gamma_0,
        \qquad
        h_\omega(2)
        =(1-p_1)\gamma_1p_0\gamma_0.
    \]
    Consequently,
    \[
        b_\omega
        :=h_\omega(1)+h_\omega(2)
        =\gamma_0\gamma_1
          \left[p_1(1-p_0)+(1-p_1)p_0\right]
        \ge 2p_*(1-p_*)\gamma_*^2
        =:b_*>0.
    \]
    Here the bracket is minimized on $[p_*,1-p_*]^2$ when $p_0=p_1=p_*$ or $p_0=p_1=1-p_*$.
    The criterion above, therefore, gives
    \[
        K_\omega^{(2)}(\rho,\,\cdot\,)
        \ge b_*\widetilde\lambda_\omega,
        \qquad \rho\in\states,
    \]
    with
    \[
        \widetilde\lambda_\omega
        :=
        \frac{
            h_\omega(1)\delta_{\rho^{(1)}}
            +h_\omega(2)\delta_{\rho^{(2)}}
        }{b_\omega}.
    \]
    Thus \Cref{assumption_doeblin} holds. Since each $h_\omega(k)\ge p_*^2\gamma_*^2$, one may alternatively use the fixed uniform minorizer:
    \[
        K_\omega^{(2)}(\rho,\,\cdot\,)
        \ge
        2p_*^2\gamma_*^2
        \frac{\delta_{\rho^{(1)}}+\delta_{\rho^{(2)}}}{2}.
    \]
    The minimal block length is exactly two.
    Indeed, basis inputs force any one-step common lower measure to be supported on $\{\rho^{(1)},\rho^{(2)}\}$.
    For
    \[
        \rho_t=(1-t)\rho^{(1)}+t\rho^{(2)},
        \qquad 0<t<1,
    \]
    both no-jump posteriors are positive definite.
    Hence only the jump outcomes contribute to the two basis atoms, giving
    \[
        K_\omega(\rho_t,\{\rho^{(1)}\})
        =p_\omega\gamma_\omega t,
        \qquad
        K_\omega(\rho_t,\{\rho^{(2)}\})
        =(1-p_\omega)\gamma_\omega(1-t).
    \]
    Taking $t\downarrow0$ and $t\uparrow1$, respectively, shows that a common lower measure must assign zero mass to both atoms. Thus, no one-step minorization is possible.
  

\subsection{Synchronizing channel words}
\label{subsec:examples-synchronizing}


    This construction gives exact block lengths greater than one even for a constant (i.e., a deterministic) instrument. 
    The visible outcomes select channels; a suitable composition of these channels has output independent of the initial state, although the individual channels need not have this property.

    Let $\mcA$ be a finite outcome alphabet.
    For each $a\in\mcA$, let $\Psi_{a;\omega}$ be a measurable family of quantum channels, and let $(p_{a;\omega})_{a\in\mcA}$ be a measurable family of probability vectors.
    Define the instrument branches by
    \[
        \mcT_{a;\omega}(X)
        :=p_{a;\omega}\Psi_{a;\omega}(X),
        \qquad a\in\mcA,\quad X\in\matrices.
    \]
    These maps are completely positive, and their sum preserves trace.
    Since each $\Psi_{a;\omega}$ preserves trace, outcome $a$ has probability $p_{a;\omega}$, independently of the initial state.
    When this probability is positive, the posterior is $\Psi_{a;\omega}(\rho)$.
    Thus the recorded outcome selects a channel to apply to the state.

    \paragraph{Synchronizing words.}
        Fix $L\ge1$.
        For a word $u=(u_1,\ldots,u_L)\in\mcA^L$, put
        \[
            c_{u;\omega}
            :=\prod_{r=1}^L p_{u_r;\theta^{r-1}\omega},
            \qquad
            \Psi_{u;\omega}^{(L)}
            :=\Psi_{u_L;\theta^{L-1}\omega}
              \circ\cdots\circ\Psi_{u_1;\omega}.
        \]
        Then \(\mcT_{u;\omega}^{(L)}(X) =c_{u;\omega}\Psi_{u;\omega}^{(L)}(X)\); in particular, the word $u$ has probability $c_{u;\omega}$, independently of the initial state.
        Suppose there is a set $U\subseteq\mcA^L$ such that, for every $u\in U$ and almost every $\omega$,
        \[
            \Psi_{u;\omega}^{(L)}(X)
            =\tr{X}\tau_{u;\omega},
            \qquad X\in\matrices,
        \]
        where $\tau_{u;\omega}\in\states$ is measurable in $\omega$.
        We call these words synchronizing: their channel compositions send every initial state to the same terminal state $\tau_{u;\omega}$.
        For each $u\in U$, we therefore have
        \[
            \mcT_{u;\omega}^{(L)}(X)
            =\tr{c_{u;\omega}I X}\tau_{u;\omega}.
        \]
        We apply \Cref{prop:posterior-block-atom} with the singleton groups $\{u\}$ and effects $F_{u;\omega}=c_{u;\omega}I$.
        If
        \[
            b_\omega:=\sum_{u\in U}c_{u;\omega}\ge b_*>0
        \]
        almost surely for some deterministic $b_*$, the criterion gives
        \begin{equation}
        \label{eq:examples-synchronizing-minorizer}
            K_\omega^{(L)}(\rho,\,\cdot\,)
            \ge b_*\widetilde\lambda_\omega,
            \qquad
            \widetilde\lambda_\omega
            :=\frac{1}{b_\omega}
              \sum_{u\in U}c_{u;\omega}\delta_{\tau_{u;\omega}},
            \qquad \rho\in\states.
        \end{equation}
        Thus \Cref{assumption_doeblin} holds with block length $L$.
        If all selected words have the same terminal state $\tau_\omega$, then $\widetilde\lambda_\omega=\delta_{\tau_\omega}$.
        This criterion gives a sufficient block length.  In the special cases below, we also prove that every shorter length fails, thereby identifying the minimal block length.    

    \paragraph{Special cases.}
        We apply the synchronizing-word criterion to the three constructions in \Cref{tab:examples-synchronizing}.
        In each row, the selected words have a common terminal state $\tau$, so the minorizing probability measure is $\widetilde\lambda_\omega=\delta_\tau$.
        The channels are fixed, while their selection probabilities may depend measurably on $\omega$.
        Taking these probabilities constant gives constant instruments.
        
        For the countdown construction, we use channels induced by maps on a finite set.
        Fix an orthonormal basis $(e_i)_{i=1}^d$, write
        $\rho^{(i)}=\ket{e_i}\bra{e_i}$, and let
        \[
            f_a:\{1,\ldots,d\}\longrightarrow\{1,\ldots,d\},
            \qquad a\in\mcA.
        \]
        Define
        \[
            \Psi_a(X)
            :=\sum_{i=1}^d X_{ii}\rho^{(f_a(i))}.
        \]
        These are channels with Kraus operators $\ket{e_{f_a(i)}}\bra{e_i}$.
        The instrument records the channel label $a$; the index $i$ is unobserved.
        If $f_a$ is the identity map, then $\Psi_a$ is the dephasing channel
        \[
            \Delta_e(X)=\sum_{i=1}^d X_{ii}\rho^{(i)}.
        \]
        For a word $u=(u_1,\ldots,u_L)$, if the composition $f_{u_L}\circ\cdots\circ f_{u_1}$ is constant with value $k_u$, then
        \[
            \Psi_{u_L}\circ\cdots\circ\Psi_{u_1}(X)
            =\tr{X}\rho^{(k_u)}.
        \]
        Thus $u$ is a synchronizing word with terminal state $\rho^{(k_u)}$.
        All probability bounds hold uniformly in $\omega$.
        In each row, $p_*$ is a deterministic positive constant, and the displayed mass gives
        \[
            K_\omega^{(L)}(\rho,\,\cdot\,)
            \ge b_*\delta_\tau,
            \qquad \rho\in\states.
        \]
        The stated minimal block lengths are justified below.

        \begingroup
        \setlength{\tabcolsep}{4pt}
        \renewcommand{\arraystretch}{1.2}

            \begin{xltabular}{\linewidth}{
                @{}>{\raggedright\arraybackslash}X
                >{\raggedright\arraybackslash}p{.16\linewidth}
                >{\centering\arraybackslash}p{.18\linewidth}
                >{\centering\arraybackslash}p{.16\linewidth}@{}
            }
            \caption{Synchronizing channel constructions and their
            minimal Doeblin block lengths.}
            \label{tab:examples-synchronizing}\\
                \toprule
                Construction and assumptions
                & Synchronizing words
                & \shortstack{Minimal $L$\\and target $\tau$}
                & \shortstack{Guaranteed\\mass $b_*$}\\
                \midrule
        
                \textit{Countdown.}\par
                $\mcA=\{a,b\}$, $d=L_0+1$, $L_0\ge2$.
                Index the basis by $0,\ldots,L_0$ and use the
                finite-state channels above with
                \[
                    f_a(i)=\max\{i-1,0\},
                    \qquad f_b(i)=i.
                \]
                Take $p_{a;\omega}\ge p_*>0$ and
                $p_{b;\omega}=1-p_{a;\omega}$.
                & $a^{L_0}$
                & $L=L_0$,
                  \par $\tau=\rho^{(0)}$
                & $p_*^{L_0}$\\
                \addlinespace
        
                \textit{Complementary dephasing.}\par
                $\mcA=\{a,b\}$, $d\ge2$.
                Take $\Psi_a=\Delta_e$ and $\Psi_b=\Delta_g$,
                where the two orthonormal bases satisfy
                $\left|\inner*{e_i}{g_j}\right|^2=1/d$.
                Let $p_{a;\omega}=p_\omega$ and
                $p_{b;\omega}=1-p_\omega$, with
                $p_\omega\in[p_*,1-p_*]$ and
                $0<p_*\le1/2$.
                & $ab$ and $ba$
                & $L=2$,
                  \par $\tau=I/d$
                & $2p_*(1-p_*)$\\
                \addlinespace
        
                \textit{Local subsystem reinitialization.}\par
                $\mcA=\{1,\ldots,m\}$, $m\ge2$, and
                $\mcH=\bigotimes_{j=1}^m\mbC^{d_j}$,
                with $d_j\ge2$.
                The channel $\Psi_j$ traces out subsystem $j$
                and inserts a fixed full-rank state $\tau_j$
                in its original tensor position.
                Take $\sum_jp_{j;\omega}=1$ and
                $p_{j;\omega}\ge p_*>0$.
                & All $m!$ permutations of $1,\ldots,m$
                & $L=m$,
                  \par $\displaystyle\tau=\bigotimes_{j=1}^m\tau_j$
                & $m!p_*^m$\\
        
                \bottomrule
            \end{xltabular}
        \endgroup
    
    For complementary dephasing, mutual unbiasedness gives
    \[
        \Delta_g\Delta_e(X)=\Delta_e\Delta_g(X)
        =\tr{X}\frac{I}{d}.
    \]
    The synchronizing words $ab$ and $ba$ have total probability
    \[
        b_\omega
        =p_\omega(1-p_{\theta\omega})
         +(1-p_\omega)p_{\theta\omega}
        \ge2p_*(1-p_*),
    \]
    since the expression $p(1-q)+(1-p)q$ is minimized on $[p_*,1-p_*]^2$ at equal endpoints.
    Thus the criterion gives the displayed two-step bound.
    To prove minimality, compare $I/d$, fixed by both channels, with
    \[
        \rho=\tfrac12\left(
            \ket{e_1}\bra{e_1}
            +\ket{g_1}\bra{g_1}
        \right).
    \]
    Its two possible one-step posteriors are
    \[
        \Delta_e(\rho)
        =\tfrac12\left(\ket{e_1}\bra{e_1}+I/d\right),
        \qquad
        \Delta_g(\rho)
        =\tfrac12\left(\ket{g_1}\bra{g_1}+I/d\right).
    \]
    Neither equals $I/d$.
    The two initial states therefore have disjoint one-step posterior supports, proving that the minimal block length is exactly two.

    \medskip

    In the countdown construction, the word $a^{L_0}$ sends every basis index to zero.
    Consequently,
    \[
        \Psi_a^{L_0}(X)=\tr{X}\rho^{(0)}.
    \]
    This word has probability
    \[
        \prod_{r=0}^{L_0-1}p_{a;\theta^r\omega}
        \ge p_*^{L_0},
    \]
    so the synchronizing-word criterion gives
    \[
        K_\omega^{(L_0)}(\rho,\,\cdot\,)
        \ge p_*^{L_0}\delta_{\rho^{(0)}},
        \qquad \rho\in\states.
    \]
    To prove minimality, fix $1\le n<L_0$ and compare the initial states $\rho^{(0)}$ and $\rho^{(L_0)}$.
    Both channels fix $\rho^{(0)}$.
    Starting from $\rho^{(L_0)}$, each occurrence of $a$ decreases the basis index by one, while $b$ leaves it unchanged.
    After $n$ outcomes, the index is therefore at least $L_0-n>0$.
    Hence
    \[
        K_\omega^{(n)}(\rho^{(0)},\,\cdot\,)
        =\delta_{\rho^{(0)}},
        \qquad
        K_\omega^{(n)}(\rho^{(L_0)},\{\rho^{(0)}\})=0.
    \]
    A measure dominated by both posterior laws must therefore be zero.
    Thus no positive common minorant exists at any length  $n<L_0$, and the minimal block length is exactly $L_0$.
    In particular, for every prescribed integer $L_0\ge2$, taking $d=L_0+1$ and constant probabilities
    $p_a=p_b=1/2$ gives a constant instrument whose minimal Doeblin block length is $L_0$, with guaranteed minorization mass $2^{-L_0}$.

    \medskip

    Reinitialization channels acting on different subsystems commute.
    Once every subsystem has been reinitialized, the output is the fixed product state $\tau=\bigotimes_{j=1}^m\tau_j$, independently of the initial state.
    More precisely, for every permutation $\pi\in\mfS_m$,
    \[
        \Psi_{\pi(m)}\circ\cdots\circ\Psi_{\pi(1)}(X)
        =\tr{X}\tau.
    \]
    All permutation words therefore contribute to the same posterior atom, with total probability
    \[
        b_\omega
        =\sum_{\pi\in\mfS_m}
          \prod_{r=1}^m p_{\pi(r);\theta^{r-1}\omega}
        \ge m!p_*^m.
    \]
    The synchronizing-word criterion gives
    \[
        K_\omega^{(m)}(\rho,\,\cdot\,)
        \ge m!p_*^m\delta_\tau,
        \qquad \rho\in\states.
    \]
    To prove minimality, compare the initial state $\tau$  with a pure product state
    \[
        \rho=\bigotimes_{j=1}^m\ket{v_j}\bra{v_j},
        \qquad \|v_j\|_{\mbC^{d_j}}=1.
    \]
    Every channel fixes $\tau$.
    After any word of length $n<m$, at least one subsystem of the second input has not been selected.
    Its marginal remains pure, whereas the corresponding marginal $\tau_j$ of the target is full rank on a space of dimension $d_j\ge2$.
    Thus no such word sends $\rho$ to $\tau$, and
    \[
        K_\omega^{(n)}(\tau,\,\cdot\,)=\delta_\tau,
        \qquad
        K_\omega^{(n)}(\rho,\{\tau\})=0.
    \]
    These posterior laws have no positive common minorant.
    Therefore, the minimal block length is exactly $m$.


\subsection{Periodic regeneration}
\label{subsec:examples-periodic}


    We now use a periodic environment to make a reset available once per period.
    Fix $m\ge2$ and equip $\Omega=\{0,\ldots,m-1\}$ with the uniform probability measure and the transformation
    \[
        \theta(j)=j+1\pmod m.
    \]
    We call the environmental state $j$ the current phase.
    All instruments below act on $\mbC^d$, with $d\ge2$.

    At each phase $0\le j<m-1$, choose a finite outcome alphabet $\mcU_j$ and an instrument
    $(\mcS_u^{(j)})_{u\in\mcU_j}$.
    At phase $m-1$, choose a density matrix $\tau$, a channel $\Psi$, and a constant $0<\alpha<1$.
    Define the two active branches by
    \[
        \mcT_{r;m-1}(X)=\alpha\tr{X}\tau,
        \qquad
        \mcT_{c;m-1}(X)=(1-\alpha)\Psi(X).
    \]
    Thus $r$ is a recorded reset outcome with probability  $\alpha$, while $c$ applies the channel $\Psi$ with probability $1-\alpha$.
    Take the outcome space to be 
    \[
        \mcA
        :=\left(
            \bigsqcup_{j=0}^{m-2}\{j\}\times\mcU_j
          \right)\sqcup\{r,c\}.
    \]
    At phase $j<m-1$, put
    \[
        \mcT_{(j,u);j}:=\mcS_u^{(j)},
        \qquad u\in\mcU_j.
    \]
    Set every inactive branch at each phase to zero.
    These maps form a normalized instrument at every phase.

    A block of $m$ observations starting at phase $j$  encounters phase $m-1$ exactly once, at observation
    $m-j$.
    Whatever state has been reached before that observation,  the reset outcome $r$ has probability $\alpha$ and prepares $\tau$.
    After this reset, the remaining $j$ observations begin at phase zero and have posterior law $K_0^{(j)}(\tau,\,\cdot\,)$.
    Summing over all records before and after the reset gives
    \begin{equation}
    \label{eq:examples-periodic-minorizer}
        K_j^{(m)}(\rho,\,\cdot\,)
        \ge\alpha\widetilde\lambda_j,
        \qquad
        \widetilde\lambda_j
        :=K_0^{(j)}(\tau,\,\cdot\,),
        \qquad 0\le j<m,
    \end{equation}
    where $K_0^{(0)}(\tau,\,\cdot\,):=\delta_\tau$.
    In particular, the reset is at the end of the block only when the initial phase is zero.
    Thus \Cref{assumption_doeblin} holds with block length $L=m$ and minorization mass $\alpha$.
    The instruments at the other phases may have state-dependent outcome probabilities.

    \paragraph{Special case: minimal block length $m$.}
            Fix an orthonormal basis $(e_i)_{i=1}^d$, put  $\rho^{(i)}:=\ket{e_i}\bra{e_i}$, and define the dephasing channel
        \[
            \Delta(X):=\sum_{i=1}^d\rho^{(i)}X\rho^{(i)}.
        \]
        At each phase $0\le j<m-1$, take a single outcome:
        \[
            \mcU_j=\{u\},
            \qquad
            \mcS_u^{(j)}=\Delta.
        \]
        At the reset phase, take $\Psi=\Delta$, keeping the prescribed reset state $\tau$ and probability $\alpha$.
        Since $\Delta^2=\Delta$, the minorizing measures in \eqref{eq:examples-periodic-minorizer} are
        \[
            \widetilde\lambda_0=\delta_\tau,
            \qquad
            \widetilde\lambda_j=\delta_{\Delta(\tau)},
            \qquad 1\le j<m.
        \]
        Thus $L=m$ is sufficient, with minorization mass $\alpha$.

        To prove minimality, start at phase zero and compare the initial states $\rho^{(1)}$ and $\rho^{(2)}$, which are distinct because $d\ge2$.
        For every $1\le n<m$, the first $n$ observations occur before the reset phase.
        Each applies $\Delta$, which fixes both basis states.
        Consequently,
        \[
            K_0^{(n)}(\rho^{(1)},\,\cdot\,)
            =\delta_{\rho^{(1)}},
            \qquad
            K_0^{(n)}(\rho^{(2)},\,\cdot\,)
            =\delta_{\rho^{(2)}}.
        \]
        These two laws have disjoint supports and therefore admit no positive common minorant.
        Since phase zero has environmental probability $1/m>0$, \Cref{assumption_doeblin} cannot hold at any length $n<m$.
        The minimal block length is therefore exactly $m$. 
        This realizes every prescribed $m\ge2$ for any fixed system dimension $d\ge2$.
        This example is also covered by the synchronizing-channel construction in \Cref{subsec:examples-synchronizing}.

    \medskip 
    
    In the following example, the outcome probabilities at the phases before the reset depend on the input state.

    \paragraph{Special case: a minimal block length $m$.}
       We work in dimension $d=2$. At every phase $0\le j<m-1$, take $\mcU_j=\{+,-\}$ and define
        \[
            D_+
            :=\frac12
              \begin{pmatrix}
                  \sqrt3&0\\
                  0&1
              \end{pmatrix},
            \qquad
            D_-
            :=\frac12
              \begin{pmatrix}
                  1&0\\
                  0&\sqrt3
              \end{pmatrix},
        \]
        with instrument branches
        \[
            \mcS_\pm^{(j)}(X):=D_\pm XD_\pm^*.
        \]
        Their outcome probabilities depend on the input:
        \[
            \tr{\mcS_+^{(j)}(\rho)}
            =\tfrac34\rho_{11}+\tfrac14\rho_{22},
            \qquad
            \tr{\mcS_-^{(j)}(\rho)}
            =\tfrac14\rho_{11}+\tfrac34\rho_{22}.
        \]
        Consequently, these branches cannot be written as scalar multiples of trace-preserving channels.
        This instrument therefore cannot be written in the form considered in \Cref{subsec:examples-synchronizing}.
        At the reset phase, take $\tau=I/2$ and  $\Psi=\operatorname{id}$, so that
        \[
            \mcT_{r;m-1}(X)=\alpha\tr{X}\frac I2,
            \qquad
            \mcT_{c;m-1}(X)=(1-\alpha)X.
        \]
        The general periodic bound gives
        \[
            K_j^{(m)}(\rho,\,\cdot\,)
            \ge\alpha\widetilde\lambda_j,
            \qquad
            \widetilde\lambda_j
            =K_0^{(j)}(I/2,\,\cdot\,),
            \qquad 0\le j<m.
        \]
        Thus $L=m$ is sufficient.

        To prove minimality, start at phase zero and compare the initial states $I/2$ and $\ket{e_1}\bra{e_1}$.
        For every $1\le n<m$, no reset phase has yet been encountered.
        Every possible record therefore applies a product of the matrices $D_+$ and $D_-$.
        Both matrices are invertible, so these products preserve the rank of the input state.
        Consequently, every posterior obtained from $I/2$ has rank two, while every posterior obtained from $\ket{e_1}\bra{e_1}$ has rank one.
        The two posterior laws therefore have disjoint supports and admit no positive common minorant.
        Since phase zero has environmental probability $1/m>0$, \Cref{assumption_doeblin} fails at every length $n<m$. 

   \begin{remark}
        The finite cyclic environment considered here is ergodic but not mixing.
        These examples therefore verify \Cref{assumption_doeblin}, but do not satisfy the environmental mixing condition in \Cref{assumption_mix}.
   \end{remark}


\end{appendix}

\section*{Data availability statement}
No data was used in this research.

\section*{Conflict of interest}
The author declare that they have  no conflict of interest. 
\setlength{\bibitemsep}{2pt}
\printbibliography[heading=bibintoc,title={References}]

@Article{Attal_2014,
  author    = {Attal, Stéphane and Guillotin-Plantard, Nadine and Sabot, Christophe},
  journal   = {Annales Henri Poincaré},
  title     = {Central Limit Theorems for Open Quantum Random Walks and Quantum Measurement Records},
  year      = {2014},
  issn      = {1424-0661},
  month     = mar,
  number    = {1},
  pages     = {15--43},
  volume    = {16},
  doi       = {10.1007/s00023-014-0319-3},
  publisher = {Springer Science and Business Media LLC},
}

@Article{tristant_2025,
  author    = {Amini, Nina H. and Benoist, Tristan and Bompais, Maël and Pellegrini, Clément},
  journal   = {Journal of Mathematical Physics},
  title     = {Asymptotic stability and ergodic properties of quantum trajectories under imperfect measurement},
  year      = {2026},
  issn      = {1089-7658},
  month     = Aug,
  number    = {8},
  volume    = {67},
  doi       = {10.1063/5.0315396},
  publisher = {AIP Publishing},
}

@article{AP06,
    title = {{From Repeated to Continuous Quantum Interactions}},
    year = {2006},
    journal = {Annales Henri Poincar{\'{e}}},
    author = {Attal, Stéphane and Pautrat, Yan},
    number = {1},
    pages = {59--104},
    volume = {7},
    url = {https://doi.org/10.1007/s00023-005-0242-8},
    doi = {10.1007/s00023-005-0242-8},
    issn = {1424-0661}
}

@article{BBB13,
  author  = {Michel Bauer and Tristan Benoist and Denis Bernard},
  title   = {Repeated Quantum Non-Demolition Measurements: Convergence and Continuous Time Limit},
  journal = {Annales Henri Poincar\'e},
  volume  = {14},
  pages   = {639--679},
  year    = {2013},
  doi     = {10.1007/s00023-012-0204-x},
}

@article{BJPP18,
  author  = {Tristan Benoist and Vojkan Jak\v{s}i\'c and Yan Pautrat and Claude-Alain Pillet},
  title   = {On Entropy Production of Repeated Quantum Measurements I. General Theory},
  journal = {Communications in Mathematical Physics},
  volume  = {357},
  pages   = {77--123},
  year    = {2018},
  doi     = {10.1007/s00220-017-2947-1},
}

@article{BFPP19,
  author  = {Tristan Benoist and Martin Fraas and Yan Pautrat and Cl{\'e}ment Pellegrini},
  title   = {Invariant measure for quantum trajectories},
  journal = {Probability Theory and Related Fields},
  volume  = {174},
  pages   = {307--334},
  year    = {2019},
  doi     = {10.1007/s00440-018-0862-9},
}

@Article{Benoist_2021,
  author    = {Benoist, T. and Cuneo, N. and Jakšić, V. and Pillet, C-A.},
  journal   = {Journal of Statistical Physics},
  title     = {On Entropy Production of Repeated Quantum Measurements II. Examples},
  year      = {2021},
  issn      = {1572-9613},
  month     = feb,
  number    = {3},
  volume    = {182},
  doi       = {10.1007/s10955-021-02725-1},
  publisher = {Springer Science and Business Media LLC},
}

@article{BFP23,
  author  = {Tristan Benoist and Jan-Luka Fatras and Cl{\'e}ment Pellegrini},
  title   = {Limit theorems for quantum trajectories},
  journal = {Stochastic Processes and their Applications},
  volume  = {164},
  pages   = {288--310},
  year    = {2023},
  doi     = {10.1016/j.spa.2023.07.014},
}

@book{BG09,
    title = {{Quantum Trajectories and Measurements in Continuous Time: The Diffusive Case}},
    year = {2009},
    booktitle = {Lecture Notes in Physics},
    author = {Barchielli, Alberto and Gregoratti, Matteo},
    publisher = {Springer Berlin Heidelberg},
    url = {http://dx.doi.org/10.1007/978-3-642-01298-3},
    isbn = {9783642012983},
    doi = {10.1007/978-3-642-01298-3},
    issn = {1616-6361}
}

@article{BGM04,
    title = {{Stochastic Schr{\"{o}}dinger equations}},
    year = {2004},
    journal = {Journal of Physics A: Mathematical and General},
    author = {Bouten, Luc and Guta, Madalin and Maassen, Hans},
    number = {9},
    month = {2},
    pages = {3189--3209},
    volume = {37},
    publisher = {IOP Publishing},
    url = {http://dx.doi.org/10.1088/0305-4470/37/9/010},
    doi = {10.1088/0305-4470/37/9/010},
    issn = {1361-6447}
}

@Article{BGP25,
  author    = {Benoist, Tristan and Greggio, Linda and Pellegrini, Clément},
  journal   = {Annales Henri Poincaré},
  title     = {Exponentially Fast Selection of Sectors for Quantum Trajectories Beyond Non-demolition Measurements},
  year      = {2025},
  issn      = {1424-0661},
  month     = Nov,
  doi       = {10.1007/s00023-025-01641-4},
  publisher = {Springer Science and Business Media LLC},
}

@Article{BHP25,
  author    = {Benoist, Tristan and Hautecœur, Arnaud and Pellegrini, Clément},
  journal   = {Journal of Functional Analysis},
  title     = {Quantum trajectories. Spectral gap, quasi-compactness \& limit theorems},
  year      = {2025},
  issn      = {0022-1236},
  month     = sep,
  number    = {5},
  pages     = {110932},
  volume    = {289},
  doi       = {10.1016/j.jfa.2025.110932},
  publisher = {Elsevier BV},
}

@Article{BJM14,
  author    = {Bruneau, Laurent and Joye, Alain and Merkli, Marco},
  journal   = {Journal of Mathematical Physics},
  title     = {Repeated interactions in open quantum systems},
  year      = {2014},
  issn      = {1089-7658},
  month     = jun,
  number    = {7},
  volume    = {55},
  doi       = {10.1063/1.4879240},
  publisher = {AIP Publishing},
}

@article{BPS24,
  author  = {Tristan Benoist and Cl{\'e}ment Pellegrini and Anna Szczepanek},
  title   = {Dark Subspaces and Invariant Measures of Quantum Trajectories},
  journal = {arXiv preprint arXiv:2409.18655},
  year    = {2024},
  doi     = {10.48550/arXiv.2409.18655},
  publisher = {arXiv},
  URL       ={hhttps://arxiv.org/abs/2409.18655}
}

@Article{Beck_1957,
  author    = {Beck, Anatole and Schwartz, J. T.},
  journal   = {Proceedings of the American Mathematical Society},
  title     = {A vector-valued random ergodic theorem},
  year      = {1957},
  issn      = {1088-6826},
  month     = dec,
  number    = {6},
  pages     = {1049--1059},
  volume    = {8},
  doi       = {10.1090/s0002-9939-1957-0098162-6},
  publisher = {American Mathematical Society (AMS)},
}

@article{Bur+13,
    title = {{Ergodic and mixing quantum channels in finite dimensions}},
    year = {2013},
    journal = {New Journal of Physics},
    author = {Burgarth, D and Chiribella, G and Giovannetti, V and Perinotti, P and Yuasa, K},
    number = {7},
    month = {7},
    pages = {073045},
    volume = {15},
    url = {https://iopscience.iop.org/article/10.1088/1367-2630/15/7/073045},
    doi = {10.1088/1367-2630/15/7/073045},
    issn = {1367-2630}
}

@book{Car93,
    title = {{An Open Systems Approach to Quantum Optics: Lectures Presented at the Universit{\'{e}} Libre de Bruxelles October 28 to November 4, 1991}},
    year = {1993},
    booktitle = {Lecture Notes in Physics Monographs},
    author = {Carmichael, Howard},
    publisher = {Springer Berlin Heidelberg},
    url = {http://dx.doi.org/10.1007/978-3-540-47620-7},
    isbn = {9783540476207},
    doi = {10.1007/978-3-540-47620-7},
    issn = {0940-7677}
}

@article{ChiribellaDArianoPerinotti2009,
  author  = {Chiribella, Giulio and D'Ariano, Giacomo Mauro and Perinotti, Paolo},
  title   = {Realization schemes for quantum instruments in finite dimensions},
  journal = {Journal of Mathematical Physics},
  volume  = {50},
  number  = {4},
  pages   = {042101},
  year    = {2009},
  doi     = {10.1063/1.3105923}
}

@Article{Choi_1975,
  author    = {Choi, Man-Duen},
  journal   = {Linear Algebra and its Applications},
  title     = {Completely positive linear maps on complex matrices},
  year      = {1975},
  issn      = {0024-3795},
  month     = jun,
  number    = {3},
  pages     = {285--290},
  volume    = {10},
  doi       = {10.1016/0024-3795(75)90075-0},
  publisher = {Elsevier BV},
}

@Book{Cinlar_2011,
  author    = {Çinlar, Erhan},
  publisher = {Springer New York},
  title     = {Probability and Stochastics},
  year      = {2011},
  isbn      = {9780387878591},
  doi       = {10.1007/978-0-387-87859-1},
  issn      = {0072-5285},
  journal   = {Graduate Texts in Mathematics},
}

@book{Dav76,
    title = {{Quantum theory of open systems}},
    year = {1976},
    author = {Davies, E B},
    edition = {1},
    publisher = {Academic Press},
    isbn = {0122061500}
}

@Article{DH_PTRF,
  author    = {Dolgopyat, Dmitry and Hafouta, Yeor},
  journal   = {Probability Theory and Related Fields},
  title     = {Berry Esseen theorems for sequences of expanding maps},
  year      = {2025},
  issn      = {},
  month     = aug,
  number    = {1},
  pages     = {1075-1119},
  volume    = {193},
  doi       = {https://doi.org/10.1007/s00440-025-01368-7},
  publisher = {Springer},
}

@Article{Davies_Lewis_1970,
  author    = {Davies, E. B. and Lewis, J. T.},
  journal   = {Communications in Mathematical Physics},
  title     = {An operational approach to quantum probability},
  year      = {1970},
  issn      = {1432-0916},
  month     = sep,
  number    = {3},
  pages     = {239--260},
  volume    = {17},
  doi       = {10.1007/bf01647093},
  publisher = {Springer Science and Business Media LLC},
}

@Article{traj,
  author    = {Ekblad, Owen and Moreno-Nadales, Eloy and Pathirana, Lubashan and Schenker, Jeffrey},
  journal   = {Annales Henri Poincaré},
  title     = {Asymptotic Purification of Quantum Trajectories under Disordered Generalized Measurements},
  year      = {2025},
  issn      = {1424-0661},
  month     = nov,
  doi       = {10.1007/s00023-025-01638-z},
  publisher = {Springer Science and Business Media LLC},
}

@Article{qtlln,
  author    = {Ekblad, Owen and Moreno-Nadales, Eloy and Pathirana, Lubashan},
  journal   = {Journal of Physics A: Mathematical and Theoretical},
  title     = {Ergodic theorems for quantum trajectories under disordered generalized measurements},
  year      = {2026},
  issn      = {1751-8121},
  month     = Jun,
  number    = {22},
  pages     = {225305},
  volume    = {59},
  doi       = {10.1088/1751-8121/ae7165},
  publisher = {IOP Publishing},
}

@article{Gis84,
    title = {{Quantum Measurements and Stochastic Processes}},
    year = {1984},
    journal = {Physical Review Letters},
    author = {Gisin, N},
    number = {19},
    month = {5},
    pages = {1657--1660},
    volume = {52},
    publisher = {American Physical Society (APS)},
    url = {http://dx.doi.org/10.1103/physrevlett.52.1657},
    doi = {10.1103/physrevlett.52.1657},
    issn = {0031-9007}
}

@Article{Gue+07,
  author    = {Guerlin, Christine and Bernu, Julien and Deléglise, Samuel and Sayrin, Clément and Gleyzes, Sébastien and Kuhr, Stefan and Brune, Michel and Raimond, Jean-Michel and Haroche, Serge},
  journal   = {Nature},
  title     = {Progressive field-state collapse and quantum non-demolition photon counting},
  year      = {2007},
  issn      = {1476-4687},
  month     = aug,
  number    = {7156},
  pages     = {889--893},
  volume    = {448},
  doi       = {10.1038/nature06057},
  publisher = {Springer Science and Business Media LLC},
}

@Article{YHMD,
  author    = {Hafouta, Yeor},
  journal   = {Annales de l’Institut Henri Poincaré, Probabilités et Statistiques},
  title     = {Nonconventional moderate deviations theorems and exponential concentration inequalities},
  year      = {2020},
  issn      = {0246-0203},
  month     = feb,
  number    = {1},
  volume    = {56},
  doi       = {10.1214/19-aihp967},
  publisher = {Institute of Mathematical Statistics},
}

@Article{YH_JSP,
  author    = {Hafouta, Yeor},
  journal   = {Journal of Statistical Physics},
  title     = {On the Asymptotic Moments and Edgeworth Expansions for Some Processes in Random Dynamical Environment},
  year      = {2020},
  issn      = {1572-9613},
  month     = may,
  number    = {4},
  pages     = {945--971},
  volume    = {179},
  doi       = {10.1007/s10955-020-02568-2},
  publisher = {Springer Science and Business Media LLC},
}

@Article{HafMS1,
  author    = {Hafouta, Yeor},
  journal   = {arXiv preprint arXiv:2510.07757},
  title     = {Statistical properties of Markov shifts (part I)},
  year      = {2025},
  copyright = {arXiv.org perpetual, non-exclusive license},
  doi       = {10.48550/ARXIV.2510.07757},
  publisher = {arXiv},
  URL       = {https://arxiv.org/abs/2510.07757},
}

@Article{HafMS2,
  author    = {Hafouta, Yeor},
  journal   = {arXiv preprint arXiv:2510.24244},
  title     = {Statistical properties of Markov shifts: part II-LLT},
  year      = {2025},
  copyright = {arXiv.org perpetual, non-exclusive license},
  doi       = {10.48550/ARXIV.2510.24244},
  publisher = {arXiv},
  URL       ={https://arxiv.org/abs/2510.24244},
}

@Article{YH26,
  author    = {Hafouta, Yeor},
  journal   = {arXiv preprint arXiv:2601.00467},
  title     = {Effective geometric ergodicty for Markov chains in random environment},
  year      = {2026},
  copyright = {arXiv.org perpetual, non-exclusive license},
  doi       = {10.48550/ARXIV.2601.00467},
  publisher = {arXiv},
  URL       = {https://arxiv.org/abs/2601.00467},
}

@Book{HR06,
  author    = {Haroche, Serge and Raimond, Jean-Michel},
  editor    = {Jean-Michel Raimond},
  publisher = {Oxford University Press},
  title     = {Exploring the quantum},
  year      = {2013},
  address   = {Oxford},
  edition   = {First published in paperback},
  isbn      = {9780198509141},
  month     = Aug,
  note      = {Hier auch später erschienene, unveränderte Nachdrucke},
  series    = {Oxford graduate texts},
  doi       = {10.1093/acprof:oso/9780198509141.001.0001},
  pagetotal = {605},
  ppn_gvk   = {1611279070},
  subtitle  = {Atoms, cavities, and photons},
}

@Book{HK,
  author    = {Hafouta, Yeor and Kifer, Yuri},
  publisher = {WORLD SCIENTIFIC},
  title     = {Nonconventional Limit Theorems and Random Dynamics},
  year      = {2018},
  isbn      = {9789813235014},
  month     = may,
  doi       = {10.1142/10849},
}

@book{Hol01,
    title = {{Statistical Structure of Quantum Theory}},
    year = {2001},
    booktitle = {Lecture Notes in Physics Monographs},
    author = {Holevo, Alexander S},
    publisher = {Springer Berlin Heidelberg},
    url = {http://dx.doi.org/10.1007/3-540-44998-1},
    isbn = {9783540449980},
    doi = {10.1007/3-540-44998-1},
    issn = {0940-7677}
}

@article{Holevo1998RN,
  author  = {Holevo, A. S.},
  title   = {Radon--Nikodym derivatives of quantum instruments},
  journal = {Journal of Mathematical Physics},
  volume  = {39},
  number  = {3},
  pages   = {1373--1387},
  year    = {1998},
  doi     = {10.1063/1.532385}
}

@Article{HW25,
  author    = {Hafouta, Yeor and Williams, Brenden},
  journal   = {arXiv preprint arXiv:2510.15323},
  title     = {Limit theorems for inhomogeneous $\phi$-mixing Markov chains},
  year      = {2025},
  copyright = {arXiv.org perpetual, non-exclusive license},
  doi       = {10.48550/ARXIV.2510.15323},
  publisher = {arXiv},
  URL       = {https://arxiv.org/abs/2510.15323},
}

@Article{Jamio_kowski_1972,
  author    = {Jamiołkowski, A.},
  journal   = {Reports on Mathematical Physics},
  title     = {Linear transformations which preserve trace and positive semidefiniteness of operators},
  year      = {1972},
  issn      = {0034-4877},
  month     = dec,
  number    = {4},
  pages     = {275--278},
  volume    = {3},
  doi       = {10.1016/0034-4877(72)90011-0},
  publisher = {Elsevier BV},
}

@Book{Kallenberg_2017,
  author    = {Kallenberg, Olav},
  publisher = {Springer International Publishing},
  title     = {Random Measures, Theory and Applications},
  year      = {2017},
  isbn      = {9783319415987},
  doi       = {10.1007/978-3-319-41598-7},
  issn      = {2199-3149},
  journal   = {Probability Theory and Stochastic Modelling},
}

@Book{Kallenberg_2021,
  author    = {Kallenberg, Olav},
  publisher = {Springer International Publishing},
  title     = {Foundations of Modern Probability},
  year      = {2021},
  isbn      = {9783030618711},
  doi       = {10.1007/978-3-030-61871-1},
  issn      = {2199-3149},
  journal   = {Probability Theory and Stochastic Modelling},
}

@Article{kummerer2003ergodic,
  author    = {K{\"u}mmerer, B and Maassen, H},
  journal   = {Journal of Physics A: Mathematical and General},
  title     = {An ergodic theorem for quantum counting processes},
  year      = {2003},
  issn      = {0305-4470},
  month     = feb,
  number    = {8},
  pages     = {2155--2161},
  volume    = {36},
  doi       = {10.1088/0305-4470/36/8/312},
  publisher = {IOP Publishing},
}

@article{KM04,
  author  = {Burkhard K{\"u}mmerer and Hans Maassen},
  title   = {A pathwise ergodic theorem for quantum trajectories},
  journal = {Journal of Physics A: Mathematical and General},
  volume  = {37},
  number  = {49},
  pages   = {11889--11896},
  year    = {2004},
  doi     = {10.1088/0305-4470/37/49/008},
}

@Article{Khatri_2020,
  author    = {Khatri, Sumeet and Sharma, Kunal and Wilde, Mark M.},
  journal   = {Physical Review A},
  title     = {Information-theoretic aspects of the generalized amplitude-damping channel},
  year      = {2020},
  issn      = {2469-9934},
  month     = Jul,
  number    = {1},
  pages     = {012401},
  volume    = {102},
  doi       = {10.1103/physreva.102.012401},
  publisher = {American Physical Society (APS)},
}

@Article{KV,
  author    = {Kifer, Yuri and Varadhan, S. R. S.},
  journal   = {The Annals of Probability},
  title     = {Nonconventional limit theorems in discrete and continuous time via martingales},
  year      = {2014},
  issn      = {0091-1798},
  month     = mar,
  number    = {2},
  volume    = {42},
  doi       = {10.1214/12-aop796},
  publisher = {Institute of Mathematical Statistics},
}

@InBook{MK06,
  author    = {Maassen, Hans and Kümmerer, Burkhard},
  pages     = {252--261},
  publisher = {Institute of Mathematical Statistics},
  title     = {Purification of quantum trajectories},
  year      = {2006},
  isbn      = {0940600641},
  booktitle = {Dynamics \& Stochastics},
  doi       = {10.1214/lnms/1196285826},
  issn      = {0749-2170},
}

@Book{neveu1965mathematical,
  author    = {Neveu, Jacques},
  publisher = {Holden-Day, Inc., San Francisco, Calif.-London-Amsterdam},
  title     = {Mathematical foundations of the calculus of probability},
  year      = {1965},
  note      = {Translated by Amiel Feinstein},
  mrnumber  = {198505},
  pages     = {xiii+223},
}

@article{NP09,
  author  = {Ion Nechita and Cl{\'e}ment Pellegrini},
  title   = {Quantum trajectories in random environment: the statistical model for a heat bath},
  journal = {Confluentes Mathematici},
  volume  = {1},
  number  = {2},
  pages   = {249--289},
  year    = {2009},
  doi     = {10.1142/S1793744209000109},
}

@Article{Ozawa_1984,
  author    = {Ozawa, Masanao},
  journal   = {Journal of Mathematical Physics},
  title     = {Quantum measuring processes of continuous observables},
  year      = {1984},
  issn      = {1089-7658},
  month     = jan,
  number    = {1},
  pages     = {79--87},
  volume    = {25},
  doi       = {10.1063/1.526000},
  publisher = {AIP Publishing},
}

@article{Ozawa1985Posteriori,
  author  = {Ozawa, Masanao},
  title   = {Conditional probability and a posteriori states in quantum mechanics},
  journal = {Publications of the Research Institute for Mathematical Sciences},
  volume  = {21},
  number  = {2},
  pages   = {279--295},
  year    = {1985},
  doi     = {10.2977/prims/1195179625}
}

@Article{luba_clt,
  author    = {Pathirana, Lubashan},
  journal   = {arXiv preprint arXiv:2603.28893},
  title     = {Central Limit Theorems for Outcome Records in Disordered Quantum Trajectories},
  year      = {2026},
  copyright = {Creative Commons Attribution 4.0 International},
  doi       = {10.48550/ARXIV.2603.28893},
  publisher = {arXiv},
  URL       = {https://arxiv.org/abs/2603.28893},
}

@Article{Pel08,
  author    = {Pellegrini, Clément},
  journal   = {The Annals of Probability},
  title     = {Existence, uniqueness and approximation of a stochastic Schrödinger equation: The diffusive case},
  year      = {2008},
  issn      = {0091-1798},
  month     = nov,
  number    = {6},
  volume    = {36},
  doi       = {10.1214/08-aop391},
  publisher = {Institute of Mathematical Statistics},
}

@Article{RR04,
  author    = {Roberts, Gareth O. and Rosenthal, Jeffrey S.},
  journal   = {Probability Surveys},
  title     = {General state space Markov chains and MCMC algorithms},
  year      = {2004},
  issn      = {1549-5787},
  month     = Jan,
  volume    = {1},
  doi       = {10.1214/154957804100000024},
  publisher = {Institute of Mathematical Statistics},
}

@Article{Sayrin_2011,
  author    = {Sayrin, Clément and Dotsenko, Igor and Zhou, Xingxing and Peaudecerf, Bruno and Rybarczyk, Théo and Gleyzes, Sébastien and Rouchon, Pierre and Mirrahimi, Mazyar and Amini, Hadis and Brune, Michel and Raimond, Jean-Michel and Haroche, Serge},
  journal   = {Nature},
  title     = {Real-time quantum feedback prepares and stabilizes photon number states},
  year      = {2011},
  issn      = {1476-4687},
  month     = aug,
  number    = {7362},
  pages     = {73--77},
  volume    = {477},
  doi       = {10.1038/nature10376},
  publisher = {Springer Science and Business Media LLC},
}

@Book{erg_walter,
  author    = {Peter Walters},
  publisher = {Springer New York},
  title     = {An Introduction to Ergodic Theory},
  year      = {1982},
  isbn      = {9781461257752},
  doi       = {10.1007/978-1-4612-5775-2},
  issn      = {0072-5285},
  journal   = {Graduate Texts in Mathematics},
}

@Book{Watrous_2018,
  author    = {Watrous, John},
  publisher = {Cambridge University Press},
  title     = {The Theory of Quantum Information},
  year      = {2018},
  isbn      = {9781107180567},
  month     = apr,
  doi       = {10.1017/9781316848142},
}

@book{WM09,
    title = {{Quantum Measurement and Control}},
    year = {2009},
    author = {Wiseman, Howard M and Milburn, Gerard J},
    month = {11},
    publisher = {Cambridge University Press},
    url = {http://dx.doi.org/10.1017/cbo9780511813948},
    isbn = {9781107424159},
    doi = {10.1017/cbo9780511813948}
}

@article{BJM08,
    title = {{Random Repeated Interaction Quantum Systems}},
    year = {2008},
    journal = {Communications in Mathematical Physics},
    author = {Bruneau, Laurent and Joye, Alain and Merkli, Marco},
    number = {2},
    pages = {553--581},
    volume = {284},
    url = {https://doi.org/10.1007/s00220-008-0580-8},
    doi = {10.1007/s00220-008-0580-8},
    issn = {1432-0916}
}

@article{NP12,
    title = {{Random repeated quantum interactions and random invariant states}},
    year = {2012},
    journal = {Probability Theory and Related Fields},
    author = {Nechita, Ion and Pellegrini, Clément},
    number = {1},
    pages = {299--320},
    volume = {152},
    url = {https://doi.org/10.1007/s00440-010-0323-6},
    doi = {10.1007/s00440-010-0323-6},
    issn = {1432-2064}
}

@article{PS23,
    title = {{Law of large numbers and central limit theorem for ergodic quantum processes}},
    year = {2023},
    journal = {Journal of Mathematical Physics},
    author = {Pathirana, Lubashan and Schenker, Jeffrey},
    number = {8},
    month = {8},
    volume = {64},
    publisher = {AIP Publishing},
    url = {http://dx.doi.org/10.1063/5.0153483},
    doi = {10.1063/5.0153483},
    issn = {1089-7658}
}


\end{document}